\documentclass[11pt,a4paper]{article}
\usepackage[utf8]{inputenc}
\usepackage{amsmath}
\usepackage{amsfonts}
\usepackage{amssymb}
\usepackage{amsthm}
\usepackage{enumerate}
\PassOptionsToPackage{normalem}{ulem}
\usepackage{ulem}
\usepackage[unicode=true,colorlinks,citecolor=linkcolor,linkcolor=linkcolor,urlcolor=red]{hyperref}
\usepackage[left=3.5cm,right=3.5cm,top=2.5cm,bottom=2.5cm]{geometry}
\usepackage{array}
\usepackage{cases}

\usepackage{fancyhdr}
\usepackage{color}
\usepackage{dsfont}

\theoremstyle{plain}
\newtheorem{thm}{\protect Theorem}[section]
\newtheorem{prop}[thm]{\protect Proposition}
\newtheorem{defn}[thm]{\protect Definition}
\newtheorem{representation}[thm]{\protect Representation}
\newtheorem{lem}[thm]{\protect Lemma}
\newtheorem{rem}[thm]{\protect Remark}

\newtheorem{cor}[thm]{\protect Corollary}

\newcommand{\Tr}{\operatorname{Tr}}

\definecolor{linkcolor}{rgb}{0,0,0.502}

\begin{document}

\title{\textsc{Fluctuations for mean-field interacting \\
age-dependent Hawkes processes}}
\author{Julien Chevallier\footnote{Corresponding author: {\sf{e-mail: \href{mailto:julien.chevallier@u-cergy.fr}{julien.chevallier@u-cergy.fr}}}} \medskip \\
Université de Cergy-Pontoise, AGM UMR-CNRS 8088, \\
95302 Cergy-Pontoise
}
\date{}

\maketitle

\begin{abstract}
The propagation of chaos and associated law of large numbers for mean-field interacting age-dependent Hawkes processes (when the number of processes $n$ goes to $+\infty$) being granted by the study performed in \cite{chevallier2015mean}, the aim of the present paper is to prove the resulting functional central limit theorem. It involves the study of a measure-valued process describing the fluctuations (at scale $n^{-1/2}$) of the empirical measure of the ages around its limit value. This fluctuation process is proved to converge towards a limit process characterized by a limit system of stochastic differential equations driven by a Gaussian noise instead of Poisson (which occurs for the law of large numbers limit).
\end{abstract}

\noindent \textit{Keywords}: Hawkes process,  central limit theorem, interacting particle systems, stochastic partial differential equation, neural network.

\smallskip
\noindent \textit{Mathematical Subject Classification}: 60G55, 60F05, 60G57, 60H15, 92B20

\bigskip

\section{Introduction}
\label{sec:introduction}

In the recent years, the self-exciting point process known as the Hawkes process \cite{hawkes_1971} has been used in very diverse areas. First introduced to model earthquake replicas \cite{kagan_2010} or \cite{ogata_1998} (ETAS model), it has been used in criminology to model burglary \cite{mohler_2011}, in genomic data analysis to model occurrences of genes \cite{gusto_2005,reynaud_2010}, in social networks analysis to model viewing or popularity \cite{bao_2015,crane2008robust}, as well as in finance \cite{bacry_2012,CLT_Hawkes_lineaire}. We refer to \cite{liniger2009multivariate} or \cite{zhu7531nonlinear} for more extensive reviews on applications of Hawkes processes.

Part of our analysis finds its motivation in the use of Hawkes processes for the modelling in neuroscience. They are used to describe spike trains associated with several neurons (see e.g. \cite{chevallier2015microscopic}). In that case, it is common to consider a multivariate framework : multivariate Hawkes processes consist of multivariate point processes $(N^{1},\dots ,N^{n})$ whose intensities are respectively given  for $i=1,\dots ,n$ by
\begin{equation}\label{eq:def:Hawkes:multivarie}
\lambda^{i}_{t}= \Phi\left( \sum_{j=1}^{n} \int_{0}^{t-} h_{j\to i}(t-z) N^{j}(dz) \right),
\end{equation}
where $\Phi:\mathbb{R}\rightarrow\mathbb{R}_{+}$ is called the intensity function and $h_{j\to i}$ is the \emph{interaction function} describing the influence of each point of $N^{j}$ on the appearance of a new point onto $N^{i}$, via its intensity $\lambda^{i}$. Notice that we implicitly assume here that there is no influence of the possible points of $N^{j}$ that are before time $0$.\\

In the present paper, as in \cite{chevallier2015mean}, we study a generalization of multivariate Hawkes process by adding an age dependence. 
\begin{defn}
For any point process $N$, we call \emph{predictable age process} associated with $N$, the process defined by
\begin{equation}\label{eq:def:age:process:predictable:intro}
S_{t-}:= t-\sup\{T\in N,\, T<t\}=t-T_{N_{t-}}, \quad \text{ for all } t>0,
\end{equation}
and extended by continuity in $t=0$. In particular, its value in $t=0$ is entirely determined by $N\cap\mathbb{R}_-$ and is well-defined as soon as there is a point therein. 
\end{defn}
In comparison with the standard multivariate Hawkes processes \eqref{eq:def:Hawkes:multivarie}, we add an age dependence, as it is done in \cite{chevallier2015mean}, by assuming that the intensity function $\Phi$ in \eqref{eq:def:Hawkes:multivarie} (which is then denoted by $\Psi$ to avoid confusion) may also depend on the predictable age process $(S^{i}_{t-})_{t\geq 0}$ associated with the point process $N^{i}$, like for instance
\begin{equation}\label{eq:intensity:ADHP}
\lambda^{i}_{t}= \Psi\left( S^{i}_{t-}, \frac{1}{n} \sum_{j=1}^{n} \int_{0}^{t-}  h(t-z) N^{j}(dz) \right).
\end{equation}
We refer to \cite{chevallier2015mean} where the neurobiological motivation for such a form of intensity is given. Under suitable assumptions, it is shown in \cite{chevallier2015mean} that a multivariate point process satisfying \eqref{eq:intensity:ADHP} exists and we call it an age dependent Hawkes process (ADHP). Furthermore, ADHPs are well approximated, when the dimension $n$ goes to infinity, by i.i.d. \emph{limit} point processes of the McKean-Vlasov type whose stochastic intensity depends on the time $t$ and on the age \cite[Theorem IV.1.]{chevallier2015mean}. More precisely, the intensity of the limit process associated with the framework \eqref{eq:intensity:ADHP}, denoted by $\overline{N}$, is given by the following implicit formula $\overline{\lambda}_{t}=\Psi( \overline{S}_{t-}, \int_{0}^{t}  h(t-z) \mathbb{E}\left[ \overline{\lambda}_{z} \right] dz )$ where $(\overline{S}_{t-})_{t\geq 0}$ is the predictable age process associated with $\overline{N}$.\\

As usual with McKean-Vlasov dynamics, the asymptotic evolution (when $n$ goes to infinity) of the distribution of the population at hand can be described as the solution of a nonlinear partial differential equation (PDE). In our case, it is shown that, starting from a density, the distribution of the limit predictable age process $(\overline{S}_{t-})_{t\geq 0}$, denoted by $u_{t}$, admits a density for all time $t\geq 0$ which is furthermore the unique solution of the non-linear system
\begin{equation}\label{eq:edp:PPS}
\begin{cases} 
\displaystyle \frac{\partial u\left(t,s\right)}{\partial t}+\frac{\partial u\left(t,s\right)}{\partial s} +\Psi\left(s,  X(t)  \right) u\left(t,s\right)=0, \\
\displaystyle u\left(t,0\right)=\int_{s\in \mathbb{R}} \Psi\left(s,   X(t) \right)u\left(t,s\right)ds,
\end{cases}
\end{equation}
with initial condition that $u(0,\cdot)=u_0$ (the initial density of the age at time $0$), where for all $t\geq 0$, $X(t)=\int_{0}^{t} h(t-z) u(z,0)dz$ \cite[Proposition III.8.]{chevallier2015mean}. Such a form of PDE system is known either as \emph{age-structure system} or \emph{refractory density equation} or even \emph{von Foerster-McKendrick system}. Here, the age is represented by the variable $s$. We refer to \cite{dumont2016theoretical} for a linear version of \eqref{eq:edp:PPS} and its theoretical connection with the integrate and fire model, and to \cite{gerstner2002spiking,pakdaman2010dynamics,pakdaman2013relaxation} for analytical studies of \eqref{eq:edp:PPS}.

The relation between mean-field age dependent Hawkes processes and the PDE system \eqref{eq:edp:PPS} is completed by the following law of large numbers (consequence of the functional law of large numbers \cite[Corollary IV.4.]{chevallier2015mean}):
\begin{equation}\label{eq:LLN:intro}
\overline{\mu}^{n}_{S_{t-}}:=\frac{1}{n}\sum_{i = 1}^n \delta_{S^{n,i}_{t-}} \xrightarrow[n\to +\infty]{} u_{t}.
\end{equation}
Moreover, the rate of this convergence is at least $n^{-1/2}$. In light of this bound obtained on the rate of convergence, the fluctuation process defined, for all $t\geq 0$, by $\eta^{n}_{t}=\sqrt{n}(\overline{\mu}_{S_{t-}}^{n}-u_{t})$ is expected to describe, on the right scale, the second order term appearing in the expansion of the mean-field approximation, the first order term being given by the law of large numbers. 

The study of the random fluctuations allows to go beyond the first order mean field limit and its main drawback: propagation of chaos. It means independence of the neurons' activities which is unrealistic from the biological viewpoint \cite{MTGAUE,faugeras2014asymptotic}. Hence, the derivation of the second order term is of great importance regarding neural networks modelling since it gives an approximation of the fluctuations coming from the finiteness of the number of neurons $n$ (finite size effects) \cite{brunel1999fast,buice2013dynamic,mattia2004finite}. A partial but promising answer to this problematic is given by highlighting a stochastic partial differential equation system which could be interpreted as an intermediate modelling scale between the microscopic scale given by ADHP and the macroscopic one given by \eqref{eq:edp:PPS}.\\

Following the approach developed in \cite{Meleard_97}, we prove in the present article that the fluctuations satisfy a functional central limit theorem (CLT) in a suitable distributional space: the limit of the normalized fluctuations is described by means of a stochastic differential equation in infinite dimension driven by a Gaussian noise in comparison with the Poisson noise appearing in \cite{chevallier2015mean}. To do so, we regard the fluctuation process $\eta^{n}$ as taking values in a Hilbert space, namely the dual of some Sobolev space of test functions. The index of regularity of the dual space, in one-to-one correspondence with the regularity of the test functions in the Sobolev space, is prescribed by the tightness property we are able to provide to the sequence $(\eta^{n})_{n\geq 1}$ and by the form of the generator of the limiting McKean-Vlasov dynamics identified in \cite{chevallier2015mean}. Let us precise that this generator is the one associated with the renewal dynamics of the system \eqref{eq:edp:PPS} as highlighted by Proposition \ref{prop:age:process:limit:equation:PPS} given hereafter. \\

Although the choice of this index of regularity is rather constrained, the choice of the domain supporting the Sobolev space is somewhat larger. Indeed, two options are available, depending on the way we consider the process $\eta^{n}$, either over a finite time horizon, namely $(\eta^{n}_{t})_{0\leq t\leq \theta}$ for some $\theta\geq 0$, or in infinite horizon, namely $(\eta^{n}_{t})_{t\geq 0}$.

In the first case, we may use the fact that there exists a compact $K_{\theta}$ (which is growing with $\theta$) such that $\eta^{n}_{t}$ is supported in $K_{\theta}$ for all $t$ in $[0,\theta]$. Hence, one could regard, for all $\theta\geq 0$, the fluctuation process $(\eta^{n}_{t})_{0\leq t\leq \theta}$ as a process with values in the dual of a standard Sobolev space of functions with support in $K_{\theta}$. The main drawback of such an approach is that the space of trajectories within which the CLT takes place depends on the time horizon $\theta$. To bypass this issue, one may be willing to work directly on the entire positive time line $\mathbb{R}_{+}$, but then, it is not possible anymore to find a compact subset $K$ supporting the measures $\eta^{n}_{t}$, for all $t\geq 0$, since $\cup_{\theta\geq 0} K_{\theta}=\mathbb{R}_{+}$.
A convenient strategy to sidestep this fact is to use a Sobolev space supported by the entire $\mathbb{R}_{+}$. Yet, standard Sobolev spaces supported by $\mathbb{R}_{+}$ fail to accommodate with our purpose, since, as made clear by the proof below, constant functions are required to belong to the space of test functions. Therefore, instead of a standard Sobolev space, we may use a weighted Sobolev space, provided that the weight satisfies suitable integrability properties.

In order to state our CLT on the whole time interval, the second approach is preferred. Furthermore, the weights of the Sobolev spaces are chosen to be polynomial (see Section \ref{sec:preliminaries:sobolev} below). This choice is quite convenient because Sobolev spaces with polynomial weights are well-documented in the literature. In particular, results on the connection between spaces weighted by different powers, Sobolev embedding theorems and Maurin’s theorem, are well-known. It is worth noting that, provided that constant functions can be chosen as test functions, the precise value of the power in the polynomial weight of the Sobolev space does not really matter in our analysis: more generally, a different choice of family of weights would have been possible and, somehow, it would have led to a result equivalent to ours.
In this regard, we stress, at the end of the paper, the fact that our result in infinite horizon is in fact equivalent to what we would have obtained by implementing the first of the two  approaches mentioned above instead of the second one: roughly speaking, one can recover our result by sticking together the CLTs obtained on each finite interval of the form $[0,\theta]$, for $\theta\geq 0$; conversely, one can prove, from our statement, that, on any finite interval $[0,\theta]$, the CLT holds true in the dual space of a standard Sobolev space supported by $K_{\theta}$.\\

The Hilbertian approach used in this article has been already implemented in the diffusion processes framework \cite{Meleard_97,jourdain1998propagation,luccon2015transition,Meleard_96}. Let us mention here what are the main differences between these earlier results and ours:
\begin{itemize}
\item Under general non-degeneracy conditions, the marginal laws of a diffusion process are not compactly supported. The unboundedness of the support imposes the choice of weighted Sobolev spaces even in finite time horizon. In this framework, Sobolev spaces with polynomial weights are especially adapted to carry solutions with moments that are finite up to some order only. In that case, the choice of the power in the weight is explicitly prescribed by the maximal order up to which the solution has a finite moment. As already mentioned, this differs from our case: in the present article, the particles (namely, the ages of the neurons) are compactly supported over any finite time interval and thus, have finite moments of any order. Once again, this is the reason why the choice of the power, and more generally of the weight, in the Sobolev space is much larger.
\item Unlike point processes, diffusion processes are time continuous. Also, their generator is both local and of second order, whereas the generator for the point process identified in the mean-field limit in \cite{chevallier2015mean} is both of the first order and nonlocal. As a first consequence, the indices of regularity of the various Sobolev spaces used in this paper differ from those used in the diffusive framework. Also, the space of trajectories cannot be the same: although the limit process in our CLT has continuous trajectories, we must work with a space of càdlàg functions in order to accommodate with the jumps of the fluctuation process. Surprisingly, jumps do not just affect the choice of the functional space used to state the CLT (namely space of càdlàg versus space of continuous functions) but it also dictates the metric used to estimate the error in the Sznitman coupling between the age-dependent Hawkes process and its mean-field limit (which is also a point process). Indeed, the standard trick used for diffusion processes that consists in getting stronger estimates for the Sznitman coupling by considering $L^{p}$-norms, for $p>2$, is not adapted to point processes. Therefore, we develop a specific approach by providing higher order estimates of the error in the Sznitman coupling in the total variation sense. Up to our knowledge, this argument is completely new.
\end{itemize}
Let us mention that the fluctuations of jump processes have been the object of previous publications \cite{riedler2012limit,wainrib2010randomness}. However, the CLTs are established in the fluid limit, namely small jumps at high frequency so that the jumps vanish at the limit. The techniques developed in those articles are useless here since the framework of the present article does not fall into the fluid limit framework: in our case, the limit processes are also jump processes. \\

The present paper is organized as follows. The model is described in Section \ref{sec:definitions:and:propagation}. Then, the main estimates required in this work are given in Section \ref{sec:estimates:total:variation}. These can be seen as the extension, to higher orders, of the estimates used in \cite{chevallier2015mean} to get the bound $n^{-1/2}$ on the rate of the convergence \eqref{eq:LLN:intro}. These key estimates are used to prove tightness for the distribution $\eta^{n}$ in a Hilbert space that is the dual of some weighted Sobolev space. Under regularity assumptions on the intensity function $\Psi$ and the interaction function $h$, we finally prove in Section \ref{sec:uniqueness:limit:equation} the convergence of the fluctuation process which states our CLT. Furthermore, its limit is characterized by a system of stochastic differential equations, driven by a Gaussian process with explicit covariance, and involving an auxiliary process with values in $\mathbb{R}$ (Theorem \ref{thm:convergence:eta:Gamma}). Finally, the CLT is applied to give some justification to a stochastic partial differential equation which can be seen as a better approximation than the PDE system \eqref{eq:edp:PPS} in the mean-field limit.

\paragraph{General notations}
\begin{itemize}
\item Statistical distributions are referred to as laws of random variables to avoid confusion with distributions in the analytical sense that are linear forms acting on some test function space.
\item The space of bounded functions of class $\mathcal{C}^{k}$, with bounded derivatives of each order less than $k$ is denoted by $\mathcal{C}^{k}_{b}$.
\item The space of càdlàg (right continuous with left limits) functions is denoted by $\mathcal{D}$.
\item For $\mu$ a measure on $E$ and $\varphi$ a function on $E$, we denote $\left<\mu,\varphi\right>:=\int_{E} \varphi(x)\mu(dx)$ when it makes sense.
\item If a quantity $Q$ depends on the time variable $t$, then we most often use the notation $Q_{t}$ when it is a random process in comparison with $Q(t)$ when it is a deterministic function.
\item We say that the quantity $Q_{n}(\sigma)$, which depends on an integer $n$ and a parameter $\sigma\in \mathbb{R}^{d}$, is bounded up to a locally bounded function (which does not depend on $n$) by $f(n)$, denoted by $Q_{n}(\sigma)\lesssim_{\sigma} f(n)$, if there exists a locally bounded function $g:\mathbb{R}^{d}\to \mathbb{R}_{+}$ such that, for all $n$, $|Q_{n}(\sigma)|\leq g(\sigma)f(n)$.
\item Throughout this paper, $C$ denotes a constant that may change from line to line.
\end{itemize}

\section{Definitions and propagation of chaos}
\label{sec:definitions:and:propagation}

In all the sequel, we focus on locally finite point processes, $N$, on $(\mathbb{R},\mathcal{B}(\mathbb{R}))$ that are random countable sets of points of $\mathbb{R}$ such that for any bounded measurable set $A\subset \mathbb{R}$, the number of points in $N\cap A$ is finite almost surely (a.s.). The associated points define an ordered sequence $(T_{n})_{n\in \mathbb{Z}}$.
For a measurable set $A$, $N(A)$ denotes the number of points of $N$ in $A$. We are interested in the behaviour of $N$ on $(0,+\infty)$ and we denote $t\in\mathbb{R}_{+} \mapsto N_t:=N((0,t])$ the associated counting process. Furthermore, the point measure associated with $N$ is denoted by $N(dt)$. In particular, for any non-negative measurable function $f$, $\int_{\mathbb{R}} f(t) N(dt) = \sum_{i\in \mathbb{Z}} f(T_{i})$. For any point process $N$, we call \emph{age process} associated with $N$ the process $(S_{t})_{t\geq 0}$ given by
\begin{equation}\label{eq:def:age:process}
S_{t}= t-\sup\{T\in N,\, T\leq t\}, \quad \text{ for all } t\geq 0.
\end{equation}
In comparison with the age process, we call \emph{predictable age process} associated with $N$ the predictable process $(S_{t-})_{t\geq 0}$ given by
\begin{equation}\label{eq:def:age:process:predictable}
S_{t-}= t-\sup\{T\in N,\, T<t\}, \quad \text{ for all } t>0,
\end{equation}
and extended by continuity in $t=0$.

We work on a filtered probability space $(\Omega,\mathcal{F},(\mathcal{F}_{t})_{t\geq 0},\mathbb{P})$ and suppose that the canonical filtration associated with $N$, namely $(\mathcal{F}_{t}^{N})_{t\geq 0}$ defined by $\mathcal{F}_{t}^{N}:=\sigma(N\cap (-\infty,t])$, is such that for all $t\geq 0$, $\mathcal{F}_{t}^{N}\subset	\mathcal{F}_{t}$. 
Let us denote $\mathbb{F}:=(\mathcal{F}_{t})_{t\geq 0}$.
We call $\mathbb{F}$-(predictable) intensity of $N$ any non-negative $\mathbb{F}$-predictable process $(\lambda_{t})_{t\geq 0}$ such that $(N_{t}-\int_{0}^{t} \lambda_{s}ds)_{t\geq 0}$ is an $\mathbb{F}$-local martingale. Informally, $\lambda_{t}dt$ represents the probability that the process $N$ has a new point in $[t,t+dt]$ given $\mathcal{F}_{t-}$. Under some assumptions that are supposed here, this intensity process exists, is essentially unique and characterizes the point process (see \cite{Bremaud_PP} for more insights). In particular, since $N$ admits an intensity, for any $t\geq 0$, the probability that $t$ belongs to $N$ is null. Moreover, notice the following properties satisfied by the age processes:
\begin{itemize}
\item the two age processes are equal for all $t\geq 0$ except the positive times $T$ in $N$ (almost surely a set of null measure in $\mathbb{R}_{+}$),
\item for any fixed $t\geq 0$, $S_{t-}=S_{t}$ almost surely (since $N$ admits an intensity),
\item and the value $S_{0-}=S_{0}$ is entirely determined by $N\cap\mathbb{R}_-$ and is well-defined as soon as there is a point therein.
\end{itemize}
\vspace{0.2cm}

The exact behaviour of $N\cap\mathbb{R}_-$ is not of great interest in the present article. We only assume that there is a point in it almost surely such that $S_{0-}=S_{0}$ is well-defined. Furthermore, we assume that the random variable $S_{0}$ admits $u_0$ as a probability density.

\subsection{Parameters and list of assumptions}

The definition of an \emph{age dependent Hawkes process} (ADHP) is given bellow, but let us first introduce the parameters of the model:
\begin{itemize}
\item a positive integer $n$ which is the number of particles (e.g. neurons) in the network (for $i=1,\dots ,n$, $N^{i}$ represents the occurrences of the events, e.g. spikes, associated with the particle $i$);
\item a probability density $u_0$;
\item an interaction function $h:\mathbb{R}_{+}\to \mathbb{R}$;
\item an intensity function $\Psi:\mathbb{R}_{+}\times \mathbb{R} \rightarrow \mathbb{R}_{+}$.
\end{itemize}

For sake of simplicity, all the assumptions made on the parameters are gathered here:
\smallskip

\newlength{\asslabel}
\setlength{\asslabel}{1.8cm}
\newlength{\assdef}
\setlength{\assdef}{11.5cm}
\newcommand{\spacenot}{\vspace{-0.2cm}}
\noindent
\begin{tabular}{>{\centering}p{\asslabel}|p{\assdef}}
\phantomsection
\label{ass:initial:condition:density:compact:support} 
\spacenot $\left(\mathcal{A}^{u_0}_{\infty}\right)$:   & The probability density $u_0$ is uniformly bounded with compact support so that there exists a constant $C>0$ such that $S_{0}\leq C$ almost surely (a.s.). The smallest possible constant $C$ is denoted by $M_{S_{0}}$.
\end{tabular}

\bigskip
\noindent
\begin{tabular}{>{\centering}p{\asslabel}|p{\assdef}}
\phantomsection
\label{ass:h:infty} 
\spacenot $\left(\mathcal{A}^{h}_{\infty}\right)$:      & 		The interaction function $h$ is locally bounded. Denote by, for all $t\geq 0$, $h_{\infty}(t):=\max_{s\in [0,t]} h(s)<+\infty$.
\end{tabular}

\smallskip
\noindent
\begin{tabular}{>{\centering}p{\asslabel}|p{\assdef}}
\phantomsection
\label{ass:h:Holder} 
\spacenot $\left(\mathcal{A}^{h}_{\rm H\ddot{o}l}\right)$:      & 		There exist two positive constants denoted by ${\rm H\ddot{o}l}(h)$ and $\beta(h)$ such that for all $t,s\geq 0$, $|h(t)-h(s)|\leq {\rm H\ddot{o}l}(h)|t-s|^{\beta(h)}$.
\end{tabular}

\bigskip
\noindent
\begin{tabular}{>{\centering}p{\asslabel}|p{\assdef}}
\phantomsection
\label{ass:Psi:C2} 
\spacenot $\left(\mathcal{A}^{\Psi}_{y,\mathcal{C}^{2}}\right)$:     & For all $s\geq 0$, the function $\Psi_{s}:y\mapsto \Psi(s,y)$ is of class $\mathcal{C}^{2}$.
Furthermore, $||\frac{\partial \Psi}{\partial y}||_{\infty}:=\sup_{s,y} |\frac{\partial \Psi}{\partial y}(s,y)|<+\infty$ and $||\frac{\partial^{2} \Psi}{\partial y^{2}}||_{\infty}<+\infty$. The constant $||\frac{\partial \Psi}{\partial y}||_{\infty}$ is denoted by ${\rm Lip}(\Psi)$.
\end{tabular}

\smallskip
\noindent
\begin{tabular}{>{\centering}p{\asslabel}|p{\assdef}}
\phantomsection
\label{ass:Psi:uniformly:bounded:fluctuations} 
\spacenot $\left(\mathcal{A}^{\Psi}_{\infty}\right)$:      & The function $\Psi$ is uniformly bounded, that is $|| \Psi ||_{\infty}<+\infty$.
\end{tabular}

\smallskip
\noindent
\begin{tabular}{>{\centering}p{\asslabel}|p{\assdef}}
\phantomsection
\label{ass:Psi:dPsi:C2b} 
\spacenot $\left(\mathcal{A}^{\Psi}_{s,\mathcal{C}^{2}_{b}}\right)$:      & For all $y$ in $\mathbb{R}$, the functions $s\mapsto \Psi(s,y)$ and $s\mapsto \frac{\partial \Psi}{\partial y} (s,y)$ respectively belong to $\mathcal{C}^{2}_{b}$ and $\mathcal{C}^{1}_{b}$. Furthermore, the functions $y\mapsto ||\Psi(\cdot,y)||_{\mathcal{C}^{2}_{b}}$ and $y\mapsto ||\frac{\partial \Psi}{\partial y}(\cdot,y)||_{\mathcal{C}^{1}_{b}}$ are locally bounded\footnotemark.
\end{tabular}
\footnotetext{The definitions of the norms $||\cdot||_{\mathcal{C}^{k}_{b}}$, for all $k\geq 0$, can be found in Section \ref{sec:preliminaries:sobolev}}

\smallskip
\noindent
\begin{tabular}{>{\centering}p{\asslabel}|p{\assdef}}
\phantomsection
\label{ass:Psi:C4b}
\spacenot $\left(\mathcal{A}^{\Psi}_{s,\mathcal{C}^{4}_{b}}\right)$:      & For all $y$ in $\mathbb{R}$, the function $s\mapsto \Psi(s,y)$ belongs to $\mathcal{C}^{4}_{b}$ and $y\mapsto ||\Psi(\cdot,y)||_{\mathcal{C}^{4}_{b}}$ is locally bounded.
\end{tabular}

\medskip

\begin{rem}
Note that:
\begin{itemize}
\item Assumption (\hyperref[ass:h:Holder]{$\mathcal{A}^{h}_{\rm H\ddot{o}l}$}) implies Assumption (\hyperref[ass:h:infty]{$\mathcal{A}^{h}_{\infty}$}),
\item the assumptions regarding the intensity function $\Psi$ are rather technical, nevertheless Assumptions (\hyperref[ass:Psi:C2]{$\mathcal{A}^{\Psi}_{y,\mathcal{C}^{2}}$}), (\hyperref[ass:Psi:uniformly:bounded:fluctuations]{$\mathcal{A}^{\Psi}_{\infty}$}) and (\hyperref[ass:Psi:dPsi:C2b]{$\mathcal{A}^{\Psi}_{s,\mathcal{C}^{2}_{b}}$}) are satisfied as soon as $\Psi$ belongs to $\mathcal{C}^{2}_{b}$. Furthermore, Assumption (\hyperref[ass:Psi:C4b]{$\mathcal{A}^{\Psi}_{s,\mathcal{C}^{4}_{b}}$}) is satisfied if $\Psi$ is in $\mathcal{C}^{4}_{b}$.
\end{itemize}
\end{rem}

Let \phantomsection \label{ass:for:LLN}$(\mathcal{A}_{\text{\tiny{LLN}}})$ be satisfied if (\hyperref[ass:initial:condition:density:compact:support]{$\mathcal{A}^{u_0}_{\infty}$}), (\hyperref[ass:h:infty]{$\mathcal{A}^{h}_{\infty}$}), (\hyperref[ass:Psi:C2]{$\mathcal{A}^{\Psi}_{y,\mathcal{C}^{2}}$}) and (\hyperref[ass:Psi:uniformly:bounded:fluctuations]{$\mathcal{A}^{\Psi}_{\infty}$}) are satisfied. These four assumptions also appear in \cite{chevallier2015mean}, where they are used to prove propagation of chaos as stressed below. Furthermore, let \phantomsection \label{ass:for:TGN}$(\mathcal{A}_{\text{\tiny{TGN}}})$ be satisfied if (\hyperref[ass:for:LLN]{$\mathcal{A}_{\text{\tiny{LLN}}}$}) and (\hyperref[ass:Psi:dPsi:C2b]{$\mathcal{A}^{\Psi}_{s,\mathcal{C}^{2}_{b}}$}) are satisfied. It is used in the present article to prove tightness of the fluctuations. Finally, let \phantomsection \label{ass:for:CLT}$(\mathcal{A}_{\text{\tiny{CLT}}})$ be satisfied if (\hyperref[ass:for:TGN]{$\mathcal{A}_{\text{\tiny{TGN}}}$}), (\hyperref[ass:h:Holder]{$\mathcal{A}^{h}_{\rm H\ddot{o}l}$}) and (\hyperref[ass:Psi:C4b]{$\mathcal{A}^{\Psi}_{s,\mathcal{C}^{4}_{b}}$}) are satisfied. It is used in the present article to prove convergence of the fluctuations.

Notice that Assumption (\hyperref[ass:initial:condition:density:compact:support]{$\mathcal{A}^{u_0}_{\infty}$}) implies that the age processes associated with $N$ are such that, almost surely,
\begin{equation}\label{eq:control:age:uniform}
\text{for all $t\geq 0$, } S_{t}\leq M_{S_{0}}+t \text{ and } S_{t-}\leq M_{S_{0}}+t.
\end{equation}

\subsection{Already known results}

Below is given the definition of an ADHP by providing its representation as a system of stochastic differential equations (SDE) driven by Poisson noise.

\begin{representation}\label{def:ADHP:Thinning}
Let $(\Pi^{i}(dt,dx))_{i\geq 1}$ be some i.i.d. $\mathbb{F}$-Poisson measures with intensity $1$ on $\mathbb{R}_{+}^{2}$. Let $(S_{0}^{i})_{i\geq 1}$ be some i.i.d. random variables distributed according to $u_0$.

Let $(N_{t}^i)^{i=1,..,n}_{t \geq 0}$ be a family of counting processes such that, for $i=1,..,n$, and all $t\geq 0$,
\begin{equation} \label{eq:ADHP:Thinning}
N_{t}^i= \int_0^t \int_0^{+\infty} \mathds{1}_{\displaystyle \Big\{ x\leq \Psi\Bigg( S^{i}_{t'-}, \frac{1}{n} \sum_{j=1}^{n} \Big( \int_{0}^{t'-} h(t'-z) N^{j}(dz) \Big) \Bigg) \Big\}} \, \Pi^i(dt',dx),
\end{equation}
where $(S^{i}_{t-})_{t\geq 0}$ is the predictable age process associated with $N^{i}$. Then, $(N^{i})_{i=1,..,n}$ is an age dependent Hawkes process (ADHP) with parameters $(n,h,\Psi,u_0)$.
\end{representation}

\begin{rem}
Note that an ADHP is in fact a (deterministic) measurable function of the Poisson measures $(\Pi^{i}(dt,dx))_{i\geq 1}$. More classically, an ADHP can be characterized by its stochastic intensity \eqref{eq:intensity:ADHP}. Going back and forth between the definition via the intensities \eqref{eq:intensity:ADHP} and Representation \ref{def:ADHP:Thinning} is standard (see \cite[Section II.4]{chevallier2015mean} for more insights). Furthermore, \cite[Proposition II.4]{chevallier2015mean} gives that, under  Assumption (\hyperref[ass:for:LLN]{$\mathcal{A}_{\text{\tiny{LLN}}}$}), there exists an ADHP  $(N^i)_{i=1,..,n}$ with parameters $(n,h,\Psi,u_0)$ such that $t\mapsto \mathbb{E}[ N_{t}^{1} ]$ is locally bounded.

Notice that, since the initial conditions $(S_{0}^{i})_{i=1,..,n}$ are i.i.d. and the Poisson measures $(\Pi^{i}(dt,dx))_{i\geq 1}$ are i.i.d., the processes $N^{i}$, $i=1,\dots ,n$, defined by \eqref{eq:ADHP:Thinning} are exchangeable.
\end{rem}

Here, we give a brief overview of the results obtained in \cite{chevallier2015mean} in order to set the context of the present article. We expect ADHPs to be well approximated, when $n$ goes to infinity, by i.i.d. solutions of the following \emph{limit equation},
\begin{equation}\label{eq:limit:equation:counting:process:fluctuations}
\!\!\!\! \forall t>0,\  \overline{N}_t= \int_0^t \int_0^{+\infty} \mathds{1}_{\displaystyle \Big\{ x\leq \Psi\bigg( \overline{S}_{t'-},  \int_{0}^{t'-} h(t'-z) \mathbb{E} \Big[\overline{N}(dz) \Big] \bigg) \Big\}}  \, \Pi(dt',dx),
\end{equation}
where $\Pi(dt', dx)$ is an $\mathbb{F}$-Poisson measure on $\mathbb{R}_{+}^{2}$ with intensity $1$ and $(\overline{S}_{t-})_{t\geq 0}$ is the predictable age process associated with $\overline{N}$ where $\overline{S}_{0}$ is distributed according to $u_0$. 

Under Assumption (\hyperref[ass:for:LLN]{$\mathcal{A}_{\text{\tiny{LLN}}}$}), \cite[Proposition III.6]{chevallier2015mean} states existence and uniqueness of the limit process $\overline{N}$. In particular, there exists a continuous function $\overline{\lambda}:\mathbb{R}_{+}\to \mathbb{R}$ (which depends on the parameters $h$, $\Psi$ and $u_0$) such that if $(\overline{N}_t)_{t\geq 0}$ is a solution of \eqref{eq:limit:equation:counting:process:fluctuations} then $\mathbb{E} [\overline{N}(dt) ]=\overline{\lambda}(t)dt$. Let us define the deterministic function $\overline{\gamma}$ by, for all $t\geq 0$, 
\begin{equation}\label{eq:def:gamma:barre:t}
\overline{\gamma}(t):=\int_{0}^{t} h(t-z) \overline{\lambda}(z)dz.
\end{equation}
Notice that $\overline{\gamma}(t')$ is the integral term $\int_{0}^{t'-} h(t'-z) \mathbb{E} [\overline{N}(dz) ]$ appearing in \eqref{eq:limit:equation:counting:process:fluctuations}. 

Furthermore, the limit predictable age process $(\overline{S}_{t-})_{t\geq 0}$ is closely related to the PDE system \eqref{eq:edp:PPS}.

\begin{prop}[{\cite[Proposition III.8]{chevallier2015mean}}]\label{prop:age:process:limit:equation:PPS}
Under Assumption (\hyperref[ass:for:LLN]{$\mathcal{A}_{\text{\tiny{LLN}}}$}), the unique solution $u$ to the system \eqref{eq:edp:PPS} with initial condition that  $u_0$ is such that $u(t,\cdot)$ is the density of the age $\overline{S}_{t-}$ (or $\overline{S}_{t}$ since they are equal a.s.). 
\end{prop}

Once the limit equation is well-posed, following the ideas of Sznitman in \cite{Sznitman_91}, it is easy to construct a suitable coupling between ADHPs and i.i.d. solutions of the limit equation \eqref{eq:limit:equation:counting:process:fluctuations}. More precisely, consider
\begin{itemize}
\item a sequence $(S_{0}^{i})_{i\geq 1}$ of i.i.d. random variables distributed according to $u_0$;
\item a sequence $(\Pi^i(dt',dx))_{i\geq 1}$ of i.i.d. $\mathbb{F}$-Poisson measures with intensity $1$ on $\mathbb{R}_{+}^{2}$.
\end{itemize}
Under Assumption (\hyperref[ass:for:LLN]{$\mathcal{A}_{\text{\tiny{LLN}}}$}), we have existence of both ADHPs and the limit process $\overline{N}$.
Hence, one can build simultaneously:
\smallskip

\noindent - a sequence (indexed by $n\geq 1$) $(N^{n,i})_{i=1,\dots ,n}$ of ADHPs with parameters $(n,h,\Psi,u_0)$ according to Representation \ref{def:ADHP:Thinning} namely
\begin{equation}\label{eq:definition:coupling:MFHawkes}
N^{n,i}_t= \int_0^t \int_0^{+\infty} \mathds{1}_{\displaystyle  \left\{x \leq \Psi\left( S^{n,i}_{t'-}, \gamma^{n}_{t'} \right)  \right\}} \Pi^i(dt',dx)
\end{equation}
where $S^{n,i}_{0}=S^{i}_{0}$ and $\gamma^{n}_{t'}:=n^{-1}  \sum_{j=1}^n \int_{0}^{t'-} h(t'-z)N^{n,j}(dz)$,

\noindent - and a sequence $(\overline{N}^i_t)^{i\geq 1}_{t\geq 0}$ of i.i.d. solutions of the limit equation namely
\begin{equation}\label{eq:definition:coupling:limit:equation:fluctuations}
\overline{N}^{i}_t= \int_0^t \int_0^{+\infty} \mathds{1}_{\displaystyle \left\{x \leq \Psi \left( \overline{S}^{i}_{t'-}, \overline{\gamma}(t') \right)\right\}} \Pi^i(dt',dx),
\end{equation}
where $\overline{S}^{i}_{0}=S^{i}_{0}$ and $\overline{\gamma}$ is defined by \eqref{eq:def:gamma:barre:t}. 

Moreover, denote by $\lambda^{n,i}_{t}:=\Psi( S^{n,i}_{t-}, \gamma^{n}_{t})$ and $\overline{\lambda}^{i}_{t}:= \Psi( \overline{S}^{i}_{t-}, \overline{\gamma}(t) )$ the respective intensities of $N^{n,i}$ and $\overline{N}^{i}$.
\smallskip

\begin{rem}
Notice that the coupling above is based on the sharing of common initial conditions $(S_{0}^{i})_{i\geq 1}$ and a common underlying randomness, that are the $\mathbb{F}$-Poisson measures $(\Pi^i(dt',dx))_{i\geq 1}$.
Note also that the sequence of ADHPs is indexed by the size of the network $n$ whereas the solutions of the limit equation which represent the behaviour under the mean field approximation are not.
\end{rem}

Then, standard computations mainly based on Grönwall lemma lead to the following estimates \cite[Corollary IV.3]{chevallier2015mean}:
for all $i=1,\dots ,n$ and $\theta>0$, 
\begin{equation}\label{eq:control:coupling:L1:norm}
\mathbb{E}\left[ \sup_{t\in [0,\theta]} | S^{n,i}_{t-} -\overline{S}^i_{t-} |  \right] \lesssim_{\theta} \mathbb{P}\left(\left(S_{t-}^{n,i}\right)_{t\in\left[0,\theta\right]}\neq\left(\overline{S}_{t-}^{i}\right)_{t\in\left[0,\theta\right]}\right) \lesssim_{\theta} n^{-1/2}.
\end{equation}
Finally, these estimates ensure the propagation of chaos property\footnote{For any fixed integer $k$, the processes $(S^{n,1}_{t})_{t\geq 0},\dots ,(S^{n,k}_{t})_{t\geq 0}$ are asymptotically independent.} \cite[Corollary IV.4]{chevallier2015mean} and, in particular, the convergence (as $n\to +\infty$) of the empirical measure $\overline{\mu}^{n}_{S_{t}}:=\frac{1}{n}\sum_{i = 1}^n \delta_{S^{n,i}_{t}}$ towards the law of $\overline{S}^{1}_{t}$ for all $t\geq 0$.

\subsection{What next ? The purpose of the present paper}

As a straight follow-up to the convergence of the empirical measure $\overline{\mu}^{n}_{S_t}$, we are interested in the dynamics of the fluctuations of this empirical measure around its limit. For any $t\geq 0$,  $\overline{S}^{1}_{t}$ and $\overline{S}^{1}_{t-}$ have the same probability law since they are equal almost surely. Furthermore, this law, denoted by $u_t$ admits the density $u(t,\cdot)$ with respect to the Lebesgue measure, where $u$ is the unique solution of \eqref{eq:edp:PPS} according to Proposition \ref{prop:age:process:limit:equation:PPS}, thus
\begin{equation*}
\left< u_{t},\varphi \right>=\int_{0}^{+\infty} \varphi(s) u(t,s) ds.
\end{equation*}
The analysis of the coupling (Equation \eqref{eq:control:coupling:L1:norm}) gives a rate of convergence at least in $n^{-1/2}$ so we want to find the limit law of the fluctuation process defined, for all $t\geq 0$, by
\begin{equation}\label{eq:def:eta:n:t} 
\eta_{t}^{n}:=\sqrt{n}\left( \overline{\mu}^{n}_{S_{t}} - u_{t} \right).
\end{equation}
Notice that $\eta_{t}^{n}$ is a distribution in the functional analysis sense on the state space of the ages, i.e. $\mathbb{R}_{+}$, and is devoted to be considered as a linear form acting on test functions $\varphi$ by means of $\left< \eta^{n}_{t},\varphi \right>$.

\section{Estimates in total variation norm}
\label{sec:estimates:total:variation}

The bound ($n^{-1/2}$) on the rate of convergence, given by \eqref{eq:control:coupling:L1:norm}, is not sufficient in order to prove convergence or even tightness of the fluctuation process $\eta^{n}$. Some refined estimates are necessary. For instance, when dealing with diffusions, one looks for higher order moment estimates on the difference between the particles driven by the real dynamics and the limit particles (see \cite{Meleard_97,jourdain1998propagation,luccon2015transition,Meleard_96} for instance). Here, we deal with pure jump processes and, up to our knowledge, there is no reason why one could obtain better rates for higher order moments. A simple way to catch this fact is by looking at the coupling between the counting processes. Indeed, the difference between two counting processes, say $\delta^{n,i}_{t}=|N^{n,i}_{t}-\overline{N}^{i}_{t}|$, takes value in $\mathbb{N}$ so that for all $p\geq 1$, $(\delta^{n,i}_{t})^{p}\geq \delta^{n,i}_{t}$, and the moment of order $p$ is greater than the moment of order one. 

In order to accommodate this fact, the key idea is to estimate the coupling \eqref{eq:definition:coupling:MFHawkes}-\eqref{eq:definition:coupling:limit:equation:fluctuations} in the total variation distance. Hence, the estimates needed in the next section (and proved in the present section) are the analogous of higher order moments but with respect to the total variation norm, i.e. the probabilities
\begin{eqnarray}
\chi^{(k)}_{n}(\theta) &:=&\mathbb{P}\left( (S^{n,k'}_{t-})_{t\in [0,\theta]} \neq (\overline{S}^{k'}_{t-})_{t\in [0,\theta]} \emph{ for every } k'=1,... ,k\right) \nonumber\\
&=&\mathbb{P}\left( (S^{n,k'}_{t})_{t\in [0,\theta]} \neq (\overline{S}^{k'}_{t})_{t\in [0,\theta]} \emph{ for every } k'=1,...,k\right), \label{eq:def:gamma:k:n}
\end{eqnarray}
for all positive integer $k$ and real number $\theta\geq 0$. 

The heuristics underlying the result stated below, in Proposition \ref{prop:rate:of:convergence:proba:k:tuple:different:ages}, relies on the asymptotic independence between the $k$ age processes $(S^{n,k'}_{t-})_{t\in [0,\theta]}$, $k'=1,... ,k$. Indeed, if they were independent then we would have (remind \eqref{eq:control:coupling:L1:norm}),
$$\chi^{(k)}_{n}(\theta)=\prod_{k'=1}^{k} \mathbb{P}\big( (S^{n,k'}_{t-})_{t\in [0,\theta]} \neq (\overline{S}^{k'}_{t-})_{t\in [0,\theta]} \big)= (\chi^{(1)}_{n}(\theta))^{k}\lesssim_{\theta} n^{-k/2},$$
which is exactly the rate of convergence we find below.

\begin{prop}\label{prop:rate:of:convergence:proba:k:tuple:different:ages}
Under Assumption (\hyperref[ass:for:LLN]{$\mathcal{A}_{\text{\tiny{LLN}}}$}),
\begin{equation*}
\chi^{(k)}_{n}(\theta)\lesssim_{(\theta,k)} n^{-k/2} \quad \text{ and } \quad \xi^{(k)}_{n}(t):=\mathbb{E}\left[  |\gamma^{n}_{t} - \overline{\gamma}(t)| ^{k} \right] \lesssim_{(t,k)} n^{-k/2}.
\end{equation*}
\end{prop}
\begin{rem}
In addition to the explanation given in the beginning of this section, let us mention that the analogous to the higher moment estimates obtained for diffusions is obtained here for the difference between $\gamma^{n}_{t}$ and $\overline{\gamma}(t)$. Indeed, as $k$ grows, the convergence of $\xi^{(k)}_{n}(t)$ quickens. However, this gain in the rate of convergence does not apply when looking at the difference between the ages $S^{n,1}_{t}$ and $\overline{S}^{1}_{t}$ or the difference between the intensities $\lambda^{n,1}_{t}$ and $\overline{\lambda}^{1}_{t}$ (except if $\Psi$ does not depend on the age $s$).
\end{rem}
\begin{proof}
The core of this proof lies on a trick using the exchangeability of the processes in order to obtain Grönwall-type inequalities involving $\chi^{(k)}_{n}$ and $\xi^{(k)}_{n}$.

Denote by $A \triangle B$ the symmetric difference of the sets $A$ and $B$.
Then, for any $i\leq n$, let us define $\Delta^{n,i}:=N^{n,i}\Delta\overline{N}^{i}$ that is the set of points that are not common to $N^{n,i}$ and $\overline{N}^i$. 
From \eqref{eq:definition:coupling:MFHawkes}-\eqref{eq:definition:coupling:limit:equation:fluctuations}, one has
\begin{equation*}
\Delta^{n,i}_{t}= \int_0^t \int_0^{+\infty} \mathds{1}_{\displaystyle  \left\{x \in [[ \lambda^{n,i}_{t'},\overline{\lambda}^{i}_{t'}]] \right\}} \Pi^i(dt',dx),
\end{equation*}
where $[[ \lambda^{n,i}_{t'},\overline{\lambda}^{i}_{t'}]]$ is the non empty interval which is either $[ \lambda^{n,i}_{t'},\overline{\lambda}^{i}_{t'}]$ or $[ \overline{\lambda}^{i}_{t'},\lambda^{n,i}_{t'}]$.
Then, the intensity of the point process $\Delta^{n,i}$ is given by $\lambda^{\Delta,n,i}_{t}:=|\lambda^{n,i}_{t} - \overline{\lambda}^{i}_{t}|$. 

Note that, for all $n\geq 1$ and $i=1,\dots ,n$, $S^{n,i}_{0-}=\overline{S}^{i}_{0-}$ so that the equality between the processes $(S^{n,1}_{t-})_{t\in [0,\theta]}$ and $(\overline{S}^{1}_{t-})_{t\in [0,\theta]}$ is equivalent to $\Delta^{n,1}_{\theta-}=0$. In particular, one has 
\begin{equation}\label{eq:bound:Gammakn:by:epsilonkn}
\chi^{(k)}_{n}(\theta) \leq \mathbb{E}\left[ \prod_{i=1}^{k} \Delta^{n,i}_{\theta-} \right],
\end{equation}
since counting processes take value in $\mathbb{N}$.
For any positive integers $k$ and $p$, let us denote, for all $n\geq k$,
\begin{equation*}
\varepsilon^{(k,p)}_{n}(\theta):= \mathbb{E}\left[ \prod_{i=1}^{k} \left( \Delta^{n,i}_{\theta-} \right)^{p} \right].
\end{equation*}
Let us show, by induction on $k$, that 
\begin{equation}\label{eq:inductive:hypothesis}
\varepsilon^{(k,p)}_{n}(\theta)\lesssim_{(\theta,k,p)} n^{-k/2}
\end{equation}
which will end the proof thanks to \eqref{eq:bound:Gammakn:by:epsilonkn}.
First, note that the case $k=1$ and $p=1$ is already treated. Indeed, \cite[Theorem IV.1]{chevallier2015mean} gives
\begin{equation}\label{eq:rate:of:convergence:expectation:coupling:area}
\varepsilon^{(1,1)}_{n}(\theta)= \int_{0}^{\theta} \mathbb{E}\left[ |\lambda^{n,1}_{t}-\overline{\lambda}^{1}_{t}| \right] dt \lesssim_{\theta} n^{-1/2}.
\end{equation}
Then, note that for any two positive integers $p$ and $q$,
\begin{equation}\label{eq:ordering:epsilon}
\varepsilon^{(k,p)}_{n}(\theta)\leq \varepsilon^{(k,q)}_{n}(\theta) \text{ as soon as } p\leq q.
\end{equation}
This is due to the fact that counting processes take value in $\mathbb{N}$. The rest of the proof is divided in two steps: initialization and inductive step.

\paragraph{Step one.}
For $k=1$ and $p$ a positive integer, it holds that
\begin{equation}\label{eq:stieltjes:Delta:n:1:p}
(\Delta^{n,1}_{\theta-})^{p}= \sum_{p'=0}^{p-1} \binom{p}{p'} \int_{0}^{\theta-} (\Delta^{n,1}_{t-})^{p'} \Delta^{n,1}(dt).
\end{equation}
Indeed, each time the process $(\Delta^{n,1}_{t})_{t\geq 0}$ jumps (from $\Delta^{n,1}_{t-}$ to $\Delta^{n,1}_{t-}+1$) then $(\Delta^{n,1}_{t-})^{p}$ jumps from $(\Delta^{n,1}_{t-})^{p}$ to $(\Delta^{n,1}_{t-}+1)^{p}$ so the infinitesimal variation is $$(\Delta^{n,1}_{t-}+1)^{p}-(\Delta^{n,1}_{t-})^{p}=\sum_{p'=0}^{p-1} \binom{p}{p'} (\Delta^{n,1}_{t-})^{p'}.$$

The right-hand side of \eqref{eq:stieltjes:Delta:n:1:p} involves integrals of predictable processes, that are the $(\Delta^{n,1}_{t-})^{p'}$, with respect to a point measure under which it is convenient to take expectation.

More precisely, since $(\Delta^{n,1}_{t-})^{p'}\leq (\Delta^{n,1}_{t-})^{p}$ as soon as $0<p'\leq p-1$, it holds that
\begin{eqnarray}
\varepsilon^{(1,p)}_{n}(\theta)=\mathbb{E}\left[ (\Delta^{n,1}_{\theta-})^{p} \right] &\leq& \mathbb{E}\left[ \int_{0}^{\theta} \Delta^{n,1}(dt) \right]+ 2^{p} \mathbb{E}\left[ \int_{0}^{\theta} (\Delta^{n,1}_{t-})^{p} \Delta^{n,1}(dt) \right].\nonumber\\
&\leq & \varepsilon^{(1,1)}_{n}(\theta) + 2^{p} \int_{0}^{\theta} \mathbb{E}\left[ (\Delta^{n,1}_{t-})^{p} \lambda^{\Delta,n,1}_{t} \right] dt.\label{eq:epsilonp:epsilon1}
\end{eqnarray}
Yet the intensity $\lambda^{\Delta,n,1}_{t}$ is bounded by $||\Psi||_{\infty}$ and $\varepsilon^{(1,1)}_{n}(\theta)\lesssim_{\theta} n^{-1/2}$, see \eqref{eq:rate:of:convergence:expectation:coupling:area}, so
\begin{equation*}
\varepsilon^{(1,p)}_{n}(\theta)\lesssim_{(\theta,p)} n^{-1/2} + \int_{0}^{\theta} \varepsilon^{(1,p)}_{n}(t) dt,
\end{equation*}
and Lemma \ref{lem:Gronwall:Generalization} gives $\varepsilon^{(1,p)}_{n}(\theta)\lesssim_{(\theta,p)} n^{-1/2}$.

\paragraph{Step two.}
For all integers $k\geq 2$ and  $p\geq 1$, one can generalize the argument used to prove \eqref{eq:stieltjes:Delta:n:1:p} in order to end up with
\begin{equation*}
\prod_{i=1}^{k} (\Delta^{n,i}_{\theta-})^{p} = \sum_{j=1}^{k} \sum_{p'=0}^{p-1} \binom{p}{p'} \int_{0}^{\theta-} \prod_{i\neq j, i=1}^{k} (\Delta^{n,i}_{t-})^{p}  (\Delta^{n,j}_{t-})^{p'}  \Delta^{n,j}(dt), \quad \text{almost surely.}
\end{equation*}
Hence, thanks to the exchangeability of the processes $(\Delta^{n,i})_{i=1,\dots ,n}$ and the predictability of the integrated processes,  we have
\begin{eqnarray}
\!\!\! \varepsilon^{(k,p)}_{n}(\theta) &= & \sum_{j=1}^{k} \sum_{p'=0}^{p-1} \binom{p}{p'} \mathbb{E}\left[ \int_{0}^{\theta} \prod_{i\neq j, i=1}^{k} (\Delta^{n,i}_{t-})^{p}  (\Delta^{n,j}_{t-})^{p'}  \Delta^{n,j}(dt)  \right]\nonumber\\
&=& k \sum_{p'=0}^{p-1} \binom{p}{p'} \int_{0}^{\theta} \mathbb{E}\left[ (\Delta^{n,1}_{t-})^{p'} \prod_{i=2}^{k} (\Delta^{n,i}_{t-})^{p} \lambda^{\Delta,n,1}_{t}\right] dt \nonumber\\
& \leq & k  \int_{0}^{\theta} \mathbb{E}\left[ \prod_{i=2}^{k} (\Delta^{n,i}_{t-})^{p} \lambda^{\Delta,n,1}_{t}\right] + 2^{p} \mathbb{E}\left[ (\Delta^{n,1}_{t-})^{p} \prod_{i=2}^{k} (\Delta^{n,i}_{t-})^{p} \lambda^{\Delta,n,1}_{t}\right] dt, \label{eq:N1p:lambda:1:N2p}
\end{eqnarray}
where we used that $(\Delta^{n,1}_{t-})^{p'}\leq (\Delta^{n,1}_{t-})^{p}$ as soon as $0<p'\leq p-1$.

On the one hand, using that $\lambda^{\Delta,n,1}_{t}\leq ||\Psi||_{\infty}$, the second expectation in \eqref{eq:N1p:lambda:1:N2p} is bounded by $||\Psi||_{\infty} \varepsilon^{(k,p)}_{n}(t)$.

On the other hand, we use (\hyperref[ass:Psi:C2]{$\mathcal{A}^{\Psi}_{y,\mathcal{C}^{2}}$}) which gives the following bound on the intensity, 
$$\lambda^{\Delta,n,1}_{t}\leq {\rm Lip}(\Psi) |\gamma^{n}_{t} - \overline{\gamma}(t)|+ ||\Psi||_{\infty} \mathds{1}_{S^{n,1}_{t-}\neq \overline{S}^{1}_{t-}} \leq {\rm Lip}(\Psi) |\gamma^{n}_{t} - \overline{\gamma}(t)|+ ||\Psi||_{\infty}(\Delta^{n,1}_{t-})^{p}.$$ 
Hence the first  expectation in \eqref{eq:N1p:lambda:1:N2p} is bounded by
\begin{equation}\label{eq:N1p-1:N2p:gammadiff}
{\rm Lip}(\Psi) D(t) + ||\Psi||_{\infty} \varepsilon^{(k,p)}_{n}(t),
\end{equation}
with $D(t):=\mathbb{E}[ \prod_{i=2}^{k} (\Delta^{n,i}_{t-})^{p} |\gamma^{n}_{t} - \overline{\gamma}(t)| ]$. 
The second term of \eqref{eq:N1p-1:N2p:gammadiff} is convenient to use a Grönwall-type lemma. To deal with the first term, we use a trick involving the exchangeability of the particles. Indeed, using the exchangeability we can replace each of the $k-1$ terms $(\Delta^{n,i}_{t-})^{p}$ in the expression of $D(t)$ by the following sum
\begin{equation*}
\frac{1}{\lfloor \frac{n}{k} \rfloor} \sum_{j_{i}=(i-1) \lfloor \frac{n}{k} \rfloor +1}^{i\lfloor \frac{n}{k} \rfloor} (\Delta^{n,j_{i}}_{t-})^{p}
\end{equation*}
without modifying the value of the expectation since the sums are taken on disjoined indices. Hence, using for the second line a generalization of Hölder's inequality with $k$ exponents equal to $1/k$, we have
\begin{eqnarray}
\hspace{-2cm} D(t) &\leq & \mathbb{E}\left[  \prod_{i=2}^{k} \left( \frac{1}{\lfloor \frac{n}{k} \rfloor} \sum_{j_{i}=(i-1) \lfloor \frac{n}{k} \rfloor +1}^{i\lfloor \frac{n}{k} \rfloor} (\Delta^{n,j_{i}}_{t-})^{p} \right) |\gamma^{n}_{t} - \overline{\gamma}(t)| \right] \nonumber\\
&\leq & \left( \prod_{i=2}^{k} \mathbb{E}\left[ \left(\frac{1}{\lfloor \frac{n}{k} \rfloor} \sum_{j=1}^{\lfloor \frac{n}{k} \rfloor} (\Delta^{n,j}_{t-})^{p} \right)^{k} \right]^{1/k} \right) \xi^{(k)}_{n}(t)^{1/k} \leq E_{n,k,p}(t)^{\frac{k-1}{k}} \xi^{(k)}_{n}(t)^{1/k}, \label{eq:trick:exchang+Holder:p:p}
\end{eqnarray}
with $E_{n,k,p}(t):=\mathbb{E}[ ( (1/\lfloor \frac{n}{k} \rfloor) \sum_{j=1}^{\lfloor \frac{n}{k} \rfloor} (\Delta^{n,j}_{t-})^{p} )^{k} ]$. 
Yet, computations given in Section \ref{sec:proof:prop:rate:of:convergence} give the two following statements: there exists a constant $C(k)$ which does not depend on $n$ or $p$ such that
\begin{equation}\label{eq:bound:Enkp}
E_{n,k,p}(t) \leq C(k) \left( \sum_{k'=1}^{k-1} n^{k'-k} \varepsilon^{(k',pk)}_{n}(t) + \varepsilon^{(k,p)}_{n}(t) \right),
\end{equation} 
and $\xi^{(k)}_{n}(t)$ satisfy the following bound,
\begin{equation}\label{eq:bound:xikn}
\xi^{(k)}_{n}(t) \lesssim_{(t,k)}  n^{-k/2} + \sum_{k'=1}^{k-1} n^{k'-k} \varepsilon^{(k',k)}_{n}(t) + \varepsilon^{(k,1)}_{n}(t).
\end{equation}
Then, using the induction hypothesis \eqref{eq:inductive:hypothesis}, that is for all $1\leq k'\leq k-1$ and for all positive integer $p$, $\varepsilon^{(k',p)}_{n}(t)\lesssim_{(t,k,p)} n^{-k'/2}$, one has,
\begin{equation}\label{eq:Enkp:xin:bounds}
\hspace{-0.5cm}
\begin{cases}
E_{n,k,p}(t)\lesssim_{(t,k,p)} \! \sum_{k'=1}^{k-1} n^{k'-k} n^{-k'/2} + \varepsilon^{(k,p)}_{n}(t) \lesssim_{(t,k,p)} \! n^{-(k+1)/2} + \varepsilon^{(k,p)}_{n}(t)\\
\xi^{(k)}_{n}(t)\lesssim_{(t,k,p)} \! n^{-k/2} + \sum_{k'=1}^{k-1} n^{k'-k} n^{-k'/2} + \varepsilon^{(k,1)}_{n}(t) \lesssim_{(t,k,p)} \! n^{-k/2} + \varepsilon^{(k,1)}_{n}(t).
\end{cases}
\end{equation}

Gathering \eqref{eq:N1p:lambda:1:N2p}, \eqref{eq:N1p-1:N2p:gammadiff}, \eqref{eq:trick:exchang+Holder:p:p} and \eqref{eq:Enkp:xin:bounds} gives (remind that $\varepsilon^{(k,1)}_{n}(t)\leq \varepsilon^{(k,p)}_{n}(t)$)
\begin{equation*}
\varepsilon^{(k,p)}_{n}(\theta)\lesssim_{(\theta,k,p)} n^{-k/2} + \int_{0}^{\theta} \varepsilon^{(k,p)}_{n}(t) dt,
\end{equation*}
and so the Grönwall-type Lemma \ref{lem:Gronwall:Generalization} gives $\varepsilon^{(k,p)}_{n}(\theta)\lesssim_{(\theta,k,p)} n^{-k/2}$ which ends the proof thanks to \eqref{eq:bound:Gammakn:by:epsilonkn}.

\end{proof}

\section{Tightness}
\label{sec:tightness}

The aim of this section is to prove tightness of the sequence of the laws of $(\eta^{n})_{n\geq 1}$ regarded as stochastic processes (in time) with values in a suitable space of distributions. Thus, we consider $(\eta^{n}_{t})_{t\geq 0}$ as a random process with values in the dual space of some well-chosen space of test functions. In Section \ref{sec:preliminaries:sobolev}, we give the definition of these spaces of test functions. Following the Hilbertian approach developed in \cite{Meleard_97}, we work with weighted Sobolev Hilbert spaces. Finally, the tightness result is stated in Theorem \ref{thm:tightness:M:n:eta:n}.\\

The following study takes benefit of the Hilbert structure of the Sobolev spaces considered. Let us state here the Aldous tightness criterion for Hilbert space valued stochastic processes (cf. \cite[p. 34-35]{joffe1986weak}) used in the present paper. \phantomsection \label{sec:tightness:criterion}
Let $H$ be a separable Hilbert space. A sequence of processes $(X^{n})_{n\geq 1}$ in $\mathcal{D}(\mathbb{R}_{+},H)$ defined on the respective filtered probability spaces $(\Omega^{n},\mathcal{F}^{n},(\mathcal{F}^{n}_{t})_{t\geq 0},\mathbb{P}^{n})$ is tight if both conditions below hold true:
\smallskip

\noindent
\begin{tabular}{>{\centering}p{\asslabel}|p{\assdef}}
\phantomsection
\label{condition:A1}
\spacenot $\left(\texttt{A}_{1}\right)$:      & 		for every $t\geq 0$ and $\varepsilon>0$, there exists a compact set $K\subset H$ such that $$\sup_{n\geq 1} \mathbb{P}^{n}\left( X^{n}_{t}\notin K \right)\leq \varepsilon,$$
\end{tabular}
\smallskip

\noindent
\begin{tabular}{>{\centering}p{\asslabel}|p{\assdef}}
\phantomsection
\label{condition:A2}
\spacenot $\left(\texttt{A}_{2}\right)$:      & 		for every $\varepsilon_{1},\varepsilon_{2}>0$ and $\theta\geq 0$, there exists $\delta_{0}>0$ and an integer $n_{0}$ such that for all $(\mathcal{F}^{n}_{t})_{t\geq 0}$-stopping time $\tau_{n}\leq \theta$,
\begin{equation*}
\sup_{n\geq n_{0}} \sup_{\delta\leq \delta_{0}} \mathbb{P}^{n}\left( ||X^{n}_{\tau_{n}+\delta} - X^{n}_{\tau_{n}}||_{H}\geq \varepsilon_{1} \right)\leq \varepsilon_{2}.
\end{equation*}
\end{tabular}
\smallskip

Note that \hyperref[condition:A1]{$(\texttt{A}_{1})$} is implied by the condition \hyperref[condition:A1']{$(\texttt{A}_{1'})$} stated below which is much easier to ensue.
\smallskip

\noindent
\begin{tabular}{>{\centering}p{\asslabel}|p{\assdef}}
\phantomsection
\label{condition:A1'}
\spacenot $\left(\texttt{A}_{1'}\right)$:      & 		There exists a Hilbert space $H_{0}$ such that $H_{0}\hookrightarrow_{K} H$ and, for all $t\geq 0$, 
$$\sup_{n\geq 1} \mathbb{E}^{n}[|| X^{n}_{t} ||^{2}_{H_{0}}]<+\infty,$$ 
where the notation $\hookrightarrow_{K}$ means that the embedding is compact and $\mathbb{E}^{n}$ denotes the expectation associated with the probability $\mathbb{P}^{n}$.
\end{tabular}
\smallskip

The fact that \hyperref[condition:A1']{$(\texttt{A}_{1'})$} implies \hyperref[condition:A1]{$(\texttt{A}_{1})$} is easily checked: by compactness of the embedding, closed balls in $H_{0}$ are compact in $H$ so, Markov's inequality gives \hyperref[condition:A1]{$(\texttt{A}_{1})$}.

\subsection{Preliminaries on weighted Sobolev spaces}
\label{sec:preliminaries:sobolev}

Here are listed some definitions and technical results about the weighted Sobolev spaces used in the present article. To avoid confusion, let us stress the fact that the test functions we use are supported in the state space of the ages, namely $\mathbb{R}_{+}$. For any integer $k$ and any real $\alpha$ in $\mathbb{R}_{+}$, we denote by $\mathcal{W}^{k,\alpha}_{0}:=\mathcal{W}^{k,\alpha}_{0}(\mathbb{R}_{+})$ the completion of the set of compactly  supported (in $\mathbb{R}_{+}$) functions of class $\mathcal{C}^{\infty}$ for the following norm
\begin{equation*}
||f||_{k,\alpha}:=\left( \sum_{k'= 0}^{k} \int_{\mathbb{R}_{+}} \frac{|f^{(k')}(x)|^{2}}{1+|x|^{2\alpha}} dx \right)^{1/2},
\end{equation*}
where $f^{(j)}$ denotes the $j$\textsuperscript{th} derivative of $f$. Then, $\mathcal{W}^{k,\alpha}_{0}$ equipped with the norm $||\cdot||_{k,\alpha}$ is a separable Hilbert space and we denote $(\mathcal{W}^{-k,\alpha}_{0},||\cdot||_{-k,\alpha})$ its dual space. Notice that
\begin{equation}\label{eq:inclusion:sobolev}
\begin{cases}
\text{if $k'\geq k$, then } ||.||_{k,\alpha}\leq ||.||_{k',\alpha} \text{ and } ||.||_{-k',\alpha}\leq ||.||_{-k,\alpha},\\
\text{if $\alpha'\geq \alpha$, then } \mathcal{W}^{k,\alpha}_{0}\hookrightarrow \mathcal{W}^{k,\alpha'}_{0} \text{ and } \mathcal{W}^{-k,\alpha'}_{0}\hookrightarrow \mathcal{W}^{-k,\alpha}_{0},\\
\end{cases}
\end{equation}
where the notation $\hookrightarrow$ means that the embedding is continuous.

Let $\mathcal{C}^{k,\alpha}$ be the space of functions $f$ on $\mathbb{R}_{+}$ with continuous derivatives up to order $k$ such that, for all $k'\leq k$, $\sup_{x\in \mathbb{R}_{+}} |f^{(k')}(x)|/(1+|x|^{\alpha})<+\infty$. We equip this space with the norm
\begin{equation*}
||f||_{\mathcal{C}^{k,\alpha}}:= \sum_{k'=0}^{k} \sup_{x\in \mathbb{R}_{+}} \frac{|f^{(k')}(x)|}{1+|x|^{\alpha}}.
\end{equation*}
Recall that $\mathcal{C}^{k}_{b}$ is the space of bounded functions of class $\mathcal{C}^{k}$ with bounded derivatives of every order less than $k$. Notice that $\mathcal{C}^{k}_{b}=\mathcal{C}^{k,0}$ as normed spaces. Denote by $\mathcal{C}^{-k}_{b}$ its dual space. For any $\alpha>1/2$ and any integer $k$ (so that $\int_{\mathbb{R}_{+}} 1/(1+|x|^{2\alpha}) dx<+\infty $),  we have $\mathcal{C}^{k}_{b}\hookrightarrow \mathcal{W}^{k,\alpha}_{0}$, i.e. there exists a constant $C$ such that
\begin{equation}\label{eq:f:k:alpha:leq:f:Ckb}
||\cdot||_{k,\alpha}\leq C ||\cdot||_{\mathcal{C}^{k}_{b}}.
\end{equation}
We recall the following Sobolev embeddings (see \cite[Section 2.1.]{Meleard_97}):
\begin{enumerate}[(i)]
\item Sobolev embedding theorem: $\mathcal{W}^{m+k,\alpha}_{0}\hookrightarrow \mathcal{C}^{k,\alpha}$ for $m\geq 1$, $k\geq 0$ and $\alpha$ in $\mathbb{R}_{+}$, i.e. there exists a constant $C$ such that
\begin{equation}\label{eq:embedding:Ck:sobolev}
||f||_{\mathcal{C}^{k,\alpha}}\leq C||f||_{m+k,\alpha}.
\end{equation}
\item Maurin's theorem: $\mathcal{W}^{m+k,\alpha}_{0}\hookrightarrow_{H.S.} \mathcal{W}^{k,\alpha+\beta}_{0}$ for $m\geq 1$, $k\geq 0$, $\alpha$  in $\mathbb{R}_{+}$ and $\beta>1/2$, where $H.S.$ means that the embedding is of Hilbert-Schmidt type\footnote{Here, it means that $\sum_{j\geq 1} ||\varphi_{j}||^{2}_{k,\alpha+\beta}<+\infty$ if $(\varphi_{j})_{j\geq 1}$ is an orthonormal basis of $\mathcal{W}^{m+k,\alpha}_{0}$.}. In particular, the embedding is compact and there exists a constant $C$ such that
\begin{equation}\label{eq:embedding:sobolev:sobolev}
||f||_{k,\alpha+\beta} \leq C ||f||_{k+m,\alpha}.
\end{equation}
\end{enumerate}
Hence, the following dual embeddings hold true:
\begin{equation}\label{eq:embedding:dual}
\begin{cases}
\mathcal{W}^{-k,\alpha}_{0}\hookrightarrow \mathcal{C}^{-k}_{b}, \text{ for $k\geq 0$ and $\alpha>1/2$, (dual embedding of \eqref{eq:f:k:alpha:leq:f:Ckb})}\\
\mathcal{W}^{-k,\alpha+\beta}_{0}\hookrightarrow_{H.S.} \mathcal{W}^{-(m+k),\alpha}_{0}, \text{ for $m\geq 1$, $k\geq 0$, $\alpha$  in $\mathbb{R}_{+}$ and $\beta>1/2$.}
\end{cases}
\end{equation}

In some of the proofs given in the next section, we consider an orthonormal basis $(\varphi_{j})_{j\geq 1}$ of $\mathcal{W}^{k,\alpha}_{0}$ composed of $\mathcal{C}^{\infty}$ functions with compact support. The existence of such a basis follows from the fact that the functions of class $\mathcal{C}^{\infty}$ with compact support are dense in $\mathcal{W}^{k,\alpha}_{0}$. Furthermore, if $(\varphi_{j})_{j\geq 1}$ is an orthonormal basis of $\mathcal{W}^{k,\alpha}_{0}$ and $w$ belongs to $\mathcal{W}^{-k,\alpha}_{0}$, then $||w||^{2}_{-k,\alpha}=\sum_{j\geq 1} \left< w,\varphi_{j}\right>^{2}$ thanks to Parseval's identity. Let us precise that we stick with the notation $(\varphi_{j})_{j\geq 1}$ even if the space $\mathcal{W}^{k,\alpha}_{0}$ (in particular the regularity $k$) may differ from page to page.

The three lemmas below are useful throughout the analysis.

\begin{lem}\label{lem:control:for:linear:operator}
For every test function $\varphi$ in $\mathcal{W}^{2,\alpha}_{0}$, $||\varphi'||_{1,\alpha}\leq ||\varphi||_{2,\alpha}$.
If $f$ belongs to $\mathcal{C}^{k}_{b}$ for some $k\geq 1$ then, for any fixed $\alpha$ in $\mathbb{R}_{+}$, there exists a constant $C$ such that for every test function $\varphi$ in $\mathcal{W}^{k,\alpha}_{0}$, $||f\varphi||_{k,\alpha}\leq C ||f||_{\mathcal{C}^{k}_{b}} ||\varphi||_{k,\alpha}$.
\end{lem}
\begin{proof}
The first assertion follows from the definition of $|| \cdot ||_{2,\alpha}$, and the second one follows from Leibniz's rule and the definition of $|| \cdot ||_{k,\alpha}$.
\end{proof}

Let us denote $R$ (for \emph{reset}) the linear mapping defined by $R\varphi:=\varphi(0)-\varphi(\cdot)$ where $\varphi$ is some test function. This mapping naturally appears in our problem since the age process jumps to the value $0$ at each point of the underlying point process, as it appears below in Proposition \ref{prop:decomposition:eta}.

\begin{lem}\label{lem:R:continuous}
For any integer $k\geq 1$ and $\alpha>1/2$, the linear mapping $R$ is continuous from $\mathcal{W}^{k,\alpha}_{0}$ to itself.
\end{lem}
\begin{proof}
The function $R\varphi$ only differs from $\varphi$ by a constant so the derivatives of $R\varphi$ are equal to the derivatives of $\varphi$. Hence, using the convexity of the square function, we have
\begin{eqnarray*}
||R\varphi||^{2}_{k,\alpha} &\leq & \int_{\mathbb{R}_{+}} \frac{2|\varphi(0)|^{2}}{1+|x|^{2\alpha}} dx + \int_{\mathbb{R}_{+}} \frac{2|\varphi(x)|^{2}}{1+|x|^{2\alpha}} dx + \sum_{k'=1}^{k} \int_{\mathbb{R}_{+}} \frac{|\varphi^{(k')}(x)|^{2}}{1+|x|^{2\alpha}} dx \\
& \leq & 2 \int_{\mathbb{R}_{+}} \frac{1}{1+|x|^{2\alpha}} dx |\varphi(0)|^{2} + 2||\varphi ||^{2}_{k,\alpha}.
\end{eqnarray*}
Yet, $|\varphi(0)|\leq || \varphi ||_{\mathcal{C}^{0,\alpha}}\leq C ||\varphi||_{k,\alpha}$ by \eqref{eq:embedding:Ck:sobolev} and $\int_{\mathbb{R}_{+}} 1/(1+|x|^{2\alpha}) dx<+\infty $, for any fixed $\alpha>1/2$, so that  $||R\varphi||^{2}_{k,\alpha}\leq C|| \varphi ||^{2}_{k,\alpha}$.
\end{proof}

\begin{lem}\label{lem:control:normes:D:H}
For any fixed $\alpha$ in $\mathbb{R}_{+}$ and $x,y$ in $\mathbb{R}$, the mappings $\delta_{x}$ and $D_{x,y}:\mathcal{W}^{1,\alpha}_{0}\to \mathbb{R}$, defined by $\delta_{x}(\varphi):=\varphi(x)$ and $D_{x,y}(\varphi):=\varphi(x)-\varphi(y)$ are linear continuous.
In particular, for all $\alpha$ in $\mathbb{R}_{+}$, there exist some positive constants $C_{1}$ and $C_{2}$ such that, if $x$ and $y$ are bounded by some constant $M$, i.e. $|x|\leq M$ and $|y|\leq M$, then
\begin{equation}
\begin{cases}
|| \delta_{x} ||_{-2,\alpha}\leq || \delta_{x} ||_{-1,\alpha} \leq C_{1} (1+ M^{\alpha}),\\
|| D_{x,y} ||_{-2,\alpha}\leq || D_{x,y} ||_{-1,\alpha} \leq C_{2} (1+ M^{\alpha}).
\end{cases}
\end{equation}
\end{lem}
\begin{proof}
Remark that $|D_{x,y}(\varphi)|\leq |\varphi(x)|+|\varphi(y)|= |\delta_{x}(\varphi)| + |\delta_{y}(\varphi)|$. Hence, it suffices to show that there exists some positive constant $C$ such that
$
|| \delta_{x} ||_{-1,\alpha} \leq C (1+ |x|^{\alpha}).
$
Yet, $|\delta_{x}(\varphi)|=|\varphi(x)|\leq ||\varphi||_{\mathcal{C}^{0,\alpha}} (1+|x|^{\alpha})\leq C ||\varphi||_{1,\alpha} (1+|x|^{\alpha})$ by \eqref{eq:embedding:Ck:sobolev}.
\end{proof}

\begin{rem}
At this point, let us mention two reasons why weighted Sobolev spaces are more appropriate than standard (non-weighted) Sobolev spaces of functions on $\mathbb{R}_{+}$:
\begin{itemize}
\item we want to be able to consider functions of $\mathcal{C}^{k}_{b}$ as test functions: indeed, $\Psi$ must be considered as a test function, in Equation \eqref{eq:closed:equation:n:2} below for instance, yet we do not want $\Psi$ to be compactly supported with respect to the age $s$ or even to rapidly decrease when $s$ goes to infinity. The natural space to which $\Psi$ belongs is some $\mathcal{C}^{k}_{b}$ space,
\item in order to ensue criterion \hyperref[condition:A1']{$(\texttt{A}_{1'})$}, a compact embedding is required but Maurin's theorem does not apply for standard Sobolev spaces on $\mathbb{R}_{+}$ (see \cite[Theorem 6.37]{adams2003sobo}).
\end{itemize}
\end{rem}

In order to apply Lemma \ref{lem:R:continuous} and to satisfy the first point in the remark above, the weight $\alpha$ is assumed to be greater than $1/2$ in all the next sections so that \eqref{eq:f:k:alpha:leq:f:Ckb} holds true.

\subsection{Decomposition of the fluctuations}

Here, we give a semi-martingale representation of $\eta^{n}$ used to simplify the study of tightness (recall that $R$ is defined above in Lemma \ref{lem:R:continuous}).

\begin{prop}\label{prop:decomposition:eta}
Under Assumption (\hyperref[ass:for:LLN]{$\mathcal{A}_{\text{\tiny{LLN}}}$}), for every test function $\varphi$ in $\mathcal{C}^{1}_{b}$ and $t\geq 0$, 
\begin{equation}\label{eq:decomposition:eta}
\left< \eta_{t}^{n},\varphi\right> -\left< \eta_{0}^{n},\varphi\right> =\int_{0}^{t} \big( \left< \eta_{z}^{n},L_{z}\varphi\right> + A^{n}_{z}(\varphi) \big) dz+M_{t}^{n}(\varphi),
\end{equation}
with $L_{z}\varphi(s)=\varphi'(s)+ \Psi(s,\overline{\gamma}(z))R\varphi(s)$  for all $z\geq 0$ and $s$ in $\mathbb{R}$, where $\overline{\gamma}$ is defined by \eqref{eq:def:gamma:barre:t}, and
\begin{equation}\label{eq:def:M:n:A:n}
\begin{cases}
\displaystyle M_{t}^{n}(\varphi):= n^{-1/2} \sum_{i=1}^{n} \int_{0}^{t} R\varphi(S_{z-}^{n,i}) \left(N^{n,i}(dz)-\lambda_{z}^{n,i}dz\right),\\
\displaystyle A_{z}^{n}(\varphi):= n^{-1/2} \sum_{i=1}^{n} R\varphi(S_{z-}^{n,i}) \left(\lambda_{z}^{n,i} - \Psi(S^{n,i}_{z-},\overline{\gamma}(z) ) \right).
\end{cases}
\end{equation}
Furthermore, for any $\varphi$ in $\mathcal{C}^{1}_{b}$, $(M^{n}_{t}(\varphi))_{t\geq 0}$ is a real valued $\mathbb{F}$-martingale with angle bracket given by
\begin{equation}\label{eq:bracket:M:n:(phi)}
<M^{n}(\varphi)>_{t} = \frac{1}{n} \sum_{i=1}^{n} \int_{0}^{t} R\varphi\left( S^{n,i}_{z-} \right)^{2} \lambda^{n,i}_{z} dz.
\end{equation}
\end{prop}
\begin{rem}
To avoid confusion, let us mention that \eqref{eq:def:M:n:A:n} defines $M_{t}^{n}$ and $A_{z}^{n}$ as distributions acting on test functions. More precisely, we show below that they can be seen as distributions in $\mathcal{W}^{-2,\alpha}_{0}$ (Proposition \ref{prop:control:-1:alpha:combined}). However, we do not use the notation for the dual action $\left< \cdot,\cdot \right>$ to avoid tricky notation involving several angle brackets in \eqref{eq:bracket:M:n:(phi)} for instance.
\end{rem}

The proof of Proposition \ref{prop:decomposition:eta} relies on the integrability properties of the stochastic intensity and is given in Appendix \ref{sec:proof:prop:decomposition:eta}.

\subsection{Estimates in dual spaces}
\label{sec:estimates:dual:space}

Below are stated estimates of the terms $\eta^{n}$, $A^{n}$ and $M^{n}$ - appearing in \eqref{eq:decomposition:eta} - regarded as distributions.
More precisely, the estimates given in this section are stated in terms of the norm on either $\mathcal{W}^{-1,\alpha}_{0}$ or $\mathcal{W}^{-2,\alpha}_{0}$ for any $\alpha>1/2$ (in comparison with $\mathcal{W}^{-2,2}_{0}$ and $\mathcal{W}^{-4,1}_{0}$ in \cite{jourdain1998propagation} for instance). Usually, like in \cite{Meleard_97,jourdain1998propagation,luccon2015transition,Meleard_96}, the weight is linked to the maximal order of the moment estimates obtained on the positions of the particles. Here, the age processes are bounded in finite time horizon (remind \eqref{eq:control:age:uniform}) so the weight $\alpha$ of the Sobolev space can be taken as large as wanted. The weighted Sobolev spaces are nevertheless interesting here since, in particular, the distribution $\eta^{n}_{t}$ belongs to $\mathcal{W}^{-1,\alpha}_{0}$ for all $t\geq 0$ (see Proposition \ref{prop:control:-1:alpha:combined} below). We refer to the introductory discussion in Section \ref{sec:introduction} for complements on the usefulness of the weights.

We first give estimates in the smaller space $\mathcal{W}^{-1,\alpha}_{0}$. This is later used in order to prove tightness (remember condition \hyperref[condition:A1']{$(\texttt{A}_{1'})$} of the Aldous type criterion stated on page \pageref{condition:A1'}).

\begin{prop}
\label{prop:control:-1:alpha:combined}
Under Assumption (\hyperref[ass:for:LLN]{$\mathcal{A}_{\text{\tiny{LLN}}}$}), for any $\alpha>1/2$ and $\theta\geq 0$, the following statements hold true:
\begin{enumerate}[(i)]
\item \label{item:prop:i} the sequence $(\eta^{n})_{n\geq 1}$ is such that,
\begin{equation}\label{eq:control:eta:n:-1:alpha}
\sup_{n\geq 1}\sup_{t\in [0,\theta]} \mathbb{E}\left[ ||\eta^{n}_{t}||^{2}_{-1,\alpha} \right] <+\infty,
\end{equation}

\item \label{item:prop:ii} the process $(M^{n}_{t})_{t\geq 0}$, defined by \eqref{eq:def:M:n:A:n}, is an $\mathbb{F}$-martingale which belongs to $\mathcal{D}(\mathbb{R}_{+},\mathcal{W}^{-1,\alpha}_{0})$ almost surely. Furthermore, for any $\theta\geq 0$,
\begin{equation}\label{eq:control:M:n:-1:alpha}
\sup_{n\geq 1} \mathbb{E}\left[\sup_{t\in [0,\theta]} ||M^{n}_{t}||^{2}_{-1,\alpha} \right] <+\infty.
\end{equation}

\item \label{item:prop:iii} the sequence $(A^{n})_{n\geq 1}$, defined by \eqref{eq:def:M:n:A:n}, is such that,
\begin{equation}\label{eq:control:A:n:-2:alpha}
\sup_{n\geq 1} \sup_{t\in [0,\theta]} \mathbb{E}\left[ ||A^{n}_{t}||^{2}_{-2,\alpha} \right] <+\infty.
\end{equation}

\item \label{item:prop:iv} under (\hyperref[ass:Psi:dPsi:C2b]{$\mathcal{A}^{\Psi}_{s,\mathcal{C}^{2}_{b}}$}), for any $z$ in $\mathbb{R}_{+}$, the application $L_{z}$ defined in Proposition \ref{prop:decomposition:eta} is a linear continuous mapping from $\mathcal{W}^{2,\alpha}_{0}$ to $\mathcal{W}^{1,\alpha}_{0}$ and, for all $\varphi$ in $\mathcal{W}^{2,\alpha}_{0}$,
\begin{equation}\label{eq:control:L(f)}
\sup_{z\in [0,\theta]} \frac{|| L_{z} \varphi ||^{2}_{1,\alpha}}{|| \varphi ||^{2}_{2,\alpha}}<+\infty.
\end{equation}
\end{enumerate}
\end{prop}

The proof of Proposition \ref{prop:control:-1:alpha:combined} is given in Appendix \ref{sec:proof:prop:control:-1:alpha:combined} and mainly relies on the estimates given in Lemma  \ref{lem:control:normes:D:H}. However, let us mention that:
\begin{itemize}
\item the following expansion is used in the proof of $(iii)$ as well as in Section \ref{sec:limit:equation}: using that $\lambda^{n,i}_{t}=\Psi(S^{n,i}_{t-},\gamma^{n}_{t})$ and (\hyperref[ass:Psi:C2]{$\mathcal{A}^{\Psi}_{y,\mathcal{C}^{2}}$}), it follows from Taylor's inequality that for $\varphi$ in $\mathcal{W}^{2,\alpha}_{0}$,
\begin{equation}\label{eq:expansion:A:n:t}
A^{n}_{t}(\varphi)=\frac{1}{n} \sum_{i=1}^{n} R\varphi(S^{n,i}_{t-}) \frac{\partial \Psi}{\partial y} (S^{n,i}_{t-},\overline{\gamma}(t)) \left( \sqrt{n}( \gamma^{n}_{t}-\overline{\gamma}(t)) + \sqrt{n} r^{n,i}_{t}\right),
\end{equation}
with the rests satisfying $|r^{n,i}_{t}|\leq \sup_{s,y} |\frac{\partial^{2} \Psi}{\partial y^{2}}(s,y)| |\gamma^{n}_{t}-\overline{\gamma}(t)|^{2}/2$. This upper-bound does not depend on $\varphi$.
Let us denote $\Gamma^{n}_{t-}:=\sqrt{n}(\gamma^{n}_{t}-\overline{\gamma}(t))$ and 
$$R^{n,(1)}_{t}(\varphi):=\frac{1}{n} \sum_{i=1}^{n} \big( R\varphi(S^{n,i}_{t-}) \frac{\partial \Psi}{\partial y} (S^{n,i}_{t-},\overline{\gamma}(t)) \sqrt{n}r^{n,i}_{t} \big),$$
so that \eqref{eq:expansion:A:n:t} rewrites as
\begin{equation}\label{eq:An:function:of:Gamma}
A^{n}_{t}(\varphi)= \left< \overline{\mu}^{n}_{S_{t}},\frac{\partial \Psi}{\partial y}(\cdot,\overline{\gamma}(t))R\varphi \right> \Gamma^{n}_{t-} +R^{n,(1)}_{t}(\varphi).
\end{equation}
\item Lemma \ref{lem:control:for:linear:operator} and the following properties are used to prove point $(iv)$:
under Assumption (\hyperref[ass:Psi:dPsi:C2b]{$\mathcal{A}^{\Psi}_{s,\mathcal{C}^{2}_{b}}$}), the functions
\begin{equation}\label{eq:Psi:dPsi:C2b:locally:in:time}
t\mapsto ||\Psi(\cdot,\overline{\gamma}(t))||_{\mathcal{C}^{2}_{b}} \text{ and } t\mapsto \left\Vert \frac{\partial \Psi}{\partial y}(\cdot,\overline{\gamma}(t)) \right\Vert_{\mathcal{C}^{1}_{b}} \ \text{ are locally bounded,}
\end{equation}
since $t\mapsto \overline{\gamma}(t)$ is locally bounded. In the same way, under Assumption (\hyperref[ass:Psi:C4b]{$\mathcal{A}^{\Psi}_{s,\mathcal{C}^{4}_{b}}$}), the function
\begin{equation}\label{eq:Psi:C4b:locally:in:time}
t\mapsto ||\Psi(\cdot,\overline{\gamma}(t))||_{\mathcal{C}^{4}_{b}}  \ \text{ is locally bounded.}
\end{equation}
\end{itemize}

Proposition \ref{prop:control:-1:alpha:combined}, combined with the first line of Equation \eqref{eq:embedding:dual}, gives that $\eta^{n}$, $A^{n}$ and $M^{n}$ belong to $\mathcal{W}^{-2,\alpha}_{0}$. Hence, we may consider the following decomposition in $\mathcal{W}^{-2,\alpha}_{0}$,
\begin{equation}\label{eq:decomposition:eta:dual:space}
\eta_{t}^{n} - \eta_{0}^{n} =\int_{0}^{t} L_{z}^{*}\eta_{z}^{n}dz + \int_{0}^{t} A^{n}_{z} dz + M_{t}^{n},
\end{equation}
where $L_{z}^{*}$ is the adjoint operator of $L_{z}$. 

\begin{rem}
\label{rem:control:L(f):etoile}
As a corollary of Proposition \ref{prop:control:-1:alpha:combined}-$(iv)$, one has, for all $\alpha>1/2$, all $w$ in $\mathcal{W}^{-1,\alpha}_{0}$ and all $\theta\geq 0$,
\begin{equation}\label{eq:control:L(f):etoile}
\sup_{z\in [0,\theta]} \frac{|| L_{z}^{*} w ||^{2}_{-2,\alpha}}{|| w ||^{2}_{-1,\alpha}}<+\infty.
\end{equation}
Indeed, both $|| L_{z}^{*} w ||^{2}_{-2,\alpha}\leq \sup_{||\varphi||_{2,\alpha}=1} ||L_{z}\varphi||^{2}_{1,\alpha} ||w||^{2}_{-1,\alpha}$ and Equation \eqref{eq:control:L(f)} give the result.

Furthermore, the Doob-Meyer process \mbox{$(<\!\!<\! M^{n} \!>\!\!>_{t})_{t\geq 0}$} associated with the square integrable $\mathbb{F}$-martingale $(M^{n}_{t})_{t\geq 0}$ satisfies the following: for any $t\geq 0$, \mbox{$<\!\!<\! M^{n} \!>\!\!>_{t}$} is the linear continuous mapping from $\mathcal{W}^{2,\alpha}_{0}$ to $\mathcal{W}^{-2,\alpha}_{0}$ given, for all $\varphi_{1}$, $\varphi_{2}$ in $\mathcal{W}^{2,\alpha}_{0}$, by
\begin{equation*}
\left< {<\!\!<\! M^{n} \!>\!\!>}_{t} (\varphi_{1}),\varphi_{2} \right> = \frac{1}{n} \sum_{i=1}^{n} \int_{0}^{t} R\varphi_{1}(S^{n,i}_{z-}) R\varphi_{2}(S^{n,i}_{z-}) \lambda^{n,i}_{z} dz.
\end{equation*} 
This last equation can be retrieved thanks to the polarization identity from \eqref{eq:bracket:M:n:(phi)}.
\end{rem}

Yet, to give sense to Equation \eqref{eq:decomposition:eta:dual:space}, we need the lemma stated below.

\begin{lem}\label{lem:integrals:Bochner}
Under (\hyperref[ass:for:TGN]{$\mathcal{A}_{\text{\tiny{TGN}}}$}), the integrals $\int_{0}^{t} L_{z}^{*}\eta_{z}^{n}dz$ and $\int_{0}^{t} A^{n}_{z} dz$ are almost surely well defined as Bochner integrals in $\mathcal{W}^{-2,\alpha}_{0}$ for any $\alpha>1/2$. In particular, the functions $t\mapsto \int_{0}^{t} L_{z}^{*}\eta_{z}^{n}dz$ and $t\mapsto \int_{0}^{t} A^{n}_{z} dz$ are almost surely strongly continuous in $\mathcal{W}^{-2,\alpha}_{0}$.
\end{lem}
\begin{proof}
Since $\mathcal{W}^{-2,\alpha}_{0}$ is separable, it suffices to verify that (see Yosida \cite[p. 133]{yosida1980func}): 
\begin{enumerate}[(i)]
\item for every $\varphi$ in $\mathcal{W}^{2,\alpha}_{0}$, the functions $z\mapsto	 \left<L_{z}^{*}\eta_{z}^{n},\varphi\right>=\left<\eta_{z}^{n}, L_{z} \varphi\right>$ and $z\mapsto A^{n}_{z}(\varphi)$ are measurable,
\item the integrals $\int_{0}^{t} || L_{z}^{*}\eta_{z}^{n} ||_{-2,\alpha} dz$ and $\int_{0}^{t} || A^{n}_{z} ||_{-2,\alpha} dz$ are finite almost surely.
\end{enumerate}

The first condition is immediate. The second one follows from the controls we have shown. 

\begin{sloppypar}
Indeed, on the one hand, it follows from Equation \eqref{eq:control:L(f):etoile} that $\int_{0}^{t} || L_{z}^{*}\eta_{z}^{n} ||_{-2,\alpha} dz\lesssim_{t} \int_{0}^{t} ||\eta^{n}_{z}||_{-1,\alpha} dz$ and Proposition \ref{prop:control:-1:alpha:combined}-$(i)$ implies $\mathbb{E}[\int_{0}^{t} ||\eta^{n}_{z}||_{-1,\alpha+1} dz]<+\infty$   so that $\int_{0}^{t} || L_{z}^{*}\eta_{z}^{n} ||_{-2,\alpha} dz$ is finite a.s.
\end{sloppypar}

On the other hand, Proposition \ref{prop:control:-1:alpha:combined}-$(iii)$ gives that $\mathbb{E}[\int_{0}^{t} ||A^{n}_{z}||_{-2,\alpha} dz]$ is finite and so $\int_{0}^{t} ||A^{n}_{z}||_{-2,\alpha} dz$ is finite a.s.
\end{proof}

Now, using the decomposition \eqref{eq:decomposition:eta:dual:space} we are able to somehow exchange the expectation with the supremum in the control of $\eta$, i.e. Equation \eqref{eq:control:eta:n:-1:alpha}.

\begin{prop}\label{prop:control:eta:n:-2:alpha}
Under (\hyperref[ass:for:TGN]{$\mathcal{A}_{\text{\tiny{TGN}}}$}), for every $\alpha>1/2$ and $\theta\geq 0$,
\begin{equation}\label{eq:control:eta:n:-2:alpha}
\sup_{n\geq 1} \mathbb{E}\left[ \sup_{t\in [0,\theta]}  ||\eta^{n}_{t}||^{2}_{-2,\alpha} \right] <+\infty,
\end{equation}
and $t\mapsto \eta^{n}_{t}$ belongs to $\mathcal{D}(\mathbb{R}_{+},\mathcal{W}^{-2,\alpha}_{0})$ almost surely.
\end{prop}
\begin{proof}
Starting from \eqref{eq:decomposition:eta:dual:space}, we have by convexity of the square function
\begin{equation*}
\sup_{t\in [0,\theta]} || \eta^{n}_{t}||_{-2,\alpha}^{2} \leq 4 \big[ || \eta^{n}_{0}||_{-2,\alpha}^{2} + \theta \int_{0}^{\theta} (|| L_{z}^{*}\eta_{z}^{n} ||_{-2,\alpha}^{2} + || A^{n}_{z} ||_{-2,\alpha}^{2}) dz + \sup_{t\in [0,\theta]} || M^{n}_{t} ||_{-2,\alpha}^{2} \big].
\end{equation*}
\begin{sloppypar}
We deduce from Equation \eqref{eq:control:L(f)} that $\int_{0}^{\theta} \mathbb{E}[|| L_{z}^{*}\eta_{z}^{n} ||_{-2,\alpha}^{2}] dz\lesssim_{\theta} \sup_{z\in [0,\theta]} \mathbb{E}[||\eta_{z}^{n}||_{-1,\alpha}^{2}]$.
Hence, taking the expectation in both sides of the inequality above and applying Proposition \ref{prop:control:-1:alpha:combined} (remind \eqref{eq:embedding:dual}), we get \eqref{eq:control:eta:n:-2:alpha}. Starting from \eqref{eq:decomposition:eta:dual:space} and using that the integrals are continuous from Lemma \ref{lem:integrals:Bochner} and $M^{n}$ is càdlàg  from Proposition \ref{prop:control:-1:alpha:combined}-$(ii)$, it follows that $\eta^{n}$ is càdlàg.
\end{sloppypar}
\end{proof}

\subsection{Tightness result}
Using the estimates proved in Section \ref{sec:estimates:dual:space}, the tightness criterion stated on page \pageref{sec:tightness:criterion} can be checked.

\begin{thm}\label{thm:tightness:M:n:eta:n}
Under (\hyperref[ass:for:TGN]{$\mathcal{A}_{\text{\tiny{TGN}}}$}), for any $\alpha>1/2$, the sequences of the laws of $(M^{n})_{n\geq 1}$ and of $(\eta^{n})_{n\geq 1}$ are tight in the space $\mathcal{D}(\mathbb{R}_{+},\mathcal{W}^{-2,\alpha}_{0})$.
\end{thm}
\begin{proof}
Condition \hyperref[condition:A1']{$(\texttt{A}_{1'})$} with $H_{0}=\mathcal{W}^{-1,\alpha+1}_{0}$ and $H=\mathcal{W}^{-2,\alpha}_{0}$ is satisfied for both processes as a consequence of embedding \eqref{eq:embedding:dual} (remind that Hilbert-Schmidt operators are compact) and Proposition \ref{prop:control:-1:alpha:combined}.

On the one hand, condition \hyperref[condition:A2]{$(\texttt{A}_{2})$} holds for $(M^{n})_{n\geq 1}$ as soon as it holds for the trace of the processes \mbox{$(<\!\!<\! M^{n} \!>\!\!>)_{n\geq 1}$} given below \eqref{eq:decomposition:eta:dual:space} \cite[Rebolledo's theorem, p. 40]{joffe1986weak}. Let $(\varphi_{k})_{k\geq 1}$ be an orthonormal basis of $\mathcal{W}^{2,\alpha}_{0}$. Let $\theta\geq 0$, $\delta_{0}>0$ and $\delta\leq \delta_{0}$. Furthermore, let $\tau_{n}$ be an $\mathbb{F}$-stopping time smaller than $\theta$. 
\begin{multline*}
\left| \Tr {<\!\!<\! M^{n} \!>\!\!>}_{\tau_{n}+\delta}- \Tr {<\!\!<\! M^{n} \!>\!\!>}_{\tau_{n}} \right| \\ 
= \left| \sum_{k\geq 1} \left< {<\!\!<\! M^{n} \!>\!\!>}_{\tau_{n}+\delta}(\varphi_{k}),\varphi_{k} \right> - \left< {<\!\!<\! M^{n} \!>\!\!>}_{\tau_{n}}(\varphi_{k}),\varphi_{k} \right> \right|\\
\leq \sum_{k\geq 1} \frac{1}{n} \sum_{i=1}^{n} \int_{\tau_{n}}^{\tau_{n}+\delta} [R\varphi_{k}\left( S^{n,i}_{z-} \right)]^{2} \lambda^{n,i}_{z} dz 
\leq || \Psi ||_{\infty}  \frac{1}{n} \sum_{i=1}^{n} \int_{\tau_{n}}^{\tau_{n}+\delta} \sum_{k\geq 1} R\varphi_{k}\left( S^{n,i}_{z-} \right)^{2}  dz.
\end{multline*}
Noticing that $R\varphi_{k} (S^{n,i}_{z-})=D_{0,S^{n,i}_{z-}}(\varphi_{k})$ and then using Lemma \ref{lem:control:normes:D:H} and the fact that the ages $S^{n,i}_{z-}$ are upper bounded by $M_{S_{0}}+z+\leq M_{S_{0}}+\theta+\delta_{0}$ (thanks to (\hyperref[ass:initial:condition:density:compact:support]{$\mathcal{A}^{u_0}_{\infty}$}), remind \eqref{eq:control:age:uniform}), it follows that
\begin{equation*}
\mathbb{E}\left[ \left| \Tr {<\!\!<\! M^{n} \!>\!\!>}_{\tau_{n}+\delta}- \Tr {<\!\!<\! M^{n} \!>\!\!>}_{\tau_{n}} \right| \right] \leq \delta_{0} ||\Psi||_{\infty} (C_{2})^{2} \left( 1+(M_{S_{0}}+\theta+\delta_{0})^{\alpha} \right)^{2}.
\end{equation*}
This last bound is arbitrarily small for $\delta_{0}$ small enough which gives condition \hyperref[condition:A2]{$(\texttt{A}_{2})$} thanks to Markov's inequality.

On the other hand, using decomposition \eqref{eq:decomposition:eta:dual:space} and the fact that $(M^{n})_{n\geq 1}$ is tight, it suffices to show the tightness of the remaining terms $(R^{n}_{t}=\eta^{n}_{0} + \int_{0}^{t} L_{z}^{*} \eta^{n}_{z} dz + \int_{0}^{t} A^{n}_{z} dz)_{n\geq 1}$ in order to show tightness of $(\eta^{n})_{n\geq 1}$. Yet, using  Equation \eqref{eq:control:L(f):etoile}, we have
\begin{multline*}
|| R^{n}_{\tau_{n}+\delta} - R^{n}_{\tau_{n}}||_{-2,\alpha}^{2} = \left\Vert \int_{\tau_{n}}^{\tau_{n}+\delta} L_{z}^{*} \eta^{n}_{z} + A^{n}_{z} dz \right\Vert_{-2,\alpha}^{2}\\
\leq 2 \delta \int_{\tau_{n}}^{\tau_{n}+\delta} (|| L_{z}^{*} \eta^{n}_{z}||_{-2,\alpha}^{2} + || A^{n}_{z} ||_{-2,\alpha}^{2}) dz \leq 2\delta_{0} \int_{0}^{\theta+\delta_{0}} (C||\eta^{n}_{z}||_{-1,\alpha+1}^{2} + || A^{n}_{z} ||_{-2,\alpha}^{2}) dz,
\end{multline*}
\begin{sloppypar}
\noindent
where $C$ depends on $\theta$ and $\delta_{0}$. Then, Proposition \ref{prop:control:-1:alpha:combined} implies that $\sup_{n\geq 1} \mathbb{E}[|| R^{n}_{\tau_{n}+\delta} - R^{n}_{\tau_{n}}||_{-2,\alpha}^{2}]\leq C \delta_{0}$ for $\delta_{0}$ small enough. Finally, Markov's inequality gives condition \hyperref[condition:A2]{$(\texttt{A}_{2})$} for $(R^{n})_{n\geq 1}$ and so the tightness of $(\eta^{n})_{n\geq 1}$.
\end{sloppypar}
\end{proof}


\begin{rem}
\label{rem:C:tight}
For any $\alpha>1/2$, every limit (with respect to the convergence in law) $M$ (respectively $\eta$) in $\mathcal{D}(\mathbb{R}_{+},\mathcal{W}^{-2,\alpha}_{0})$ of the sequence $(M^{n})_{n\geq 1}$ (resp. $(\eta^{n})_{n\geq 1}$) satisfies
\begin{equation}\label{eq:C:tight}
 \mathbb{E}\left[\sup_{t\in [0,\theta]} ||M_{t}||^{2}_{-2,\alpha} \right] <+\infty \quad \bigg(\text{resp. } \mathbb{E}\left[ \sup_{t\in [0,\theta]}  ||\eta_{t}||^{2}_{-2,\alpha} \right] <+\infty \bigg).
\end{equation}
Moreover, the limit laws are supported in $\mathcal{C}(\mathbb{R}_{+},\mathcal{W}^{-2,\alpha}_{0})$.
\end{rem}
\begin{proof}
Let us first show that the limit points are continuous. According to \cite[Theorem 13.4.]{Billing_Convergence}, it suffices to prove that for all $\theta\geq 0$, the maximal jump size of $M^{n}$ and $\eta^{n}$ on $[0,\theta]$ converge to $0$ almost surely in order to prove the last point. Yet, for all $\varphi$ in $\mathcal{W}^{2,\alpha}_{0}$,
\begin{equation*}
\Delta M^{n}_{t}(\varphi):= |M^{n}_{t}(\varphi) - M^{n}_{t-}(\varphi)|= \frac{1}{\sqrt{n}} \sum_{i=1}^{n} D_{0,S^{n,i}_{t-}}(\varphi) \mathds{1}_{t\in N^{n,i}},
\end{equation*}
where we use the definition of $M^{n}_{t}(\varphi)$ given by \eqref{eq:def:M:n:A:n} for $\varphi$ in $\mathcal{C}^{1}_{b}$ and a density argument to extend it to $\varphi$ in $\mathcal{W}^{2,\alpha}_{0}$, and
\begin{equation*}
\left<\Delta\eta^{n}_{t},\varphi\right>:= |\left<\eta^{n}_{t},\varphi\right> - \left<\eta^{n}_{t-},\varphi\right>| =\frac{1}{\sqrt{n}} \sum_{i=1}^{n} D_{0,S^{n,i}_{t-}}(\varphi) \mathds{1}_{t\in N^{n,i}}
\end{equation*}
where we used the fact that $(u_{t})_{t\geq 0}$ is continuous in $\mathcal{W}^{-2,\alpha}_{0}$ (see Lemma \ref{lem:continuity:Pt}). Since almost surely there is no common point to any two of the point processes $(N^{n,i})_{i=1,\dots ,n}$, there is, almost surely, for all $t\geq 0$, at most one of the $\mathds{1}_{t\in N^{n,i}}$ which is non null. Then, Lemma \ref{lem:control:normes:D:H} implies
\begin{equation*}
\begin{cases}
\sup_{t\in[0,\theta]} || \Delta M^{n}_{t} ||_{-2,\alpha} \leq \frac{1}{\sqrt{n}} C_{2}(1+(M_{S_{0}}+\theta)^{\alpha}),\\
\sup_{t\in[0,\theta]} || \Delta \eta^{n}_{t} ||_{-2,\alpha} \leq \frac{1}{\sqrt{n}} C_{2}(1+(M_{S_{0}}+\theta)^{\alpha}),
\end{cases}
\end{equation*}
which gives the desired convergence to $0$.

Finally, the two statements of Equation \eqref{eq:C:tight} are consequences of Propositions \ref{prop:control:-1:alpha:combined}-$(ii)$ (remind \eqref{eq:embedding:dual}) and \ref{prop:control:eta:n:-2:alpha} where we use the previous step and the fact that the mapping $g\mapsto \sup_{t\in [0,\theta]} ||g_{t}||^{2}_{-2,\alpha}$ from $\mathcal{D}(\mathbb{R}_{+},\mathcal{W}^{-2,\alpha}_{0})$ to $\mathbb{R}$ is continuous at every point $g^{0}$ in $\mathcal{C}(\mathbb{R}_{+},\mathcal{W}^{-2,\alpha}_{0})$.
\end{proof}

\section{Characterization of the limit}

The aim of this section is to prove convergence of the sequence $(\eta^{n})_{n\geq 1}$ by identifying the limit fluctuation process $\eta$ as the unique solution of a SDE in infinite dimension. We first prove, in Section \ref{sec:limit:equation}, that every possible limit process $\eta$ satisfies a certain SDE (Theorem \ref{thm:limit:equation}). Then, we show, in Section \ref{sec:uniqueness:limit:equation}, that this SDE uniquely characterizes the limit law, which completes the proof of the convergence in law of $(\eta^{n})_{n\geq 1}$ to $\eta$.

\subsection{Candidate for the limit equation}
\label{sec:limit:equation}

In this section, the limit version of Equation \eqref{eq:decomposition:eta:dual:space} is stated. Apart from $\eta^{n}$, there are two random processes in \eqref{eq:decomposition:eta:dual:space} that are $A^{n}$ and $M^{n}$. The following notation encompasses the source of the stochasticity of both $A^{n}$ and $M^{n}$ and is mainly used in order to track the correlations between those two quantities: for all $n\geq 1$, let $W^{n}$ be the $\mathcal{W}^{-1,\alpha}_{0}$-valued martingale defined, for all $t\geq 0$ and $\varphi$ in $\mathcal{W}^{1,\alpha}_{0}$, by
\begin{equation*}
W^{n}_{t}(\varphi):=\frac{1}{\sqrt{n}} \sum_{i=1}^{n} \int_{0}^{t} \varphi(S^{n,i}_{z-}) (N^{n,i}(dz)-\lambda^{n,i}_{z}dz).
\end{equation*}
\begin{sloppypar}
Notice that $M^{n}_{t}(\varphi)=W^{n}_{t}(R\varphi)$. Furthermore, as for $M^{n}$, the Doob-Meyer process \mbox{$(<\!\!<\!W^{n}\!>\!\!>_{t})_{t\geq 0}$} associated with $(W^{n}_{t})_{t\geq 0}$ satisfies the following: for any $t\geq 0$, \mbox{$<\!\!<\!W^{n}\!>\!\!>_{t}$} is the linear continuous mapping from $\mathcal{W}^{2,\alpha}_{0}$ to $\mathcal{W}^{-2,\alpha}_{0}$ given, for all $\varphi_{1}$ and $\varphi_{2}$ in $\mathcal{W}^{2,\alpha}_{0}$, by
\end{sloppypar}
\begin{equation}\label{eq:bracket:W:n}
\left< {<\!\!<\!W^{n}\!>\!\!>}_{t} (\varphi_{1}), \varphi_{2} \right> = \frac{1}{n} \sum_{i=1}^{n} \int_{0}^{t} \varphi_{1}(S^{n,i}_{z-}) \varphi_{2}(S^{n,i}_{z-}) \lambda^{n,i}_{z} dz.
\end{equation}
All the results given for $M^{n}$ in the previous section can be extended to $W^{n}$. In particular, 
\begin{equation}\label{eq:W:n:tight}
\mbox{the sequence $(W^{n})_{n\geq 1}$ is tight in $\mathcal{D}(\mathbb{R}_{+},\mathcal{W}^{-2,\alpha}_{0})$.}
\end{equation}
Next, we prove that it converges towards the Gaussian process $W$ defined below.

\begin{defn}\label{def:gaussian:process}
For any $\alpha>1/2$, let $W$ be a continuous centred Gaussian process with values in $\mathcal{W}^{-2,\alpha}_{0}$ with covariance given, for all $\varphi_{1}$ and $\varphi_{2}$ in $\mathcal{W}^{2,\alpha}_{0}$, for all $t$ and $t'\geq 0$, by
\begin{eqnarray}
\mathbb{E}\left[ W_{t}(\varphi_{1})W_{t'}(\varphi_{2}) \right] &=& \int_{0}^{t\wedge t'} \left<u_{z},\varphi_{1}\varphi_{2}\Psi(\cdot,\overline{\gamma}(z))\right> dz \nonumber\\
&=& \int_{0}^{t\wedge t'} \int_{0}^{+\infty} \varphi_{1}(s)\varphi_{2}(s)\Psi(s,\overline{\gamma}(z)) u(z,s) ds dz,\label{eq:covariance:W}
\end{eqnarray}
where $u$ is the unique solution of \eqref{eq:edp:PPS}.
\end{defn}

\begin{rem}
We refer to the PhD manuscript of the author \cite{chevallier2016phd} for the existence and uniqueness in law of such a process $W$. Yet, let us mention here that the process $W$ defined above does not depend on the weight $\alpha$ in the sense that the definition is consistent with respect to the weights. Indeed, say $W^{\alpha}$ and $W^{\beta}$ are two processes is the sense of Definition \ref{def:gaussian:process} with values in $\mathcal{W}^{-2,\alpha}_{0}$ and $\mathcal{W}^{-2,\beta}_{0}$ respectively. Assume for instance that $\beta>\alpha$. Then, $W^{\beta}$ can be seen as a process with values in $\mathcal{W}^{-2,\alpha}_{0}$ via the canonical embedding $\mathcal{W}^{-2,\beta}_{0}\hookrightarrow \mathcal{W}^{-2,\alpha}_{0}$. Yet, the covariance structure \eqref{eq:covariance:W} does not depend on the weights $\alpha$ and $\beta$ so $W^{\beta}$ is also a Gaussian process with values in $\mathcal{W}^{-2,\alpha}_{0}$ with the prescribed covariance and the uniqueness in law guaranties the equality of the laws of $W^{\alpha}$ and $W^{\beta}$ as $\mathcal{C}(\mathbb{R}_{+},\mathcal{W}^{-2,\alpha}_{0})$-valued random variables.
\end{rem}

\begin{prop}\label{prop:convergence:Wn:W}
Under (\hyperref[ass:for:TGN]{$\mathcal{A}_{\text{\tiny{TGN}}}$}), for any $\alpha>1/2$, the sequence $(W^{n})_{n\geq 1}$ of processes in $\mathcal{D}(\mathbb{R}_{+},\mathcal{W}^{-2,\alpha}_{0})$ converges in law to $W$.
\end{prop}
The proof of Proposition \ref{prop:convergence:Wn:W} is given in Appendix \ref{sec:proof:prop:convergence:Wn:W}. It relies on the convergence of the bracket \eqref{eq:bracket:W:n} towards the covariance \eqref{eq:covariance:W} and an application of Rebolledo's central limit theorem (the maximum size of the jumps is bounded up to a constant by $n^{-1/2}$ and so goes to $0$).

Denote by ${\bf 1}:\mathbb{R}_{+}\to \mathbb{R}$ the constant function equal to $1$ (which belongs to $\mathcal{W}^{2,\alpha}_{0}$ since we assume $\alpha>1/2$) and note that $W^{n}_{t}({\bf 1})$ is the rescaled canonical martingale associated with the system of age-dependent Hawkes processes, namely
\begin{equation*}
W^{n}_{t}({\bf 1}) = \sqrt{n} \left( \frac{1}{n} \sum_{i=1}^{n} N^{n,i}_{t} - \int_{0}^{t} \lambda^{n,i}_{z} dz \right).
\end{equation*}
Now, let us expand the decomposition \eqref{eq:decomposition:eta:dual:space} in order to get a closed equation. Let us recall the expansion of $A^{n}$ given by \eqref{eq:An:function:of:Gamma}, that is
\begin{equation*}
A^{n}_{t}(\varphi)= \left< \overline{\mu}^{n}_{S_{t}},\frac{\partial \Psi}{\partial y}(\cdot,\overline{\gamma}(t))R\varphi \right> \Gamma^{n}_{t-} +R^{n,(1)}_{t}(\varphi),
\end{equation*}
with $\Gamma^{n}_{t-}=\sqrt{n} (\gamma^{n}_{t}-\overline{\gamma}(t))$ and the rest term:
$$R^{n,(1)}_{t}(\varphi):=\frac{1}{n} \sum_{i=1}^{n} \big( R\varphi(S^{n,i}_{t-}) \frac{\partial \Psi}{\partial y} (S^{n,i}_{t-},\overline{\gamma}(t)) \sqrt{n}r^{n,i}_{t}\big).$$
Below, we use the fact that this rest term converges to $0$ in $L^{1}$ norm: indeed, recall that 
\begin{equation}\label{eq:bound:rni:Gamma}
|r^{n,i}_{t}|\lesssim  |\gamma^{n}_{t}-\overline{\gamma}(t)|^{2}
\end{equation}
and, thanks to Proposition \ref{prop:rate:of:convergence:proba:k:tuple:different:ages},
\begin{equation*}
\mathbb{E}\left[ |\gamma^{n}_{t}-\overline{\gamma}(t)|^{2} \right]\lesssim_{t} n^{-1}.
\end{equation*}
Since $\Gamma^{n}_{t-}$ (as part of $A^{n}_{t}(\varphi)$) only appears in \eqref{eq:decomposition:eta:dual:space} as an integrand and is only discontinuous on a set of Lebesgue measure equal to zero, we can replace it by its càdlàg version denoted by $\Gamma^{n}_{t}$. Let us consider the decomposition $\Gamma^{n}_{t}= \Upsilon^{1}_{t}+\Upsilon^{2}_{t}+\Upsilon^{3}_{t}$, with
\begin{equation*}
\begin{cases}
\displaystyle \Upsilon^{1}_{t}:=\sqrt{n} \int_{0}^{t} h(t-z) \left(\frac{1}{n}\sum_{i=1}^{n} N^{n,i}(dz) - \lambda^{n,i}_{z}dz\right) = \int_{0}^{t} h(t-z) dW^{n}_{z}({\bf 1}),\\
\displaystyle \Upsilon^{2}_{t}:=\sqrt{n} \int_{0}^{t} h(t-z) \frac{1}{n}\sum_{i=1}^{n} (\lambda^{n,i}_{z}- \Psi(S^{n,i}_{z-},\overline{\gamma}(z)))dz,\\
\displaystyle \Upsilon^{3}_{t}:= \sqrt{n} \int_{0}^{t} h(t-z) \frac{1}{n}\sum_{i=1}^{n} (\Psi(S^{n,i}_{z-},\overline{\gamma}(z))- \overline{\lambda}(z))dz= \!\! \int_{0}^{t} \! h(t-z) \left< \eta^{n}_{z}, \Psi(\cdot,\overline{\gamma}(z)\right>dz,
\end{cases}
\end{equation*}
where we used, in the last line, the fact that $\overline{\mu}^{n}_{S_{z-}}=\overline{\mu}^{n}_{S_{z}}$ for almost every $z$ in $\mathbb{R}_{+}$, and $\overline{\lambda}(z)=\left< u_{z},\Psi(\cdot,\overline{\gamma}(z)) \right>$.

Based on Assumption (\hyperref[ass:Psi:C2]{$\mathcal{A}^{\Psi}_{y,\mathcal{C}^{2}}$}), as for Equation \eqref{eq:expansion:A:n:t}, one can give the Taylor expansion of the term
\begin{equation*}
\Upsilon^{2}_{t}=\sqrt{n} \int_{0}^{t} h(t-z) \frac{1}{n}\sum_{i=1}^{n} (\Psi(S^{n,i}_{z-},\gamma^{n}_{z})- \Psi(S^{n,i}_{z-},\overline{\gamma}(z)))dz.
\end{equation*}

On the one hand, gathering the decomposition \eqref{eq:decomposition:eta} with \eqref{eq:An:function:of:Gamma} and on the other hand gathering $\Gamma^{n}_{t}= \Upsilon^{1}_{t}+\Upsilon^{2}_{t}+\Upsilon^{3}_{t}$ with the Taylor expansion of $\Upsilon^{2}_{t}$ give that $(\eta^{n},\Gamma^{n})$ satisfies the following closed system for all $\varphi$ in $\mathcal{W}^{2,\alpha}_{0}$,
\begin{multline}\label{eq:closed:equation:n:1}
\left< \eta^{n}_{t},\varphi \right>-\left< \eta^{n}_{0},\varphi \right> = \int_{0}^{t} \left< \eta^{n}_{z},L_{z}\varphi \right> dz + \int_{0}^{t} \left< \overline{\mu}^{n}_{S_{z}},\frac{\partial \Psi}{\partial y}(\cdot,\overline{\gamma}(z))R\varphi \right> \Gamma^{n}_{z}dz \\
+\int_{0}^{t} R^{n,(1)}_{z}(\varphi) dz +  W^{n}_{t}(R\varphi),
\end{multline}
\begin{multline}\label{eq:closed:equation:n:2}
\Gamma^{n}_{t} = \int_{0}^{t} h(t-z) \left< \overline{\mu}^{n}_{S_{z}},\frac{\partial \Psi}{\partial y}(\cdot,\overline{\gamma}(z))\right> \Gamma^{n}_{z}dz + \int_{0}^{t} h(t-z) R^{n,(2)}_{z} dz \\
+ \int_{0}^{t} h(t-z) \left< \eta^{n}_{z}, \Psi(\cdot,\overline{\gamma}(z)\right> dz + \int_{0}^{t} h(t-z) dW^{n}_{z}({\bf 1}) ,
\end{multline}
where the rest term $R^{n,(2)}_{z}$ is defined by
\begin{equation*}
R^{n,(2)}_{z}:= \frac{1}{\sqrt{n}} \sum_{i=1}^{n} \frac{\partial \Psi}{\partial y} (S^{n,i}_{t-},\overline{\gamma}(t)) r^{n,i}_{t}.
\end{equation*}
Once again, notice that $\Gamma^{n}_{z-}$, which naturally appears in the first integral term of \eqref{eq:closed:equation:n:2}, is replaced by its càdlàg version $\Gamma^{n}_{z}$ since they are equal except on a null measure set.\\

Let us denote $V^{n}_{t}:=\int_{0}^{t}h(t-z) dW^{n}_{z}({\bf 1})$ and $V_{t}:=\int_{0}^{t}h(t-z)dW_{z}({\bf 1})$. The convergence of the sources of stochasticity in the system \eqref{eq:closed:equation:n:1}-\eqref{eq:closed:equation:n:2} is stated in the following corollary of Proposition \ref{prop:convergence:Wn:W}.

\begin{cor}\label{cor:convergence:couple}
Under (\hyperref[ass:for:TGN]{$\mathcal{A}_{\text{\tiny{TGN}}}$}) and (\hyperref[ass:h:Holder]{$\mathcal{A}^{h}_{\rm H\ddot{o}l}$}), the following convergence in law holds true in $\mathcal{D}(\mathbb{R}_{+},\mathcal{W}^{-2,\alpha}_{0}\times \mathbb{R})$,
\begin{equation*}
\Big(R^{*}W^{n}_{t},V^{n}_{t}\Big)_{t\geq 0} \Rightarrow \Big(R^{*}W_{t},V_{t}\Big)_{t\geq 0},
\end{equation*}
where $R^{*}$ denotes the adjoint of $R$.
\end{cor}
The proof of Corollary \ref{cor:convergence:couple} uses Billingsley tightness criterion for real-valued stochastic processes and is given in Appendix \ref{sec:proof:cor:convergence:couple}.\\

Before taking the limit $n\to +\infty$ in the system \eqref{eq:closed:equation:n:1}-\eqref{eq:closed:equation:n:2}, we state the tightness of $(\Gamma^{n})_{n\geq 1}$. Nevertheless, let us first mention that we use the following estimates: as a consequence of Proposition \ref{prop:rate:of:convergence:proba:k:tuple:different:ages}, for all $k\geq 0$ and $\theta\geq 0$,
\begin{equation}\label{eq:power:Gamma:n:locally:bounded}
\sup_{t\in [0,\theta]} \mathbb{E}\left[ |\Gamma^{n}_{t}|^{k} \right] <+\infty,
\end{equation}
since $\sup_{t\in [0,\theta]} \mathbb{E}\left[ |\Gamma^{n}_{t}|^{k} \right]=\sup_{t\in [0,\theta]} \mathbb{E}\left[ |\Gamma^{n}_{t-}|^{k} \right]$ because the underlying point processes admit intensities so that there is almost surely no jump at time $\theta$.

\begin{prop}\label{prop:tightness:Gamma:n}
Under (\hyperref[ass:for:TGN]{$\mathcal{A}_{\text{\tiny{TGN}}}$}) and (\hyperref[ass:h:Holder]{$\mathcal{A}^{h}_{\rm H\ddot{o}l}$}), the sequence of the laws of $(\Gamma^{n})_{n\geq 1}$ is tight in $\mathcal{D}(\mathbb{R}_{+},\mathbb{R})$. Furthermore, the possible limit laws are supported in $\mathcal{C}(\mathbb{R}_{+},\mathbb{R})$ and satisfy, for all $k\geq 0$,
\begin{equation}\label{eq:power:Gamma:locally:bounded:limit:version}
\sup_{t\in [0,\theta]} \mathbb{E}\left[ |\Gamma_{t}|^{k} \right] <+\infty,
\end{equation}
\end{prop}
The proof of Proposition \ref{prop:tightness:Gamma:n} uses Aldous tightness criterion for real-valued stochastic processes and is given in Appendix \ref{sec:proof:prop:tightness:Gamma:n}.\\

Both sequences $(\eta^{n})_{n\geq 1}$ and $(\Gamma^{n})_{n\geq 1}$ are tight with continuous limit trajectories. Tightness of $(\eta^{n},\Gamma^{n})_{n\geq 1}$ hence follows and we are now in position to give the system satisfied by any limit $(\eta,\Gamma)$.

\begin{thm}\label{thm:limit:equation}
Under (\hyperref[ass:for:TGN]{$\mathcal{A}_{\text{\tiny{TGN}}}$}) and (\hyperref[ass:h:Holder]{$\mathcal{A}^{h}_{\rm H\ddot{o}l}$}), for all $\alpha> 1/2$, any limit  $(\eta,\Gamma)$ of the sequence $(\eta^{n},\Gamma^{n})_{n\geq 1}$ is a solution in $\mathcal{C}(\mathbb{R}_{+},\mathcal{W}^{-2,\alpha}_{0}\times \mathbb{R})$ of the following system (formulated in $\mathcal{W}^{-3,\alpha}_{0}\times \mathbb{R}$),
\begin{multline}\label{eq:closed:equation:limit:1}
\forall \varphi\in \mathcal{W}^{3,\alpha}_{0}, \quad \left< \eta_{t},\varphi \right>-\left< \eta_{0},\varphi \right> = \int_{0}^{t} \left< \eta_{z},L_{z}\varphi \right> dz + \int_{0}^{t} \left< u_{z},\frac{\partial \Psi}{\partial y}(\cdot,\overline{\gamma}(z))R\varphi \right> \Gamma_{z} dz \\
+  W_{t}(R\varphi),
\end{multline}
\begin{multline}\label{eq:closed:equation:limit:2}
\Gamma_{t} = \int_{0}^{t} h(t-z) \left< \eta_{z}, \Psi(\cdot,\overline{\gamma}(z)\right> dz + \int_{0}^{t} h(t-z) \left< u_{z},\frac{\partial \Psi}{\partial y}(\cdot,\overline{\gamma}(z))\right> \Gamma_{z} dz\\
+ \int_{0}^{t} h(t-z) dW_{z}({\bf 1}).
\end{multline}
\end{thm}

The proof of Theorem \ref{thm:limit:equation} consists in proving continuity properties to apply the continuous mapping theorem. It is given in Appendix \ref{sec:proof:thm:limit:equation}.

\begin{rem}
The linear operator $L_{z}$ appearing in \eqref{eq:closed:equation:n:1} and \eqref{eq:closed:equation:limit:1} reduces the regularity of the test functions by $1$. Hence, if we consider Equation \eqref{eq:closed:equation:n:1} for test functions $\varphi$ in $\mathcal{W}^{2,\alpha}_{0}$ then we must consider $\eta^{n}$ as taking values in $\mathcal{W}^{-1,\alpha}_{0}$ when dealing with the integral term $\int_{0}^{t} \left< \eta^{n}_{z},L_{z}\varphi \right> dz$. Yet $(\eta^{n})_{n\geq 1}$ is not tight in this space. Thus we consider \eqref{eq:closed:equation:n:1} for test functions in $\mathcal{W}^{3,\alpha}_{0}$ so that every term is tight. That is why the limit equation \eqref{eq:closed:equation:limit:1} is formulated in $\mathcal{W}^{-3,\alpha}_{0}$. However, the limit process $\eta$ takes values in the smaller space $\mathcal{W}^{-2,\alpha}_{0}$.
\end{rem}

\begin{rem}
The initial condition $\eta_{0}$ of the system \eqref{eq:closed:equation:limit:1}-\eqref{eq:closed:equation:limit:2} is determined by the initial density $u_0$. It is an infinite dimensional gaussian random variable. Indeed, $\eta_{0}$ is well defined as the limit in $\mathcal{W}^{-3,\alpha}_{0}$ of $\eta^{n}_{0}$. The sequence $(\eta^{n}_{0})_{n\geq 1}$ is tight in $\mathcal{W}^{-3,\alpha}_{0}$ (it is tight in $\mathcal{W}^{-2,\alpha}_{0}$ and there is a continuous embedding of $\mathcal{W}^{-2,\alpha}_{0}$ into $\mathcal{W}^{-3,\alpha}_{0}$) and for any $\varphi$ in $\mathcal{W}^{3,\alpha}_{0}$, we have the convergence of the real-valued random variables $\left< \eta^{n}_{0},\varphi \right>=\sqrt{n} \left< \overline{\mu}^{n}_{S_{0}}-u_{0}, \varphi \right>$ by applying the standard central limit theorem since the initial conditions are i.i.d.
\end{rem}

\subsection{Uniqueness of the limit law}
\label{sec:uniqueness:limit:equation}

The next step in order to prove convergence of the sequence $(\eta^{n},\Gamma^{n})_{n\geq 1}$ is to prove uniqueness of the solutions of the limit system  \eqref{eq:closed:equation:limit:1}-\eqref{eq:closed:equation:limit:2}. Since the system is linear, the standard argument is to consider the system satisfied by the difference between two solutions and show that its unique solution is trivial. Let $(\eta,\Gamma)$ and $(\hat{\eta},\hat{\Gamma})$ be two solutions associated with the same ``noise'' $W$ and the same initial condition $\eta_{0}$. Denote by $\tilde{\eta}:=\eta-\hat{\eta}$ and $\tilde{\Gamma}:=\Gamma-\hat{\Gamma}$ the differences. Then, $(\tilde{\eta},\tilde{\Gamma})$ is a solution of the following system
\begin{equation}\label{eq:closed:equation:tilde:1}
\forall \varphi\in \mathcal{W}^{3,\alpha}_{0}, \quad \left< \tilde{\eta}_{t},\varphi \right>-\int_{0}^{t} \left< \tilde{\eta}_{z},L_{z}\varphi \right> dz - \int_{0}^{t} \left< u_{z},\frac{\partial \Psi}{\partial y}(\cdot,\overline{\gamma}(z))R\varphi \right> \tilde{\Gamma}_{z} dz =0,
\end{equation}
\begin{equation}\label{eq:closed:equation:tilde:2}
\tilde{\Gamma}_{t} - \int_{0}^{t} h(t-z) \left< \tilde{\eta}_{z}, \Psi(\cdot,\overline{\gamma}(z)\right> dz - \int_{0}^{t} h(t-z) \left< u_{z},\frac{\partial \Psi}{\partial y}(\cdot,\overline{\gamma}(z))\right> \tilde{\Gamma}_{z} dz=0.
\end{equation}
The standard follow-up is to use Grönwall's lemma. Let us show here why it is not sufficient in our case. For instance, assume we want to prove that $||\tilde{\eta}||_{-3,\alpha}=0$: heuristically, when applied to \eqref{eq:closed:equation:tilde:2}, Grönwall's argument gives that $|\tilde{\Gamma}_{t}|$ is bounded by some locally bounded function of $t$ times the integral $\int_{0}^{t} ||\tilde{\eta}_{z}||_{-3,\alpha} dz$. However, even if we use this bound for $\tilde{\Gamma}$ in \eqref{eq:closed:equation:tilde:1}, Grönwall's argument cannot be applied since the term $\int_{0}^{t} \left< \tilde{\eta}_{z},L_{z}\varphi \right> dz$ involves $||\tilde{\eta}_{z}||_{-2,\alpha}$ which is greater than the desired norm $||\tilde{\eta}_{z}||_{-3,\alpha}$. This problem cannot be bypassed by upgrading the regularity as we have done before to deal with the fact that the operator $L_{z}$ reduces the regularity of the test functions.

Since the main limitation comes from the differential part of the operator $L_{z}$, let us consider $L_{z}$ as the sum of the first order differential operator plus a perturbation. More precisely, let $\mathcal{L}:\varphi\mapsto\varphi'$ and $G_{t}:\varphi\mapsto \Psi(\cdot,\overline{\gamma}(t))R\varphi$ so that $L_{t}=\mathcal{L}+G_{t}$. Let us present here the heuristics behind the argument we use to bypass the issue induced by the differential operator $\mathcal{L}$: instead of studying the time derivative $\frac{d}{dt} \left< \tilde{\eta}_{t},\varphi \right>$ in \eqref{eq:closed:equation:tilde:1}, the idea is to find some family of test functions $(\varphi_{t})_{t\geq 0}$ such that $\left< \tilde{\eta}_{t},\frac{d}{dt} \varphi_{t}\right>=-\left< \tilde{\eta}_{t}, \mathcal{L} \varphi_{t}\right>$; thus the differential operator $\mathcal{L}$ vanishes in $\frac{d}{dt} \left< \tilde{\eta}_{t},\varphi_{t} \right>$ and Grönwall's argument can be applied.

More precisely, let us introduce the shift operators $\tau_{t}:\varphi\mapsto \varphi(\cdot+t)$ for all $t\geq 0$. Notice that these shift operators are linked with the method of characteristics applied to a transport equation with constant speed equal to $1$ which is exactly the dynamics described by the differential operator $\mathcal{L}$. Below are given some bounds for the operators $\mathcal{L}$, $G_{t}$ and $\tau_{t}$ when acting on the space $\mathcal{C}^{4}_{b}$.
\begin{lem}\label{lem:control:L:tau:G}
Let $\varphi$ be in $\mathcal{C}^{4}_{b}$. Assume that $t\mapsto ||\Psi(\cdot,\overline{\gamma}(t))||_{\mathcal{C}^{4}_{b}}$ is locally bounded. Then,
$||\mathcal{L}\varphi||_{\mathcal{C}^{3}_{b}}\leq ||\varphi||_{\mathcal{C}^{4}_{b}}$, for all $t\geq 0$, $||\tau_{t}\varphi||_{\mathcal{C}^{4}_{b}}=||\varphi||_{\mathcal{C}^{4}_{b}}$ and
\begin{equation*}
t\mapsto \frac{||G_{t}\varphi||_{\mathcal{C}^{4}_{b}}}{||\varphi||_{\mathcal{C}^{4}_{b}}} \quad \text{is locally bounded.}
\end{equation*}
\end{lem}
\begin{proof}
The first two assertions follow from the definition of the norms $||\cdot||_{\mathcal{C}^{k}_{b}}$. The third and last one follows from Leibniz rule.
\end{proof}

\begin{rem}
From now on, the test functions are considered in $\mathcal{C}^{4}_{b}$. Thus, we prove that $\eta$ is characterized by the limit equation as a process with values in the dual space $\mathcal{C}^{-4}_{b}$. Nevertheless, since $\mathcal{C}^{4}_{b}$ is dense in $\mathcal{W}^{3,\alpha}_{0}$, it is also characterized by the limit equation as a process with values in $\mathcal{W}^{-3,\alpha}_{0}$ for instance.
\end{rem}

Let $t\geq t'$ and $s$ in $\mathbb{R}$. Then,
$$\int_{t'}^{t} \tau_{t-z}\varphi'(s)dz=\int_{t'}^{t} \varphi'(s+t-z)dz= \varphi(s+t-t')-\varphi(s)=\tau_{t-t'}\varphi(s) -\varphi(s).$$
Moreover, since $\tau_{t}$ and $\mathcal{L}$ commute, one has
\begin{equation*}
\tau_{t-t'}\varphi(s) -\varphi(s)=\int_{t'}^{t} \mathcal{L}(\tau_{t-z}\varphi)(s)dz.
\end{equation*}
Yet, Lemma \ref{lem:control:L:tau:G} gives that $||\mathcal{L}(\tau_{t-z}\varphi)||_{\mathcal{C}^{3}_{b}}\leq ||\varphi||_{\mathcal{C}^{4}_{b}}$ thus $\int_{t'}^{t} \mathcal{L}\tau_{t-z}\varphi dz$ makes sense as a Bochner integral in $\mathcal{C}^{3}_{b}$ as soon as $\varphi$ is in $\mathcal{C}^{4}_{b}$. Hence, in the proof below we use the following statement: for all $\varphi$ in $\mathcal{C}^{4}_{b}$,
\begin{equation}\label{eq:tau:t:varphi:C3b}
\tau_{t-t'}\varphi -\varphi=\int_{t'}^{t} \mathcal{L}(\tau_{t-z}\varphi )dz, \quad \text{as points in $\mathcal{C}^{3}_{b}$.}
\end{equation}

\begin{prop}\label{prop:uniqueness:solution:limit:system}
Under (\hyperref[ass:for:CLT]{$\mathcal{A}_{\text{\tiny{CLT}}}$}), the system \eqref{eq:closed:equation:limit:1}-\eqref{eq:closed:equation:limit:2} has no more than one solution in $\mathcal{C}(\mathbb{R}_{+},\mathcal{W}^{-2,\alpha}_{0}\times \mathbb{R})$ once the initial condition $\eta_{0}$ and the ``noise'' $W$ are fixed.
\end{prop}
\begin{proof}
Let $(\eta,\Gamma)$ and $(\hat{\eta},\hat{\Gamma})$ be two solutions of \eqref{eq:closed:equation:limit:1}-\eqref{eq:closed:equation:limit:2} in $\mathcal{C}(\mathbb{R}_{+},\mathcal{W}^{-2,\alpha}_{0}\times \mathbb{R})$ associated with the the same ``noise'' $W$ and the same initial condition $\eta_{0}$. Denote by $\tilde{\eta}:=\eta-\hat{\eta}$ and $\tilde{\Gamma}:=\Gamma-\hat{\Gamma}$ the differences. Since $\alpha>1/2$, we have $\mathcal{W}^{-2,\alpha}_{0}\subset \mathcal{C}^{-4}_{b}$ (remind \eqref{eq:embedding:dual}) so $\tilde{\eta}$ belongs to $\mathcal{C}^{-4}_{b}$ and we will prove that $||\tilde{\eta}||_{\mathcal{C}^{-4}_{b}}=0$.

Starting from \eqref{eq:closed:equation:tilde:2}, one has
\begin{equation*}
|\tilde{\Gamma}_{t}|\leq h_{\infty}(t) ||\Psi(\cdot,\overline{\gamma}(t))||_{\mathcal{C}^{4}_{b}} \int_{0}^{t} ||\tilde{\eta}_{z}||_{\mathcal{C}^{-4}_{b}} dz + h_{\infty}(t){\rm Lip}(\Psi) \int_{0}^{t}|\tilde{\Gamma}|_{z} dz,
\end{equation*}
and Lemma \ref{lem:Gronwall:Generalization} gives $|\tilde{\Gamma}_{t}|\lesssim_{t} \int_{0}^{t} ||\tilde{\eta}_{z}||_{\mathcal{C}^{-4}_{b}} dz$. Now, let $\varphi$ be in $\mathcal{C}^{4}_{b}$ and use \eqref{eq:tau:t:varphi:C3b} and the fact that $\tilde{\eta}$ is in $\mathcal{W}^{-2,\alpha}_{0}\subset \mathcal{C}^{-3}_{b}$ to get $\left< \tilde{\eta}_{t},\varphi \right>=D_{1}-D_{2}$ where 
\begin{equation}\label{eq:def:D1:D2}
\begin{cases}
\displaystyle D_{1}:= \!\! \int_{0}^{t} \! \left< \tilde{\eta}_{t'},(\mathcal{L}+G_{t'})(\tau_{t-t'}\varphi) \right> dt' + \!\! \int_{0}^{t} \! \left< u_{t'},R\tau_{t-t'}\varphi\frac{\partial \Psi}{\partial y}(\cdot,\overline{\gamma}(t')) \right> \tilde{\Gamma}_{t'} dt'.\\
\displaystyle D_{2}:= \int_{0}^{t} \left< \tilde{\eta}_{t'},(\mathcal{L}+G_{t'})(\int_{t'}^{t} \mathcal{L}(\tau_{t-z}\varphi )dz) \right> dt' \\
\displaystyle \hphantom{\int_{0}^{t}  \left< \tilde{\eta}_{t'},(\mathcal{L}+G_{t'}) \right> } + \int_{0}^{t} \left< u_{t'},R\int_{t'}^{t} \mathcal{L}(\tau_{t-z}\varphi )dz\,\frac{\partial \Psi}{\partial y}(\cdot,\overline{\gamma}(t')) \right> \tilde{\Gamma}_{t'} dt'.
\end{cases}
\end{equation}
The linearity of the operators allows to write $D_{2}=D_{2,A}+D_{2,B}$ with
\begin{equation*}
\begin{cases}
\displaystyle D_{2,A}:= \int_{0}^{t} \int_{t'}^{t} \left< \tilde{\eta}_{t'},(\mathcal{L}+G_{t'})( \mathcal{L}(\tau_{t-z}\varphi )) \right> dz dt'\\
\displaystyle D_{2,B}:= \int_{0}^{t} \int_{t'}^{t} \left< u_{t'},R \mathcal{L}\tau_{t-z}\varphi \, \frac{\partial \Psi}{\partial y}(\cdot,\overline{\gamma}(t')) \right> \tilde{\Gamma}_{t'} dzdt'.
\end{cases}
\end{equation*}
Then, the idea is to use Fubini's theorem to exchange the two integrals $\int_{0}^{t}$ and $\int_{t'}^{t}$. 

On the one hand,
\begin{equation*}
\int_{0}^{t} \! \int_{t'}^{t} |\! \left< \tilde{\eta}_{t'},(\mathcal{L}+G_{t'})( \mathcal{L}(\tau_{t-z}\varphi )) \right> \!| dz dt' \leq \! \int_{0}^{t} \! \int_{t'}^{t} ||\tilde{\eta}_{t'}||_{\mathcal{C}^{-2}_{b}} ||(\mathcal{L}+G_{t'})( \mathcal{L}(\tau_{t-z}\varphi ))||_{\mathcal{C}^{2}_{b}} dz dt'.
\end{equation*}
Notice that $\sup_{t'\in [0,t]} ||\tilde{\eta}_{t'}||_{\mathcal{C}^{-2}_{b}}\leq C (\sup_{t'\in [0,t]} ||\eta_{t'}||_{-2,\alpha} + \sup_{t'\in [0,t]} ||\hat{\eta}_{t'}||_{-2,\alpha})<+\infty$ since $\eta$ and $\hat{\eta}$ takes values in $\mathcal{C}(\mathbb{R}_{+},\mathcal{W}^{-2,\alpha}_{0})$ and that, thanks to Lemma \ref{lem:control:L:tau:G}, for all $t'\leq t$, $||(\mathcal{L}+G_{t'})( \mathcal{L}(\tau_{t-z}\varphi ))||_{\mathcal{C}^{2}_{b}}\lesssim_{t} ||\varphi||_{\mathcal{C}^{4}_{b}} <+\infty$. Hence, Fubini's theorem gives
\begin{equation}\label{eq:D2A:Fubini}
D_{2,A}= \int_{0}^{t} \int_{0}^{z} \left< \tilde{\eta}_{t'},(\mathcal{L}+G_{t'})( \mathcal{L}(\tau_{t-z}\varphi )) \right> dt' dz.
\end{equation}

On the other hand,
\begin{equation*}
\! \int_{0}^{t} \! \int_{t'}^{t} \Big| \! \left< u_{t'},R \mathcal{L}\tau_{t-z}\varphi \,\frac{\partial \Psi}{\partial y}(\cdot,\overline{\gamma}(t')) \right> \! \Big| |\tilde{\Gamma}_{t'}| dzdt' \leq 2{\rm Lip}(\Psi) \! \int_{0}^{t} \! \int_{t'}^{t} ||\mathcal{L}(\tau_{t-z}\varphi )||_{\infty} |\tilde{\Gamma}_{t'}| dzdt'.
\end{equation*}
Remark that $||\mathcal{L}(\tau_{t-z}\varphi )||_{\infty}\leq ||\varphi||_{\mathcal{C}^{4}_{b}}$ and $\sup_{t'\in [0,t]} |\tilde{\Gamma}_{t}|\leq \sup_{t'\in [0,t]} |\Gamma_{t}| + \sup_{t'\in [0,t]} |\hat{\Gamma}_{t}|<+\infty$ since $\Gamma$ and $\hat{\Gamma}$ takes values in $\mathcal{C}(\mathbb{R}_{+},\mathbb{R})$. Hence, Fubini's theorem gives
\begin{equation}\label{eq:D2B:Fubini}
D_{2,B}= \int_{0}^{t} \int_{0}^{z} \left< u_{t'},R \mathcal{L}\tau_{t-z}\varphi \, \frac{\partial \Psi}{\partial y}(\cdot,\overline{\gamma}(t')) \right> \tilde{\Gamma}_{t'} dt' dz.
\end{equation}

Now, for any $z$ in $[0,t]$, Equation \eqref{eq:closed:equation:tilde:1} with $\varphi=\mathcal{L}(\tau_{t-z}\varphi)$ (it is a valid test function since it belongs to $\mathcal{C}^{3}_{b}\subset \mathcal{W}^{3,\alpha}_{0}$) gives
\begin{multline*}
\left< \tilde{\eta}_{z}, \mathcal{L}(\tau_{t-z}\varphi)\right>= \int_{0}^{z} \left< \tilde{\eta}_{t'},(\mathcal{L}+G_{t'})( \mathcal{L}(\tau_{t-z}\varphi )) \right> dt' \\
+ \int_{0}^{z} \left< u_{t'},R \mathcal{L}\tau_{t-z}\varphi \, \frac{\partial \Psi}{\partial y}(\cdot,\overline{\gamma}(t')) \right> \tilde{\Gamma}_{t'} dt'.
\end{multline*}
Gathering the equation above with \eqref{eq:D2A:Fubini} and \eqref{eq:D2B:Fubini} gives
\begin{equation*}
D_{2}=\int_{0}^{t} \left< \tilde{\eta}_{z}, \mathcal{L}(\tau_{t-z}\varphi)\right> dz,
\end{equation*}
which is exactly the term driven by $\mathcal{L}$ in the definition of $D_{1}$ \eqref{eq:def:D1:D2} so that, coming back to $D_{1}-D_{2}$, we have
\begin{equation*}
D_{1}-D_{2}=\left< \tilde{\eta}_{t},\varphi \right> = \int_{0}^{t} \left< \tilde{\eta}_{t'},G_{t'}(\tau_{t-t'}\varphi) \right> dt' + \int_{0}^{t} \left< u_{t'},R\tau_{t-t'}\varphi\, \frac{\partial \Psi}{\partial y}(\cdot,\overline{\gamma}(t')) \right> \tilde{\Gamma}_{t'} dt'.
\end{equation*}
Hence, using the bound we proved on $\tilde{\Gamma}$, we have for all $\varphi$ in $\mathcal{C}^{4}_{b}$,
\begin{equation*}
|\left< \tilde{\eta}_{t},\varphi \right>|\lesssim_{t} \int_{0}^{t} ||\tilde{\eta}_{t'}||_{\mathcal{C}^{-4}_{b}} ||G_{t'}(\tau_{t-t'}\varphi)||_{\mathcal{C}^{4}_{b}} dt' + 2||\varphi||_{\mathcal{C}^{4}_{b}}{\rm Lip}(\Psi) \int_{0}^{t} \int_{0}^{t'} ||\tilde{\eta}_{z}||_{\mathcal{C}^{-4}_{b}} dz,
\end{equation*}
and so $||\tilde{\eta}_{t}||_{\mathcal{C}^{-4}_{b}}\lesssim_{t} \int_{0}^{t} ||\tilde{\eta}_{t'}||_{\mathcal{C}^{-4}_{b}} dt'$ and Lemma \ref{lem:Gronwall:Generalization} gives that for all $t\geq 0$, $||\tilde{\eta}_{t}||_{\mathcal{C}^{-4}_{b}}=0$ thus $|\tilde{\Gamma}_{t}|=0$ thanks to the bound we proved on $\tilde{\Gamma}$. Finally, since $\mathcal{C}^{4}_{b}$ is dense in $\mathcal{W}^{2,\alpha}_{0}$, we have $||\tilde{\eta}_{t}||_{-2,\alpha}=0$. Thus, we have $(\eta,\Gamma)=(\hat{\eta},\hat{\Gamma})$ in $\mathcal{C}(\mathbb{R}_{+},\mathcal{W}^{-2,\alpha}_{0}\times\mathbb{R})$.
\end{proof}

We are now in position to conclude with the convergence of $(\eta^{n},\Gamma^{n})_{n\geq 1}$.

\begin{thm}\label{thm:convergence:eta:Gamma}
Under (\hyperref[ass:for:CLT]{$\mathcal{A}_{\text{\tiny{CLT}}}$}), for any $\alpha>1/2$, the sequence $(\eta^{n},\Gamma^{n})_{n\geq 1}$ converges in law in $\mathcal{D}(\mathbb{R}_{+},\mathcal{W}^{-2,\alpha}_{0}\times \mathbb{R})$ to the unique solution of the system \eqref{eq:closed:equation:limit:1}-\eqref{eq:closed:equation:limit:2} in $\mathcal{C}(\mathbb{R}_{+},\mathcal{W}^{-2,\alpha}_{0}\times \mathbb{R})$.
\end{thm}
\begin{proof}
Since $(\eta^{n},\Gamma^{n})_{n\geq 1}$ is tight (Theorem \ref{thm:tightness:M:n:eta:n} and Proposition \ref{prop:tightness:Gamma:n}), let $(\eta,\Gamma)$ be a limit point. According to Theorem \ref{thm:limit:equation}, $(\eta,\Gamma)$ is a solution of the limit system \eqref{eq:closed:equation:limit:1}-\eqref{eq:closed:equation:limit:2} in $\mathcal{C}(\mathbb{R}_{+},\mathcal{W}^{-2,\alpha}_{0}\times \mathbb{R})$. Finally, the law of $(\eta,\Gamma)$ is uniquely characterized by the limit system (Proposition \ref{prop:uniqueness:solution:limit:system} gives path-wise uniqueness and so Yamada-Watanabe theorem gives weak uniqueness by the same argument as \cite[Theorem IX.1.7(i)]{Revuz_1999}) and uniqueness of the limit law implies convergence of $(\eta^{n},\Gamma^{n})_{n\geq 1}$.
\end{proof}

\begin{rem}
As mentioned in the introduction, considering processes over finite time horizons would have lead to equivalent results. This claim is based on the fact that the limit equation \eqref{eq:closed:equation:limit:1} is independent of the values of the test function $\varphi$ outside the support $K_{t}$ of $\eta^{n}_{t}$. Indeed, 
\begin{itemize}
\item on the one hand, the test function $\varphi$ appears in the drift term, more precisely $\int_{0}^{t} \left< u_{z},\frac{\partial \Psi}{\partial y}(\cdot,\overline{\gamma}(z))R\varphi \right> \Gamma_{z} dz$, evaluated against the measure $u_{z}$ which is supported in $K_{t}$;
\item on the other hand, the covariance structure of the Gaussian process $W$ implies this independence property for $W_{t}(R\varphi)$.
\end{itemize}
In that sense, the convergence stated for the whole positive time line $\mathbb{R}_{+}$ in Theorem \ref{thm:convergence:eta:Gamma} implies that the central limit theorem also holds true for the process $(\eta^{n}_{t})_{0\leq t\leq \theta}$ as taking values in the dual of a standard Sobolev space of functions supported by $K_{\theta}$. Conversely, the limit equation is consistent in time in the sense that one can recover our result by sticking together the CLTs obtained for the finite time horizon processes $(\eta^{n}_{t})_{0\leq t\leq \theta}$.
\end{rem}

\section{Application to the ``almost'' derivation of an SPDE}

This section focuses on a system of stochastic partial differential equations (SPDE), introduced and studied in \cite{dumont2016private} where some qualitative properties are discussed. The SPDE is a noisy version of the PDE system \eqref{eq:edp:PPS} and is expected to be a more precise approximation of the age-dependent Hawkes processes in a mean-field framework. The SPDE system associated with the system size $n$ is the following
\begin{equation}\label{eq:edps:noisy:PPS}
\hspace{-0.7cm}
\begin{cases} 
\displaystyle \frac{\partial \tilde{u}^{n}\left(t,s\right)}{\partial t}+\frac{\partial \tilde{u}^{n}\left(t,s\right)}{\partial s} +\Psi\left(s,  X^{n}_{t}\right) \tilde{u}^{n}\left(t,s\right) + \sqrt{\frac{\Psi\left(s,  X^{n}_{t}\right) \tilde{u}^{n}\left(t,s\right)}{n}} \zeta(t,s) =0, \\
\\
\displaystyle \tilde{u}^{n}\left(t,0\right)=\int_{s\in \mathbb{R}} \Psi\left(s,   X^{n}_{t} \right)\tilde{u}^{n}\left(t,s\right) + \sqrt{\frac{\Psi\left(s,  X^{n}_{t}\right) \tilde{u}^{n}\left(t,s\right)}{n}} \zeta(t,s) ds,
\end{cases}
\end{equation}
where for all $t\geq 0$, $X^{n}_{t}=\int_{0}^{t} h(t-z) \tilde{u}^{n}(z,0)dz$ and $\zeta(t,s)$ is a Gaussian space-time white-noise. The important thing to note about $\zeta$ is that the $\mathcal{W}^{-2,\alpha}_{0}$-valued process defined by, for all $t\geq 0$ and $\varphi$ in $\mathcal{W}^{2,\alpha}_{0}$,
\begin{equation*}
\int_{0}^{t} \int_{0}^{+\infty} \varphi(s) \sqrt{\Psi(s,\overline{\gamma}(z)) u(z,s)} \zeta(z,s) ds dz,
\end{equation*}
is a Gaussian process with the same law as $W$ defined in Definition \ref{def:gaussian:process}.

Hence, at a first sight, there are some similarities between the system above and the limit system obtained for the fluctuation process, i.e. \eqref{eq:closed:equation:limit:1}-\eqref{eq:closed:equation:limit:2}. Let us give here some heuristics: assume that $u$, the solution of the PDE system \eqref{eq:edp:PPS}, and $\tilde{u}^{n}$ are close to each other, and similarly for the auxiliary variables $X(t)$ and $X^{n}_{t}$, then
\begin{itemize}
\item the ``non-noisy'' spiking dynamics term $\Psi\left(s,  X^{n}_{t}\right) \tilde{u}^{n}\left(t,s\right)$ is close to the mean-field spiking dynamics, appearing in the operator $L_{t}$ in the limit system \eqref{eq:closed:equation:limit:1}-\eqref{eq:closed:equation:limit:2}, modulo an error term which is expected to appear, in a linear approximation, as the term involving the derivative $\frac{\partial \Psi}{\partial y}$ in \eqref{eq:closed:equation:limit:1};
\item the covariance structure of the Gaussian process $W$, appearing in the limit system \eqref{eq:closed:equation:limit:1}-\eqref{eq:closed:equation:limit:2}, is close to the covariance structure of the noise term appearing in the SPDE system above.
\end{itemize}

What is proposed in this section is to consider the second-order approximation of the empirical measure given by the central limit theorem, namely $\hat{u}^{n}_{t}=u_{t} + n^{-1/2} \eta_{t}$ where $u_t=u(t,\cdot)$ is the probability distribution solution of \eqref{eq:edp:PPS}, and show that is is an ``almost'' solution of the SPDE system in some sense defined below. Up to our knowledge, this kind of result is novel and deserves to be developed in this article.

Let us remind that this section is devoted to an application of the CLT so Assumption (\hyperref[ass:for:CLT]{$\mathcal{A}_{\text{\tiny{CLT}}}$}) is supposed to hold true below. Furthermore, the stronger assumption that $\Psi$ is in $\mathcal{C}^{4}_{b}$ is made.

\subsection{Theoretical frame for the SPDE}

Up to our knowledge, there is no theoretical frame well established for the SPDE system \eqref{eq:edps:noisy:PPS}. A pathwise notion of solution seems to be hard to handle because of the square root term appearing in front of the Gaussian white noise. In particular, it is not trivial to show that the argument of the square root remains non negative. That is why we propose the following notion of solution.

\begin{defn}\label{def:solution:spde}
The measure-valued process $(\tilde{u}^n_t)_{t\geq 0}$ is a solution of  \eqref{eq:edps:noisy:PPS} if it satisfies: for all $\varphi$ in $\mathcal{C}^\infty_b$,
\begin{equation}\label{eq:evolution:equation:u:n}
\left< \tilde{u}^{n}_{t},\varphi\right> - \left< \tilde{u}^{n}_{0},\varphi\right> = \int_0^t \left< \tilde{u}^{n}_{z}, \tilde{L}_z^n \varphi\right> dz + \tilde{W}^n_t(R\varphi),
\end{equation}

where $\tilde{W}^n$ is a Gaussian process with Doob-Meyer process given by,
\begin{equation}\label{eq:covariance:W:tilde:n}
\left< {<\!\!<\!\tilde{W}^{n}\!>\!\!>}_{t} (\varphi_{1}), \varphi_{2} \right> = \frac{1}{n} \int_{0}^{t} \left<\tilde{u}^n_{z},\varphi_{1}\varphi_{2}\Psi(\cdot,\tilde{\gamma}^n_z)\right>dz,
\end{equation}

and $\tilde{L}_z^n \varphi:=\varphi' + \Psi(\cdot,\tilde{\gamma}^n_t)\varphi$ with 
\begin{equation}\label{eq:defintion:Xn}
\tilde{\gamma}^n_t = \int_0^t h(t-z) \left< \tilde{u}^n_z, \Psi(\cdot,\tilde{\gamma}^n_z)\right> dz + \int_0^t h(t-z) d\tilde{W}^n_z(\mathbf{1}).
\end{equation}
\end{defn}

The well-posedness of such definition is not addressed here. We only stress the fact that $\hat{u}^n$ is an ``almost'' solution in some sense related to Definition \ref{def:solution:spde}.

\subsection{Weak sense dynamics for the second-order approximation}

To catch the dynamics of $\hat{u}^{n}$, we somehow want to add up the dynamics of $u$, given by the PDE system \eqref{eq:edp:PPS}, and the dynamics of $\eta$, given by the limit system. These two are formulated by different means, the main difference being that the PDE formulation involves (in the weak sense) bivariate test functions whereas the formulation for $\eta$ involves univariate test functions: we turn to the second one in order to get a system like \eqref{eq:evolution:equation:u:n}-\eqref{eq:defintion:Xn}.\\

On the one hand, the dynamics of $\eta$ is given by the limit equation \eqref{eq:closed:equation:limit:1} that we remind here: for all $\varphi$ in $\mathcal{W}^{3,\alpha}_{0}$,
\begin{equation*}
\left< \eta_{t},\varphi \right>-\left< \eta_{0},\varphi \right> = \int_{0}^{t} \left< \eta_{z},L_{z}\varphi \right> dz + \int_{0}^{t} \left< u_{z},\frac{\partial \Psi}{\partial y}(\cdot,\overline{\gamma}(z))R\varphi \right> \Gamma_{z} dz +  W_{t}(R\varphi).
\end{equation*}

On the other hand, the dynamics of $u$ is given (see \cite[Theorem III.5.]{chevallier2015mean}) by the weak sense formulation of the PDE system \eqref{eq:edp:PPS} (which is driven by the generator $L_t$): for all $\phi$ in $\mathcal{C}^{\infty}_{c,b}(\mathbb{R}_{+}^{2})$,
\begin{equation}\label{eq:edp:PPS:weak:sense}
\int_{\mathbb{R}_{+}^{2}} \left( \frac{\partial \phi}{\partial t}\left(t,s\right) + L_t(\phi(t,\cdot)) \right) u\left( t, s\right)dtds +\int_{\mathbb{R}_{+}} \phi(0,s) u_0(s)ds =0
\end{equation}
where the test function space $\mathcal{C}^{\infty}_{c,b}(\mathbb{R}_{+}^{2})$ is defined as follows,

\medskip
\noindent$\mathcal{C}_{c,b}^{\infty}(\mathbb{R}_{+}^{2})\,
 \textrm{\begin{tabular}{|l} The function $\phi$ belongs to $\mathcal{C}_{c,b}^{\infty}(\mathbb{R}_{+}^{2})$  if  \\ 
 $\quad\bullet$ $\phi$ is continuous, uniformly bounded, \\ 
 $\quad\bullet$ $\phi$ has  uniformly bounded derivatives of every order,\\ 
 $\quad\bullet$ there exists $T>0$ such that $\phi(t,s)=0$ for all $t>T$ and $s\geq 0$.
 \end{tabular}}$\\
 
Then, taking $\phi(t,s)$ that converges to a product function of the form $\varphi(s)\mathds{1}_{t\leq T}$, we get that for all $\varphi$ in $\mathcal{C}_{b}^{\infty}(\mathbb{R}_{+})$,
\begin{equation}\label{eq:evolution:equation:PPS}
\left< u_T, \varphi \right> - \left< u_0, \varphi \right> = \int_0^T \left< u_t,L_t \varphi \right> dt.
\end{equation}
Combining \eqref{eq:evolution:equation:PPS} with the limit equation \eqref{eq:closed:equation:limit:1}, we prove that $\hat{u}^{n}$ satisfies: for all $\varphi$ in $\mathcal{C}_{b}^{\infty}(\mathbb{R}_{+})$,
\begin{multline*}
\left< \hat{u}^{n}_{t},\varphi\right> - \left< \hat{u}^{n}_{0}, \varphi\right> = \int_0^t \left< \hat{u}^{n}_{z},L_z \varphi\right> dz + n^{-1/2}  \int_{0}^{t} \left< u_{z},\frac{\partial \Psi}{\partial y}(\cdot,\overline{\gamma}(z))R\varphi \right> \Gamma_{z} dz \\
+ n^{-1/2} W_t(R\varphi).
\end{multline*}
This last equation can be rewritten, with anything new but some notation, as a system in the flavour of \eqref{eq:evolution:equation:u:n}-\eqref{eq:defintion:Xn} as stated in the proposition below.

\begin{prop}
The process $\hat{u}^n$ satisfies: for all $\varphi$ in $\mathcal{C}_{b}^{\infty}$,
\begin{equation}\label{eq:evolution:equation:v:n}
\left< \hat{u}^{n}_{t},\varphi\right> - \left< \hat{u}^{n}_{0},\varphi\right> = \int_0^t \left< \hat{u}^{n}_{z}, \hat{L}_z^n \varphi\right> dz + \hat{W}^n_t(R\varphi) + r^n_t(\varphi),
\end{equation}
where $\hat{W}^n$ is a Gaussian process with Doob-Meyer process given by,
\begin{equation}\label{eq:covariance:W:hat:n}
{\rm DM}_t(\varphi_1,\varphi_2):= \left< {<\!\!<\!\hat{W}^{n}\!>\!\!>}_{t} (\varphi_{1}), \varphi_{2} \right> = \frac{1}{n} \int_{0}^{t}  \left<u_{z},\varphi_{1}\varphi_{2}\Psi(\cdot,\overline{\gamma}(z))\right> dz,
\end{equation}
and $\hat{L}_z^n \varphi:=\varphi' + \Psi(\cdot,\hat{\gamma}^n_t)\varphi$ with 
\begin{equation}\label{eq:defintion:X:hat:n}
\hat{\gamma}^n_t = \int_0^t h(t-z) \left< \hat{u}^n_z, \Psi(\cdot,\hat{\gamma}^n_z)\right> dz + \int_0^t h(t-z) d\hat{W}^n_z(\mathbf{1}).
\end{equation}
The following notation is used above: $\hat{W}^n:=n^{-1/2}W$ where $W$ is the Gaussian process of Definition \ref{def:gaussian:process} and
\begin{equation}
r^n_t(\varphi):= \int_0^t \left< \hat{u}^{n}_{z}, (L_z-\hat{L}_z^n) \varphi\right> + \left< u_{z},\frac{\partial \Psi}{\partial y}(\cdot,\overline{\gamma}(z))R\varphi \right> \frac{\Gamma_{z}}{\sqrt{n}} dz.
\end{equation}
\end{prop}

\begin{rem}
The process $(\hat{\gamma}^n_t)_{t\geq 0}$ is characterized by the fixed point equation \eqref{eq:defintion:X:hat:n} (see Lemma \ref{lem:well:posed:X:hat:n}).
\end{rem}

Comparing \eqref{eq:evolution:equation:u:n}-\eqref{eq:defintion:Xn} with \eqref{eq:evolution:equation:v:n}-\eqref{eq:defintion:X:hat:n}, the only differences between the two systems are, the additional term $r^n_t(\varphi)$ in \eqref{eq:evolution:equation:v:n}, and the substitution of the Doob-Meyer process associated with the noise term: it should be given by
\begin{equation}\label{eq:covariance:W:hat:tilde:n}
\hat{\rm DM}_t(\varphi_1,\varphi_2):= \frac{1}{n} \int_{0}^{t} \left<\hat{u}^n_{z},\varphi_{1}\varphi_{2}\Psi(\cdot,\hat{\gamma}^n_z)\right>  dz.
\end{equation}
As stated below, these two differences are negligeable, as $n\to +\infty$, with respect to the other terms of the system that are at least of order $n^{-1/2}$.

\begin{prop}\label{prop:almost:solution}
The distribution of the process $\hat{u}^n$ is an ``almost'' solution of \eqref{eq:evolution:equation:u:n}-\eqref{eq:defintion:Xn} in the sense that: the rest term in \eqref{eq:evolution:equation:v:n} is negligible since
\begin{equation*}
\mathbb{E}\left[|r^n_t(\varphi)|\right] \lesssim_{t} n^{-1} ||\varphi||_{2,1},
\end{equation*}
and the covariance structures are ``almost'' the same since
\begin{equation*}
\mathbb{E}\left[ \left| {\rm DM}_t(\varphi_1,\varphi_2) - \hat{\rm DM}_t(\varphi_1,\varphi_2) \right| \right] \lesssim_{t} n^{-3/2} ||\varphi_1||_{3,1} ||\varphi_2||_{3,1}.
\end{equation*}
\end{prop}

The proof of Proposition \ref{prop:almost:solution} relies on some refined versions of already established estimates and is given in Appendix \ref{sec:proof:prop:almost:solution}. The main difficulty and difference in comparison with the preceding sections is that $\hat{u}^n_t$ is not a probability measure, unlike $u_t$ and $\overline{\mu}^n_t$. However, this difficulty is bypassed in the proof by means of the following estimates.

\begin{lem}\label{lem:estimates:u:hat}
For any $\varphi$ in $\mathcal{C}^2_b(\mathbb{R}_+)$, we have
\begin{equation*}
|\left< \hat{u}^n_t,\varphi \right>|\leq \left(1 + \frac{C}{\sqrt{n}} ||\eta_t||_{-2,1}\right) ||\varphi||_{\mathcal{C}^2_b}.
\end{equation*}
We mainly use this estimate with the intensity function $\Psi$: for all $y_1,y_2$,
\begin{equation}\label{eq:lipschitz:control:vs:eta}
\begin{cases}
|\left< \hat{u}^n_t, \Psi(\cdot,y_1) \right>|\leq \left(1 + \frac{C}{\sqrt{n}} ||\eta_t||_{-2,1}\right) ||\Psi||_{\mathcal{C}^2_b},\\
|\left< \hat{u}^n_t, \Psi(\cdot,y_1)-\Psi(\cdot,y_2) \right>|\leq \left(1 + \frac{C}{\sqrt{n}} ||\eta_t||_{-2,1}\right) ||\Psi||_{\mathcal{C}^3_b} |y_1-y_2|.
\end{cases}
\end{equation}
\end{lem}

The second line of \eqref{eq:lipschitz:control:vs:eta} is very useful to use some kind of Lipschitz control even when integrating with respect to $\hat{u}^n_t$, which is not as direct as in the case when the integration is done with respect to a probability measure.

\begin{proof}
The first assertion is a direct consequence of the definition of $\hat{u}^n$ and the embedding \eqref{eq:f:k:alpha:leq:f:Ckb}. Since $\Psi$ is in $\mathcal{C}^4_b(\mathbb{R}_+\times\mathbb{R})$, the next ones are direct from
\begin{equation*}
\begin{cases}
||\Psi(\cdot,y_1)||_{\mathcal{C}^2_b(\mathbb{R}_+)}\leq ||\Psi||_{\mathcal{C}^2_b(\mathbb{R}_+\times\mathbb{R})},\\
||\Psi(\cdot,y_1)-\Psi(\cdot,y_2)||_{\mathcal{C}^2_b(\mathbb{R}_+)}\leq ||\Psi||_{\mathcal{C}^3_b(\mathbb{R}_+\times\mathbb{R})} |y_1-y_2|.
\end{cases}
\end{equation*}
\end{proof}

To conclude, let us remind that we have shown how a second-order approximation of the empirical measure $\overline{\mu}^n_t$, namely $\hat{u}^n_t$ since $\sqrt{n}(\overline{\mu}^n_t-\hat{u}^n_t)$ goes to $0$ as a consequence of Theorem \ref{thm:convergence:eta:Gamma}, can be considered as an ``almost'' solution of the SPDE \eqref{eq:edps:noisy:PPS}. The next step would be to prove that the/any solution of the SPDE is a second order approximation of $\overline{\mu}^n_t$. To address such a question, we would first need a suitable theoretical framework to treat the well-posedness of the SPDE. This is the subject of a future work.

\newpage
\appendix

\section{Proofs}

\subsection{Proofs linked with Proposition \ref{prop:rate:of:convergence:proba:k:tuple:different:ages}}
\label{sec:proof:prop:rate:of:convergence}

\paragraph{Proof of \eqref{eq:bound:Enkp}.}

For simplicity, we show that, for every $m\leq n$, there exists a constant $C$ which is independent of $n$, $p$ and $m$ such that
\begin{equation}\label{eq:bound:Enmkp}
\mathbb{E}\left[ \left( \frac{1}{m} \sum_{j=1}^{m} (\Delta^{n,j}_{t-})^{p} \right)^{k} \right] \leq C \left( \sum_{k'=1}^{k-1} m^{k'-k} \varepsilon^{(k',pk)}_{n}(t) + \varepsilon^{(k,p)}_{n}(t) \right),
\end{equation}
from which \eqref{eq:bound:Enkp} follows by choosing $m=\lfloor \frac{n}{k} \rfloor$. 

Let us recall the multinomial formula using multi-indices ${\bf q}=(q_{1},\dots ,q_{m})$,
\begin{equation*}
\left( \frac{1}{m} \sum_{i=1}^{m} x_{i} \right)^{k} = \frac{1}{m^{k}} \sum_{|{\bf q}|=k} \binom{k}{{\bf q}} \prod_{i=1}^{m} x_{i}^{q_{i}},
\end{equation*}
where $|{\bf q}|=\sum_{i=1}^{m} q_{i}$. Denote by $k({\bf q})$ the number of strictly positive indices in ${\bf q}$. Since the $q_{i}$'s are integers, $|{\bf q}|=k$ implies $k({\bf q})\leq k$. First, let us remark that, for all $k'=1,\dots ,k$, the number of multi-indices ${\bf q}$ such that $k({\bf q})=k'$ and $|{\bf q}|=k$ is bounded by $p(k',k)m^{k'}$ with $p(k',k):=\binom{k'-1}{k-1}$ being the number of partitions of $k$ into exactly $k'$ parts. Indeed, the vector consisting in the $k'$ strictly positive indices forms a partition of $k$ and there are at most $m^{k'}$ ways to complete it by $m-k'$ zeros to build a vector of length $m$.

Then, using the exchangeability of the processes $\Delta^{n,j}$, we have
\begin{itemize}
\item if $k({\bf q})=k$, then all the positive $q_{i}$'s are equal to one and $\mathbb{E}[ \prod_{i=1}^{n} ((\Delta^{n,i}_{t-})^{p})^{q_{i}} ]=\varepsilon^{(k,p)}(t)$,
\item if $k({\bf q})<k$, we can bound all the positive $q_{i}$'s by $k$ so that $\mathbb{E}[ \prod_{i=1}^{n} ((\Delta^{n,i}_{t-})^{p})^{q_{i}} ] \leq \varepsilon^{(k({\bf q}),pk)}_{n}(t)$.
\end{itemize}
Hence, using that $\binom{k}{{\bf q}}\leq k!$, \eqref{eq:bound:Enmkp} holds with $C=\max_{k'=1,\dots ,k} p(k',k) k!$ for instance.

\paragraph{Proof of \eqref{eq:bound:xikn}.}

Let us first recall that $\xi^{(k)}_{n}(t)=\mathbb{E}\left[  |\gamma^{n}_{t} - \overline{\gamma}(t)| ^{k} \right]$ where $\gamma^{n}_{t}$ and $\overline{\gamma}(t)$ are respectively defined below \eqref{eq:definition:coupling:MFHawkes} and in \eqref{eq:def:gamma:barre:t}. By convexity of the function $x\mapsto |x|^{k}$ (remind that $k\geq 2$), let us consider the decomposition 
\begin{equation}\label{eq:decomposition:xi:n:k}
\xi^{(k)}_{n}(t)\leq 4^{k-1} \left( A^{n}(t) + B^{n}(t)+ C^{n}(t) +D^{n}(t) \right)
\end{equation}
where
\begin{equation*}
\begin{cases}
A^{n}(t):= \mathbb{E}\left[ \left| \int_{0}^{t} h(t-z) \frac{1}{n} \sum_{j=1}^{n} (N^{n,j}(dz) - \lambda^{n,j}_{z}dz) \right|^{k} \right],\\
B^{n}(t):= \mathbb{E}\left[ \left| \int_{0}^{t}  h(t-z) \frac{1}{n} \sum_{j=1}^{n} (\Psi(S^{n,j}_{z-},\gamma^{n}_{z}) - \Psi(\overline{S}^{j}_{z-},\gamma^{n}_{z})) dz  \right|^{k} \right],\\
C^{n}(t):=\mathbb{E}\left[ \left| \int_{0}^{t} h(t-z) \frac{1}{n} \sum_{j=1}^{n} (\Psi(\overline{S}^{j}_{z-},\gamma^{n}_{z}) - \Psi(\overline{S}^{j}_{z-},\overline{\gamma}(z))) dz \right|^{k} \right],\\
D^{n}(t):= \mathbb{E}\left[ \left| \int_{0}^{t} h(t-z) \frac{1}{n} \sum_{j=1}^{n} (\overline{\lambda}^{j}_{z} - \overline{\lambda}(z)) dz  \right|^{k} \right].\\
\end{cases}
\end{equation*}
Recall that $\lambda^{n,j}_{z}=\Psi(S^{n,j}_{z-},\gamma^{n}_{z})$ and $\overline{\lambda}^{j}_{z}=\Psi(\overline{S}^{j}_{z-},\overline{\gamma}(z))$.

- \underline{Study of $A^{n}(t)$.} Fix $t$ and consider the martingale $(M^{n,t}_{x})_{x\geq 0}$ defined, for all $x\geq 0$, by
\begin{equation*}
M^{n,t}_{x}:=\int_{0}^{x} h(t-z) \frac{1}{n} \sum_{j=1}^{n} (N^{n,j}(dz) - \lambda^{n,j}_{z}dz).
\end{equation*}
Its quadratic variation is $[M^{n,t}]_{x}= n^{-2} \int_{0}^{x} h(t-z)^{2} \sum_{j=1}^{n} N^{n,j}(dz)$. Yet, Assumption (\hyperref[ass:h:infty]{$\mathcal{A}^{h}_{\infty}$}) implies that
\begin{equation*}
[M^{n,t}]_{t}\leq n^{-2} h_{\infty}(t)^{2} \sum_{j=1}^{n} N^{n,j}_{t}.
\end{equation*}
Using the convexity of the power function (since $k/2\geq 1$) and exchangeability, one has
\begin{equation*}
\mathbb{E}\left[ [M^{n,t}]_{t}^{k/2} \right] \lesssim_{(t,k)} n^{-k} \, \mathbb{E}\left[ \big| \sum_{j=1}^{n} N^{n,j}_{t}\big|^{k/2} \right]\lesssim_{(t,k)}  n^{-k} \, n^{k/2} \mathbb{E}\left[ |N^{n,1}_{t}|^{k/2} \right].
\end{equation*}
Yet, the intensity of $N^{n,1}$ is bounded by $||\Psi||_{\infty}$ so $N^{n,1}$ is stochastically dominated by a Poisson process with intensity $||\Psi||_{\infty}$. Hence, $\mathbb{E}[ |N^{n,1}_{t}|^{k/2} ]\leq \mathbb{E}[{\rm Poiss}(t||\Psi||_{\infty})^{k/2}]$ where ${\rm Poiss}(t||\Psi||_{\infty})$ is a Poisson variable with parameter $t||\Psi||_{\infty}$. This last expectation is bounded uniformly in $n$ by a locally bounded function of the time $t$. Then, Burkholder-Davis-Gundy inequality \cite[p. 894]{shorack2009empirical} gives
\begin{equation*}
A^{n}_{1}(t)= \mathbb{E}\left[ |M^{n,t}_{t}|^{k} \right]\leq \mathbb{E}\left[ [M^{n,t}]_{t}^{k/2} \right]\lesssim_{(t,k)} n^{-k/2}.
\end{equation*}

- \underline{Study of $B^{n}(t)$.} Here, we use the fact that $S^{n,j}_{z-}=\overline{S}^{j}_{z-}$ with high probability and more precisely we recover the quantities $\varepsilon_{n}^{(k,p)}$ that we want to control. Using the convexity of the power function, Assumption (\hyperref[ass:h:infty]{$\mathcal{A}^{h}_{\infty}$}) and denoting $x^{j}_{z}= \Psi(S^{n,j}_{z-},\gamma^{n}_{z}) - \Psi(\overline{S}^{j}_{z-},\gamma^{n}_{z})$ we have
\begin{equation*}
B^{n}(t) \leq h_{\infty}(t)^{k} t^{k-1}  \int_{0}^{t} \mathbb{E}\left[ \bigg| \frac{1}{n} \sum_{j=1}^{n} x^{j}_{z} \bigg|^{k} \right] dz .
\end{equation*}
Yet, $|x^{j}_{z}|$ is bounded by $||\Psi||_{\infty}\mathds{1}_{\{ S^{n,j}_{z-} \neq \overline{S}^{j}_{z-} \}}\leq ||\Psi||_{\infty}\Delta^{n,j}_{z-}$. Hence, using \eqref{eq:bound:Enmkp} with $p=1$ and $m=n$,
\begin{eqnarray}
B^{n}(t)&\leq& ||\Psi||_{\infty}^{k} h_{\infty}(t)^{k} t^{k-1} \int_{0}^{t} \mathbb{E}\left[ \bigg| \frac{1}{n} \sum_{j=1}^{n} \Delta^{n,j}_{z-} \bigg|^{k} \right] dz \nonumber\\
&\leq & ||\Psi||_{\infty}^{k} h_{\infty}(t)^{k} C t^{k-1} \int_{0}^{t} \left( \sum_{k'=1}^{k-1} n^{k'-k} \varepsilon^{(k',k)}_{n}(z) +\varepsilon^{(k,1)}_{n}(z) \right) dz \nonumber\\
&\leq & ||\Psi||_{\infty}^{k} h_{\infty}(t)^{k} C t^{k} \left( \sum_{k'=1}^{k-1} n^{k'-k} \varepsilon^{(k',k)}_{n}(t) + \varepsilon^{(k,1)}_{n}(t) \right), \label{eq:bound:different:ages:k}
\end{eqnarray}
where the last line comes from the fact that the $\varepsilon^{(k',k)}_{n}$'s are non-decreasing functions of $t$. Hence, $B^{n}_{t}\lesssim_{(t,k)} \sum_{k'=1}^{k-1} n^{k'-k} \varepsilon^{(k',k)}_{n}(t) + \varepsilon^{(k,1)}_{n}(t)$.

- \underline{Study of $C^{n}(t)$.} Using the Lipschitz continuity of $\Psi$, Assumption (\hyperref[ass:h:infty]{$\mathcal{A}^{h}_{\infty}$}), one has
\begin{equation}\label{eq:bound:gronwall:type:k}
C^{n}(t) \leq {\rm Lip}(\Psi)^{k} h_{\infty}(t)^{k} t^{k-1}  \int_{0}^{t} \mathbb{E}\left[  |\gamma^{n}_{z} -\overline{\gamma}(z) |^{k} \right] dz\lesssim_{(t,k)}  \int_{0}^{t} \xi^{(k)}_{n}(z) dz.
\end{equation}

- \underline{Study of $D^{n}(t)$.} First remark that using Assumption (\hyperref[ass:h:infty]{$\mathcal{A}^{h}_{\infty}$}), we have
\begin{equation*}
D^{n}(t)\leq h_{\infty}(t)^{k} t^{k-1} \int_{0}^{t} \mathbb{E}\left[ \big| \frac{1}{n} \sum_{j=1}^{n} \overline{\lambda}^{j}_{z} - \overline{\lambda}(z) \big|^{k} \right] dz.
\end{equation*}
Yet, the $\overline{\lambda}^{j}_{z}$'s are i.i.d. with mean $\overline{\lambda}(z)$ and they are bounded by $||\Psi||_{\infty}$. Hence, Rosenthal inequality \cite{merlevede2013rosenthal} gives the existence of a constant $C(k)$ which depends only on $k$ and $||\Psi||_{\infty}$ such that
\begin{equation*}
\mathbb{E}\left[ \big| \frac{1}{n} \sum_{j=1}^{n} \overline{\lambda}^{j}_{z} - \overline{\lambda}(z) \big|^{k} \right] \leq C(k) n^{-k/2}.
\end{equation*}
It then follows that $D^{n}(t) \lesssim_{(t,k)} n^{-k/2}$.

One deduces from the decomposition \eqref{eq:decomposition:xi:n:k} and the four bounds on $A^{n}$, $B^{n}$, $C^{n}$ and $D^{n}$ that
\begin{equation*}
\xi^{(k)}_{n}(t) \lesssim_{(t,k)} \left( n^{-k/2} + \sum_{k'=1}^{k-1} n^{k'-k} \varepsilon^{(k',k)}_{n}(t) + \varepsilon^{(k,1)}(t) \right) + \int_{0}^{t} \xi^{(k)}_{n}(z) dz,
\end{equation*}
and so Lemma \ref{lem:Gronwall:Generalization} below gives the desired bound.

\subsection{Proof of Proposition \ref{prop:decomposition:eta}}
\label{sec:proof:prop:decomposition:eta}

By definition of $\eta^{n}$ (Equation \eqref{eq:def:eta:n:t}),
$$
\left< \eta_{t}^{n},\varphi\right> -\left< \eta_{0}^{n},\varphi\right> =\sqrt{n}\left[\frac{1}{n}\sum_{i=1}^{n} \left(\left< \delta_{S^{n,i}_{t}},\varphi\right> -\left< \delta_{S_{0}^{n,i}},\varphi\right> \right)-\left(\left< u_{t},\varphi\right> -\left< u_{0},\varphi\right> \right)\right].
$$
Since, for all $i=1,\dots,n$, the age process $(S^{n,i}_{t})_{t\geq 0}$ is piece-wise continuous, increasing with rate $1$ and jumps from $S^{n,i}_{t-}$ to $0$ when $N^{n,i}_{t}-N^{n,i}_{t-}=1$, we have
$$
\left< \delta_{S^{n,i}_{t}},\varphi\right> -\left< \delta_{S^{n,i}_{0}},\varphi\right> = \int_{0}^{t}\varphi'\left(S^{n,i}_{z}\right)dz + \int_{0}^{t} R\varphi\left(S^{n,i}_{z-}\right) N^{n,i}(dz)
$$
and so
\begin{equation}\label{eq:decomposition:mu:barre}
\left< \overline{\mu}^{n}_{S_{t}} , \varphi \right> - \left< \overline{\mu}^{n}_{S_{0}} , \varphi \right> = \int_{0}^{t} \left<\overline{\mu}^{n}_{S_{z}} , \varphi'\right> dz + \frac{1}{n} \sum_{i=1}^{n} \int_{0}^{t} R\varphi\left(S^{n,i}_{z-}\right) N^{n,i}(dz).
\end{equation}
Now, we have in the same way
\begin{equation*}
\left< \delta_{\overline{S}^{1}_{t}},\varphi\right> -\left< \delta_{\overline{S}^{1}_{0}},\varphi\right> =  \int_{0}^{t}\varphi'\left(\overline{S}^{1}_{z}\right)dz + \int_{0}^{t}R\varphi \left(\overline{S}^{1}_{z-} \right) \overline{N}^{1}(dz)
\end{equation*}
and, by definition of $(u_{t})_{t\geq 0}$,
\begin{equation}\label{eq:decomposition:P:brute}
\left< u_{t},\varphi\right> -\left< u_{0},\varphi\right> = \mathbb{E}\left[  \int_{0}^{t}\varphi'\left(\overline{S}^{1}_{z}\right)dz \right] + \mathbb{E}\left[ \int_{0}^{t} R\varphi\left(\overline{S}^{1}_{z-}\right) \overline{N}^{1}(dz) \right].
\end{equation}
Yet, since $\varphi'$ is bounded, Fubini's theorem gives that $\mathbb{E}[  \int_{0}^{t}\varphi'(\overline{S}^{1}_{z})dz ] = \int_{0}^{t} \left<u_{z},\varphi'\right> dz$. Moreover, remind that the intensity of $\overline{N}^{1}$ is $\overline{\lambda}^{1}_{t}=\Psi(\overline{S}^{1}_{t-},\overline{\gamma}(t))$. Yet, since $\varphi$ and $\Psi$ are bounded,
\begin{equation*}
\mathbb{E}\left[ \int_{0}^{t}\left| R\varphi\left(\overline{S}^{1}_{z-}\right) \right| \overline{\lambda}^{1}_{z} dz \right] < +\infty,
\end{equation*}
and so $\mathbb{E}[ \int_{0}^{t} R\varphi(\overline{S}^{1}_{z-}) \overline{N}^{1}(dz) ] = \mathbb{E}[ \int_{0}^{t} R\varphi(\overline{S}^{1}_{z-}) \overline{\lambda}^{1}_{z} dz ]$ since $(\overline{S}^{1}_{t-})_{t\geq 0}$ is a predictable process (see \cite[II. T8]{Bremaud_PP}). Using once again Fubini's theorem, we end up with
\begin{equation*}
\mathbb{E}\left[ \int_{0}^{t} R\varphi\left(\overline{S}^{1}_{z-}\right) \overline{N}^{1}(dz) \right] = \int_{0}^{t} \left<u_{z},\Psi(\cdot,\overline{\gamma}(z))R\varphi\right> dz
\end{equation*}
and so~\eqref{eq:decomposition:P:brute} becomes
\begin{equation}\label{eq:decomposition:P:fine}
\left< u_{t},\varphi\right> -\left< u_{0},\varphi\right> = \int_{0}^{t} \left<u_{z},\varphi'\right> dz + \int_{0}^{t} \left<u_{z},\Psi(\cdot,\overline{\gamma}(z))R\varphi\right> dz.
\end{equation}
Gathering~\eqref{eq:decomposition:mu:barre} and~\eqref{eq:decomposition:P:fine} gives
\begin{multline*}
\left< \eta_{t}^{n},\varphi\right> -\left< \eta_{0}^{n},\varphi\right> = \int_{0}^{t} \left< \eta^{n}_{z}, \varphi' \right> dz \\
+ \sqrt{n} \left( \frac{1}{n}\sum_{i=1}^{n}\int_{0}^{t} R\varphi\left(S^{n,i}_{z-}\right) N^{n,i}(dz) - \int_{0}^{t} \left<u_{z},\Psi(\cdot,\overline{\gamma}(z)) R\varphi\right> dz \right)
\end{multline*}
and so
\begin{multline}\label{eq:decomposition:eta:fail}
\left< \eta_{t}^{n},\varphi\right> -\left< \eta_{0}^{n},\varphi\right> = \int_{0}^{t}\left< \eta_{z}^{n},L_{z}\varphi\right> dz \\
+ \sqrt{n} \left( \frac{1}{n}\sum_{i=1}^{n}\int_{0}^{t} R\varphi\left(S^{n,i}_{z-}\right) N^{n,i}(dz) - \int_{0}^{t} \left< \overline{\mu}^{n}_{S_{z-}}, \Psi(\cdot,\overline{\gamma}(z))R\varphi \right> dz \right),
\end{multline}
where we used that, almost surely, $\overline{\mu}^{n}_{S_{z-}}=\overline{\mu}^{n}_{S_{z}}$ for almost every $z$ in $\mathbb{R}_{+}$.
Then, the second term in the right-hand side of~\eqref{eq:decomposition:eta:fail} rewrites as $M_{t}^{n}(\varphi)+\int_{0}^{t} A_{z}^{n}(\varphi) dz$.

It remains to show that $(M_{t}^{n}(\varphi))_{t\geq 0}$ is an $\mathbb{F}$-martingale.
Yet, for all $i=1,\dots,n$,
\begin{equation*}
\mathbb{E}\left[ \int_{0}^{t} \left| R\varphi\left(S^{n,i}_{z-}\right) \right| \lambda_{z}^{n,i}dz \right]\leq 2 ||\varphi||_{\infty} \mathbb{E}\left[ \int_{0}^{t} \lambda_{z}^{n,i}dz \right] = 2 ||\varphi||_{\infty} \mathbb{E}\left[ N^{n,i}_{t} \right] < +\infty,
\end{equation*}
and the $\mathbb{F}$-predictability of the age processes $(S^{n,i}_{t-})_{t\geq 0}$ gives the result (see \cite[II. T8]{Bremaud_PP}). Finally, the expression of the angle bracket \eqref{eq:bracket:M:n:(phi)} follows from standard computations for point processes (see \cite[Proposition II.4.1.]{gill_1997}).

\subsection{Proof of Proposition \ref{prop:control:-1:alpha:combined}}
\label{sec:proof:prop:control:-1:alpha:combined}

\paragraph{Proof of \hyperref[item:prop:i]{$(i)$}.}
Let $(\varphi_{k})_{k\geq 1}$ be an orthonormal basis of $\mathcal{W}^{1,\alpha}_{0}$ so that, in particular $||\eta^{n}_{t}||^{2}_{-1,\alpha}=\sum_{k\geq 1} \left<\eta^{n}_{t},\varphi_{k}\right>^{2}$.
Using the coupling \eqref{eq:definition:coupling:MFHawkes}-\eqref{eq:definition:coupling:limit:equation:fluctuations}, we have for every $k$ and $t\leq \theta$,
\begin{equation*}
\left<\eta^{n}_{t},\varphi_{k}\right>=\sqrt{n}\left( \frac{1}{n}\sum_{i=1}^{n} \varphi_{k}(S^{n,i}_{t}) - \mathbb{E}\left[ \varphi_{k}(\overline{S}^{i}_{t}) \right] \right)= S^{n}_{t}(\varphi_{k}) + T^{n}_{t}(\varphi_{k}),
\end{equation*}
where
\begin{equation*}
\begin{cases}
S^{n}_{t}(\varphi_{k}):= n^{-1/2} \sum_{i=1}^{n} \varphi_{k}(S^{n,i}_{t}) - \varphi_{k}(\overline{S}^{i}_{t})\\
T^{n}_{t}(\varphi_{k}):= n^{-1/2} \sum_{i=1}^{n} \varphi_{k}(\overline{S}^{i}_{t}) - \mathbb{E}[ \varphi_{k}(\overline{S}^{i}_{t}) ].
\end{cases}
\end{equation*}

On the one hand, using the independence of the age processes $(\overline{S}^{i}_{t})_{t\geq 0}$, we have
\begin{eqnarray*}
\mathbb{E}\left[ \sum_{k\geq 1} T^{n}_{t}(\varphi_{k})^{2} \right] &=& \sum_{k\geq 1} \mathbb{E}\left[ \frac{1}{n} \left( \sum_{i=1}^{n} \varphi_{k}(\overline{S}^{i}_{t}) - \mathbb{E}\left[ \varphi_{k}(\overline{S}^{i}_{t}) \right] \right)^{2} \right] \\
&= &  \sum_{k\geq 1} \mathbb{E}\left[ \left( \varphi_{k}(\overline{S}^{1}_{t}) - \mathbb{E}\left[ \varphi_{k}(\overline{S}^{1}_{t}) \right] \right)^{2} \right] \leq \sum_{k\geq 1} \mathbb{E}\left[ \left( \varphi_{k}(\overline{S}^{1}_{t})\right)^{2}\right] \\
&\leq & \mathbb{E}\left[ \sum_{k\geq 1} (\delta_{\overline{S}^{1}_{t}}(\varphi_{k}))^{2} \right] = \mathbb{E}\left[ ||\delta_{\overline{S}^{1}_{t}}||^{2}_{-1,\alpha} \right].
\end{eqnarray*}
Then, using Lemma \ref{lem:control:normes:D:H} and the fact that the age $\overline{S}^{1}_{t}$ is upper bounded by $M_{S_{0}}+t\leq M_{S_{0}}+\theta$ (thanks to (\hyperref[ass:initial:condition:density:compact:support]{$\mathcal{A}^{u_0}_{\infty}$}), remind \eqref{eq:control:age:uniform}), it follows that
\begin{equation*}
\mathbb{E}\left[ \sum_{k\geq 1} T^{n}_{t}(\varphi_{k})^{2} \right]\leq (C_{1})^{2} ( 1+(M_{S_{0}}+\theta)^{\alpha} )^{2}
\end{equation*}
and so $\sup_{n\geq 1}\sup_{t\in [0,\theta]} \mathbb{E}[\, \sum_{k\geq 1} T^{n}_{t}(\varphi_{k})^{2} \,] <+\infty$.

On the other hand, expanding the square and using exchangeability of the age processes $(S^{n,i}_{t})_{t\geq 0}$, one has
\begin{eqnarray}
\mathbb{E}\left[ \sum_{k\geq 1} S^{n}_{t}(\varphi_{k})^{2} \right] &=& n^{-1}  \mathbb{E}\left[ \sum_{k\geq 1} \left( \sum_{i=1}^{n} \varphi_{k}(S^{n,i}_{t}) - \varphi_{k}(\overline{S}^{i}_{t}) \right)^{2} \right] \nonumber\\
&= & (n-1)  \mathbb{E}\left[ \sum_{k\geq 1} (\varphi_{k}(S^{n,1}_{t}) - \varphi_{k}(\overline{S}^{1}_{t})) (\varphi_{k}(S^{n,2}_{t}) - \varphi_{k}(\overline{S}^{2}_{t})) \right] \nonumber\\
& & +   \mathbb{E}\left[ \sum_{k\geq 1} (\varphi_{k}(S^{n,1}_{t}) - \varphi_{k}(\overline{S}^{1}_{t}))^{2} \right]. \label{eq:step:two:1}
\end{eqnarray}
Since the ages $S^{n,1}_{t}$, $\overline{S}^{1}_{t}$, $S^{n,2}_{t}$ and $\overline{S}^{2}_{t}$ are upper bounded by $M_{S_{0}}+\theta$ and $(\varphi_{k}(x_{1}) - \varphi_{k}(x_{2})) (\varphi_{k}(y_{1}) - \varphi_{k}(y_{2}))=0$ as soon as $x_{1}=x_{2}$ or $y_{1}=y_{2}$, we have
\begin{multline}\label{eq:step:two:2}
\mathbb{E}\left[\sum_{k\geq 1} (\varphi_{k}(S^{n,1}_{t}) - \varphi_{k}(\overline{S}^{1}_{t})) (\varphi_{k}(S^{n,2}_{t}) - \varphi_{k}(\overline{S}^{2}_{t})) \right] \\
\leq \chi^{(2)}_{n}(\theta) \sup_{x,y \leq  M_{S_{0}}+\theta} \sum_{k\geq 1} |\varphi_{k}(x)-\varphi_{k}(y)|^{2},
\end{multline}
where $\chi^{(2)}_{n}(\theta)$ is defined by \eqref{eq:def:gamma:k:n}. Yet, since $(\varphi_{k})_{k\geq 1}$ is an orthonormal basis of $\mathcal{W}^{1,\alpha}_{0}$, we have
\begin{equation}\label{eq:step:two:3}
\begin{cases}
\sum_{k\geq 1} (\varphi_{k}(S^{n,1}_{t}) - \varphi_{k}(\overline{S}^{1}_{t}))^{2}=\sum_{k\geq 1} \Big< D_{S^{n,1}_{t},\overline{S}^{1}_{t}},\varphi_{k} \Big>^{2}= ||D_{S^{n,1}_{t},\overline{S}^{1}_{t}}||^{2}_{-1,\alpha}\\
\sup_{x,y \leq  M_{S_{0}}+\theta} \sum_{k\geq 1} |\varphi_{k}(x)-\varphi_{k}(y)|^{2}=\sup_{x,y \leq  M_{S_{0}}+\theta} ||D_{x,y}||^{2}_{-1,\alpha}.
\end{cases}
\end{equation}
Hence, using Lemma \ref{lem:control:normes:D:H} and once again the fact that the ages $S^{n,1}_{t}$ and $\overline{S}^{1}_{t}$ are upper bounded by $M_{S_{0}}+\theta$, we have, by gathering \eqref{eq:step:two:1}-\eqref{eq:step:two:3},
\begin{equation*}
\mathbb{E}\left[ \sum_{k\geq 1} S^{n}_{t}(\varphi_{k})^{2} \right] \leq (n-1) \chi^{(2)}_{n}(\theta) (C_{2})^{2} (1+(M_{S_{0}}+\theta)^{\alpha})^{2}  + (C_{2})^{2} (1+(M_{S_{0}}+\theta)^{\alpha})^{2},
\end{equation*}
and it follows from Proposition \ref{prop:rate:of:convergence:proba:k:tuple:different:ages} that $\sup_{n\geq 1}\sup_{t\in [0,\theta]} \mathbb{E}[ \sum_{k\geq 1} S^{n}_{t}(\varphi_{k})^{2} ] <+\infty$.

Finally, by convexity of the square function, $||\eta^{n}_{t}||^{2}_{-1,\alpha}\leq 2\sum_{k\geq 1} S^{n}_{t}(\varphi_{k})^{2} + T^{n}_{t}(\varphi_{k})^{2}$ so that \eqref{eq:control:eta:n:-1:alpha} follows from the two steps above.

\paragraph{Proof of \hyperref[item:prop:ii]{$(ii)$}.}
We first show \eqref{eq:control:M:n:-1:alpha} and then use it in order to prove that $(M^{n}_{t})_{t\geq 0}$ is càdlàg. Let $(\varphi_{k})_{k\geq 1}$ be an orthonormal basis of $\mathcal{W}^{1,\alpha}_{0}$ composed of $\mathcal{C}^{\infty}$ functions with compact support. For all $k\geq 1$, the test function $\varphi_{k}$ belongs to $\mathcal{C}^{1}_{b}$ so that $(M^{n}_{t}(\varphi_{k}))_{t\geq 0}$ is an $\mathbb{F}$-martingale (Proposition \ref{prop:decomposition:eta}). Using Doob's inequality for real-valued martingales \cite[Theorem 1.43.]{Jacod_2003} and Equation \eqref{eq:bracket:M:n:(phi)}, one has
\begin{multline*}
\mathbb{E}\left[\sup_{t\in [0,\theta]} ||M^{n}_{t}||^{2}_{-1,\alpha} \right] \leq \sum_{k\geq 1} \mathbb{E}\left[ \sup_{t\in [0,\theta]} M^{n}_{t}(\varphi_{k})^{2} \right] \\
\leq C \sum_{k\geq 1} \mathbb{E}\left[  M^{n}_{\theta}(\varphi_{k})^{2} \right] \leq C || \Psi ||_{\infty} \mathbb{E}\left[ \int_{0}^{\theta} \sum_{k\geq 1} R\varphi_{k} (S^{n,1}_{z-})^{2} dz \right],
\end{multline*}
where the last inequality comes from exchangeability and boundedness of the intensity. Noticing that $R\varphi_{k} (S^{n,1}_{z-})=D_{0,S^{n,1}_{z-}}(\varphi_{k})$ and then using Lemma \ref{lem:control:normes:D:H} as we have done in the proof of $(i)$, it follows that
\begin{equation*}
\mathbb{E}\left[ \int_{0}^{\theta} \sum_{k\geq 1} R\varphi_{k} (S^{n,1}_{z-})^{2} dz \right] \leq (C_{2})^{2} \int_{0}^{\theta} (1 + (M_{S_{0}}+\theta)^{\alpha})^{2}dz,
\end{equation*}
which does not depend on $n$ and gives \eqref{eq:control:M:n:-1:alpha}. Moreover, gathering the integrability property given by \eqref{eq:control:M:n:-1:alpha} and the fact that, for all $k\geq 1$, the process $(M^{n}_{t}(\varphi_{k}))_{t\geq 0}$ is an $\mathbb{F}$-martingale, we have that $M^{n}$ is a $\mathcal{W}^{-1,\alpha}_{0}$-valued $\mathbb{F}$-martingale.

It remains to show that $(M^{n}_{t})_{t\geq 0}$ is càdlàg. First remark that for any $k$, the $\mathbb{F}$-martingale $(M^{n}_{t}(\varphi_{k}))_{t\geq 0}$ is càdlàg. Let $\varepsilon>0$ and $t_{0}>0$. For any $n\geq 1$,
\begin{equation*}
\mathbb{E}\left[ \sum_{k\geq 1} \sup_{t\in [0,t_{0}+1]}  M^{n}_{t}(\varphi_{k})^{2} \right] <+\infty,
\end{equation*}
so there exists a set $\Omega^{n}$ such that $\mathbb{P}(\Omega^{n})=1$ and for all $\omega$ in $\Omega^{n}$,
\begin{equation*}
\sum_{k\geq 1} \sup_{t\in [0,t_{0}+1]}  \left<M^{n}_{t}(\omega),\varphi_{k}\right>^{2}<+\infty.
\end{equation*}

Once $\omega$ is fixed in $\Omega^{n}$, there exists an integer $k_{0}$ (which depends on $\omega$) such that $\sum_{k>k_{0}} \sup_{t\in [0,t_{0}+1]}  \left<M^{n}_{t}(\omega),\varphi_{k}\right>^{2} <\varepsilon$.
Let $t$ be such that $t_{0}<t\leq t_{0}+1$, using the right continuity of $t\mapsto \left<M^{n}_{t}(\omega),\varphi_{k}\right>$, we have, dropping $\omega$ for simplicity of notations,
\begin{eqnarray*}
||M^{n}_{t} - M^{n}_{t_{0}}||^{2}_{-1,\alpha} &=& \sum_{k\geq 1} (M^{n}_{t}(\varphi_{k}) - M^{n}_{t_{0}}(\varphi_{k}))^{2}\\
&\leq & \sum_{k=1}^{k_{0}} (M^{n}_{t}(\varphi_{k}) - M^{n}_{t_{0}}(\varphi_{k}))^{2} + 2 \sum_{k>k_{0}} [ M^{n}_{t}(\varphi_{k})^{2} + M^{n}_{t_{0}}(\varphi_{k})^{2} ]\\
& \leq & \sum_{k=1}^{k_{0}} \varepsilon + 4\varepsilon = (k_{0}+4)\varepsilon,
\end{eqnarray*}
as soon as $|t-t_{0}|$ is small enough. Hence, $t\mapsto M^{n}_{t}(\omega)$ is right continuous with values in $\mathcal{W}^{-1,\alpha}_{0}$.
In the same way, let $(t_{m})_{m\geq 1}$ be a sequence such that $t_{m}<t_{0}$ and $t_{m}\to t_{0}$. For any integers $m$ and $\ell$, we have, dropping $\omega$ for simplicity of notations,
\begin{eqnarray*}
||M^{n}_{t_{m}} - M^{n}_{t_{\ell}}||^{2}_{-1,\alpha} &=& \sum_{k\geq 1} (M^{n}_{t_{m}}(\varphi_{k}) - M^{n}_{t_{\ell}}(\varphi_{k}))^{2}\\
&\leq & \sum_{k=1}^{k_{0}} (M^{n}_{t_{m}}(\varphi_{k}) - M^{n}_{t_{\ell}}(\varphi_{k}))^{2} + 4\varepsilon.
\end{eqnarray*}
Yet, for all $k=1,\dots ,k_{0}$, the sequence $(M^{n}_{t_{m}}(\varphi_{k}))_{m\geq 1}$ is convergent hence Cauchy. It follows that $(M^{n}_{t_{m}}(\omega))_{m\geq 1}$ is a Cauchy sequence and so converges in $\mathcal{W}^{-1,\alpha}_{0}$. Hence, $t\mapsto M^{n}_{t}(\omega)$ admits left limits in $\mathcal{W}^{-1,\alpha}_{0}$. Finally, $t\mapsto M^{n}_{t}$ belongs to $\mathcal{D}(\mathbb{R}_{+},\mathcal{W}^{-1,\alpha}_{0})$ almost surely.

\paragraph{Proof of \hyperref[item:prop:iii]{$(iii)$}.}
Starting from \eqref{eq:An:function:of:Gamma}, we have, by convexity of the square function,
\begin{equation*}
A^{n}_{t}(\varphi)^{2} \leq 2 \Big(\left< \overline{\mu}^{n}_{S_{t}},\frac{\partial \Psi}{\partial y}(\cdot,\overline{\gamma}(t))R\varphi \right>^{2}  (\Gamma^{n}_{t-})^{2} +R^{n,(1)}_{t}(\varphi)^{2} \Big)
\end{equation*}
Let $(\varphi_{k})_{k\geq 1}$ be an orthonormal basis of $\mathcal{W}^{2,\alpha}_{0}$ so that $||A^{n}_{t}||^{2}_{-2,\alpha}=\sum_{k\geq 1} A^{n}_{t}(\varphi_{k})^{2}$. Noticing that $R\varphi_{k} (S^{n,i}_{t-})=D_{0,S^{n,i}_{t-}}(\varphi_{k})$ and then using Lemma \ref{lem:control:normes:D:H} as we have done in the proof of $(i)$, it follows that
\begin{eqnarray*}
\sum_{k\geq 1} \left< \overline{\mu}^{n}_{S_{t}},\frac{\partial \Psi}{\partial y}(\cdot,\overline{\gamma}(t))R\varphi_{k} \right>^{2} &\leq & {\rm Lip}(\Psi)^{2} \frac{1}{n} \sum_{i=1}^{n} \left( \sum_{k\geq 1} R\varphi_{k}(S^{n,i}_{t-})^{2} \right) \\
& \leq & {\rm Lip}(\Psi)^{2} (C_{2})^{2}(1+ (M_{S_{0}}+\theta)^{\alpha})^{2},
\end{eqnarray*}
and in the same way,
\begin{equation*}
\sum_{k\geq 1} R^{n,(1)}_{t}(\varphi_{k})^{2} \leq {\rm Lip}(\Psi)^{2} (C_{2})^{2}(1+ (M_{S_{0}}+\theta)^{\alpha})^{2} \frac{1}{n} \sum_{i=1}^{n} (\sqrt{n}r^{n,i}_{t})^{2}.
\end{equation*}
Hence,
\begin{equation*}
\sum_{k\geq 1} A^{n}_{t}(\varphi_{k})^{2}\leq 2 {\rm Lip}(\Psi)^{2} (C_{2})^{2}(1+ (M_{S_{0}}+\theta)^{\alpha})^{2} \left( (\Gamma^{n}_{t-})^{2} + \frac{1}{n} \sum_{i=1}^{n} (\sqrt{n}r^{n,i}_{t})^{2} \right).
\end{equation*}
Yet, as a consequence of Proposition \ref{prop:rate:of:convergence:proba:k:tuple:different:ages}, $\xi^{(2)}_{n}(t)=\mathbb{E}[  |\Gamma^{n}_{t-}| ^{2} ]/n \lesssim_{\theta} n^{-1}$ and $\xi^{(4)}_{n}(t)=\mathbb{E}[  |\Gamma^{n}_{t}| ^{4} ]/n^{2}\lesssim_{\theta} n^{-2}$. In particular, uniformly in $t\leq \theta$, the $L^{1}$ norm of $(\Gamma^{n}_{t-})^{2}$ is of order $1$ while the $L^{1}$ norm of the rest term satisfies
$$
\frac{1}{n} \sum_{i=1}^{n} (\sqrt{n}r^{n,i}_{t})^{2} \leq n \Big( \sup_{s,y} |\frac{\partial^{2} \Psi}{\partial y^{2}}(s,y)| \Big)^{2} |\gamma^{n}_{t}-\overline{\gamma}(t)|^{4}/4
$$
and so vanishes to $0$ as $n$ goes to infinity. Hence,
\begin{equation*}
\sup_{n\geq 1} \sup_{t\in [0,\theta]} \mathbb{E}\left[ ||A^{n}_{t}||^{2}_{-2,\alpha} \right] = \sup_{n\geq 1} \sup_{t\in [0,\theta]} \mathbb{E}\bigg[ \sum_{k\geq 1} A^{n}_{t}(\varphi_{k})^{2} \bigg] <+\infty.
\end{equation*}

\paragraph{Proof of \hyperref[item:prop:iv]{$(iv)$}.}
By definition of $L_{z}$ and the triangular inequality, 
$$|| L_{z} \varphi ||^{2}_{1,\alpha}\leq 2(||\varphi'||^{2}_{1,\alpha} + || \Psi(\cdot,\overline{\gamma}(z)) R\varphi||^{2}_{1,\alpha}).$$
Firstly, $||\varphi'||^{2}_{1,\alpha}\leq ||\varphi||^{2}_{2,\alpha}$.
Secondly, by Lemma \ref{lem:control:for:linear:operator}, for all $z\leq \theta$,
$$|| \Psi(\cdot,\overline{\gamma}(z)) R\varphi||^{2}_{1,\alpha}\leq C \sup_{z\in [0,\theta]} || \Psi(\cdot,\overline{\gamma}(z))||_{\mathcal{C}^{1}_{b}}^{2} ||R\varphi||^{2}_{1,\alpha}.$$
Finally, \eqref{eq:control:L(f)} follows from \eqref{eq:Psi:dPsi:C2b:locally:in:time} and the continuity of the mapping $R$ (Lemma \ref{lem:R:continuous}).

\subsection{Proof of Proposition \ref{prop:convergence:Wn:W}}
\label{sec:proof:prop:convergence:Wn:W}

As stated in Equation \eqref{eq:W:n:tight}, the sequence $(W^{n})_{n\geq 1}$ is tight. Then, let us consider the following decomposition, for any $\varphi_{1}$ and $\varphi_{2}$ in $\mathcal{W}^{2,\alpha}_{0}$,
\begin{equation*}
\left< {<\!\!<\!W^{n}\!>\!\!>}_{t} (\varphi_{1}), \varphi_{2} \right> - \int_{0}^{ t} \left<u_{z},\varphi_{1}\varphi_{2}\Psi(\cdot,\overline{\gamma}(z)\right> dz= B^{n}_{t}+C^{n}_{t},
\end{equation*}
with
\begin{equation*}
\begin{cases}
\displaystyle B^{n}_{t}:= \frac{1}{n}\sum_{i=1}^{n} \int_{0}^{t} \varphi_{1}(S^{n,i}_{z-})\varphi_{2}(S^{n,i}_{z-}) \left(\lambda_{z}^{n,i}-\Psi(S^{n,i}_{z-},\overline{\gamma}(z)) \right)dz,\\
\displaystyle C^{n}_{t}:= \int_{0}^{t} \left<\overline{\mu}^{n}_{S_{z}}-u_{z},\varphi_{1}\varphi_{2}\Psi(\cdot,\overline{\gamma}(z)) \right>dz,
\end{cases}
\end{equation*}
where we used the fact that, almost surely, $\overline{\mu}^{n}_{S_{z-}}=\overline{\mu}^{n}_{S_{z}}$ for almost every $z$ in $\mathbb{R}_{+}$.
The first term $B^{n}$ converges in $L^{1}$ to $0$ by using the Lipschitz continuity of $\Psi$ and the convergence of $\gamma^{n}$ to $\overline{\gamma}$ given by Proposition \ref{prop:rate:of:convergence:proba:k:tuple:different:ages}.
From the convergence 
$$\frac{1}{n}\sum_{i = 1}^n \delta_{(S^{n,i}_{t})_{t\geq 0}}\xrightarrow[n\rightarrow \infty]{} \mathcal{L}\big(( \overline{S}^{1}_{t})_{t\geq 0}\big),$$
which is given in \cite[Corollary IV.4]{chevallier2015mean}, one can deduce that for almost every $z$, 
$$\frac{1}{n}\sum_{i=1}^{n} \delta_{S^{n,i}_{z}}\xrightarrow[n\rightarrow \infty]{} u_{z}$$
(see for instance \cite[Proposition VI.3.14 and Lemma VI.3.12]{Jacod_2003}). Then, dominated convergence implies that the second term $C^{n}$ converges in expectation to $0$. Hence, the bracket of $W^{n}$ \eqref{eq:bracket:W:n} converges to the covariance \eqref{eq:covariance:W} for $t'=t$.

Furthermore, as for $M^{n}$ (see the proof of Remark \ref{rem:C:tight}), the maximum jump size of $W^{n}$ converges to $0$.
Hence, Rebolledo's central limit theorem for local martingales \cite{Rebolledo_80} gives, for every $\varphi_{1},\dots ,\varphi_{k}$ in $\mathcal{W}^{2,\alpha}_{0}$ and $t_{1},\dots ,t_{k}\geq 0$, the convergence of $(W^{n}_{t_{1}}(\varphi_{1}),\dots ,W^{n}_{t_{k}}(\varphi_{k}))$ to a Gaussian vector with the prescribed covariance \eqref{eq:covariance:W}. The limit law of $(W^{n})_{n\geq 1}$ is then characterized as the law of a continuous Gaussian process with covariance \eqref{eq:covariance:W}.

\subsection{Proof of Corollary \ref{cor:convergence:couple}}
\label{sec:proof:cor:convergence:couple}

First, the tightness (and convergence) of $(R^{*}W^{n})_{n\geq 1}$ comes from the continuity of $R^{*}$ as a mapping from $\mathcal{W}^{-2,\alpha}_{0}$  to $\mathcal{W}^{-2,\alpha}_{0}$ which comes from the continuity of $R$ as a mapping from $\mathcal{W}^{2,\alpha}_{0}$ to $\mathcal{W}^{2,\alpha}_{0}$ (Lemma \ref{lem:R:continuous}). Then, let us show that $(V^{n})_{n\geq 1}$ is tight in $\mathcal{D}(\mathbb{R}_{+},\mathbb{R})$.\\

Assume that $h(0)=0$ and extend the function $h$ to the whole real line by the value $0$ on the negative real numbers.

-{\it (i)} For all $n\geq 1$, $V^{n}_{0}=0$ a.s. so $(V^{n}_{0})_{n\geq 1}$ is clearly tight. 

For any $t>r\geq 0$, since $h(r-z)=0$ as soon as $z\geq r$, one has
\begin{equation*}
V^{n}_{t}-V^{n}_{r}= \int_{0}^{t} \left[ h(t-z)-h(r-z) \right] dW^{n}_{z}({\bf 1}).
\end{equation*}
Let us denote, for all $x\geq 0$, $\texttt{V}_{r,t}^{n}(x)=\int_{0}^{x}\left[h\left(t-z\right)-h\left(r-z\right)\right] dW^{n}_{z}({\bf 1}).$ It is a martingale with respect to $x$. Burkholder-Davis-Gundy inequality \cite[p. 894]{shorack2009empirical} gives the existence of a universal constant $C_{p}$ such that
$$
\mathbb{E}\left[ \sup_{x\leq t} \left|\texttt{V}_{r,t}^{n}(x)\right|^{2p} \right] \leq C_{p} \mathbb{E}\left[ \left[\texttt{V}_{r,t}^{n}\right]_{t}^{p} \right].
$$
Yet, the quadratic variation of $\texttt{V}_{r,t}^{n}$ is given by
\begin{equation}\label{eq:quadratic:variation:C:n}
\left[\texttt{V}_{r,t}^{n}\right]_{x} = \frac{1}{n}\sum_{j=1}^{n} \int_{0}^{x}\left[h\left(t-z\right)-h\left(r-z\right)\right]^{2} N^{n,j}(dz) \leq {\rm H\ddot{o}l}(h)^{2} \left|t-r\right|^{2\beta(h)} \frac{1}{n}\sum_{j=1}^{n} N_{x}^{n,j}.
\end{equation}
So, using the exchangeability, we have for all $p\geq 0$,
\begin{equation}\label{eq:control:kolmogorov:V:n}
\mathbb{E}\left[ |V^{n}_{t}-V^{n}_{r}|^{2p} \right] \leq C_{p} {\rm H\ddot{o}l}(h)^{2p} \left|t-r\right|^{2\beta(h) p} \mathbb{E}\left[ |N_{t}^{n,1}|^{p} \right].
\end{equation}
Yet, the intensity of $N^{n,1}$ is bounded so that $N^{n,1}$ is stochastically dominated by a Poisson process with intensity $||\Psi||_{\infty}$. Hence, $\mathbb{E}[ |N^{n,1}_{t}|^{p} ]\leq \mathbb{E}[{\rm Poiss}(t||\Psi||_{\infty})^{p}]$ where ${\rm Poiss}(t||\Psi||_{\infty})$ is a Poisson variable with parameter $t||\Psi||_{\infty}$. This implies that $\mathbb{E}[ |N^{n,1}_{t}|^{p} ]$ is bounded uniformly in $n$ by a locally bounded function of the time $t$, say $\tilde{C}_{p}(t)$ (which can be assumed to be increasing continuous without any loss of generality).
Then, taking $p=1$, $t=\delta$ and $r=0$ and using Markov's inequality gives

-{\it (ii)} for all $\varepsilon>0$, $\lim_{\delta\to 0} \limsup_{n} \mathbb{P}(|V^{n}_{\delta}-V^{n}_{0}|>\varepsilon)=0$.

\noindent Finally, taking $p=1/\beta(h)$ and using Markov's inequality gives

-{\it (iii)} for all $\nu>0$, $\mathbb{P}(|V^{n}_{t}-V^{n}_{r}|>\nu)\leq \nu^{-2/\beta(h)} |F(t)-F(s)|^{2}$,

\noindent where $F(t):=(C_{\frac{1}{\beta(h)}}\tilde{C}_{\frac{1}{\beta(h)}}(t) {\rm H\ddot{o}l}(h)^{2/\beta(h)} )^{1/2} \,t$ defines an increasing continuous function.

Hence, $(i)$, $(ii)$ and $(iii)$ allow to apply Billingsley's criterion for tightness \cite[Theorem VI.4.1]{Jacod_2003} to deduce that $(V^{n})_{n\geq 1}$ is tight.\\

Now, if $h(0)\neq 0$, one can use the following decomposition,
$$
V_{t}^{n}=\int_{0}^{t}\left[h\left(t-z\right)-h\left(0\right)\right] dW^{n}_{z}({\bf 1}) + h\left(0\right) W^{n}_{t}({\bf 1}).
$$
The first term is tight thanks to what we have done in the case $h\left(0\right)=0$ whereas the second term is converging since $(W^{n})_{n}$ is converging, whence $(V^{n})_{n\geq 1}$ is tight.\\

Now, since the limit trajectories of $R^{*}W^{n}$ are continuous, $(R^{*}W^{n},V^{n})_{n\geq 1}$ is tight in $\mathcal{D}(\mathbb{R}_{+},\mathcal{W}^{-2,\alpha}_{0}\times \mathbb{R})$. It now suffices to characterize the limiting finite dimensional distributions. Recall that $V_{t}=\int_{0}^{t}h(t-z)dW_{z}({\bf 1})$ and denote by $(t_{1},\dots ,t_{k})$ a $k$-tuple of positive times.

\begin{sloppypar}
First, suppose that $h$ is piecewise constant. In that case, the convergence of $W^{n}$ towards $W$ easily implies the convergence of $((R^{*}W^{n}_{t_{1}},V^{n}_{t_{1}}),\dots ,(R^{*}W^{n}_{t_{k}},V^{n}_{t_{k}}))$ to $((R^{*}W_{t_{1}},V_{t_{1}}),\dots ,(R^{*}W_{t_{k}},V_{t_{k}}))$ (use the fact that $h$ is a piecewise function to write $V^{n}_{t_{i}}$ as a sum of increments of $W^{n}({\bf 1})$).
\end{sloppypar}

Then, since $h$ is continuous, one can find, for each $\varepsilon>0$, a  piecewise constant function $h^{\varepsilon}$ such that $||h-h^{\varepsilon}||_{\infty}\leq\varepsilon$.
Denote $V^{n,\varepsilon}_{t}:=\int_{0}^{t} h^{\varepsilon}(t-z) dW^{n}_{z}({\bf 1})$ and notice that $\mathbb{E}\left[ |V^{n}_{t}-V^{n,\varepsilon}_{t}|^{2} \right]\leq 2\varepsilon^{2} \mathbb{E}\left[ <W^{n}({\bf 1})>_{t} \right]\leq 2\varepsilon^{2}||\Psi||_{\infty} t\to 0$ as $\varepsilon\to 0$. In the same way, denote $V^{\varepsilon}_{t}:=\int_{0}^{t} h^{\varepsilon}(t-z) dW_{z}({\bf 1})$ and remark that $\mathbb{E}\left[ |V_{t}-V^{\varepsilon}_{t}|^{2} \right]\leq 2\varepsilon^{2}||\Psi||_{\infty} t\to 0$ as $\varepsilon\to 0$. Yet, the previous point gives the convergence, in terms of finite dimensional distributions, of $V^{n,\varepsilon}$ to $V^{\varepsilon}$ for all $\varepsilon>0$ so the convergence, in terms of finite dimensional distributions, of $V^{n}$ to $V$ follows which ends the proof.

\subsection{Proof of Proposition \ref{prop:tightness:Gamma:n}}
\label{sec:proof:prop:tightness:Gamma:n}

The idea is to use \eqref{eq:closed:equation:n:2}. The first step is to simplify \eqref{eq:closed:equation:n:2} by using the following convergences
\begin{equation}\label{eq:convergence:to:0:rest:terms:equation:Gamma}
\begin{cases}
\displaystyle \mathbb{E}\left[\left| \sup_{t\in [0,\theta]} \int_{0}^{t} h(t-z) R^{n,(2)}_{z}dz \right|\right] \to 0,\\
\displaystyle \mathbb{E}\left[\left| \sup_{t\in [0,\theta]} \int_{0}^{t} h(t-z) \left< \overline{\mu}^{n}_{S_{z}}-u_{z} ,\frac{\partial \Psi}{\partial y}(\cdot,\overline{\gamma}(z))\right> \Gamma^{n}_{z}dz \right|\right] \to 0.
\end{cases}
\end{equation}
These two convergences follow from the two following claims: by \eqref{eq:power:Gamma:n:locally:bounded},
$$\sup_{z\in [0,\theta]} \mathbb{E}\left[ \left| R^{n,(2)}_{z} \right| \right] \leq {\rm Lip}(\Psi) C n^{-1/2} \sup_{z\in [0,\theta]} \mathbb{E}\left[ |\Gamma^{n}_{z-}|^{2} \right] \to 0,$$
and, by Cauchy-Schwarz inequality,
\begin{equation}\label{eq:convergence:ou:on:utilise:CS}
\sup_{z\in [0,\theta]} \mathbb{E}\left[ \left| \left< \overline{\mu}^{n}_{S_{z}}-u_{z},\frac{\partial \Psi}{\partial y}(\cdot,\overline{\gamma}(z))\right> \Gamma^{n}_{z} \right| \right] \to 0.
\end{equation}
Indeed,
\begin{multline*}
\mathbb{E}\left[ \left| \left< \overline{\mu}^{n}_{S_{z}}-u_{z},\frac{\partial \Psi}{\partial y}(\cdot,\overline{\gamma}(z))\right> \Gamma^{n}_{z} \right| \right] \leq \mathbb{E}\left[ \left| \frac{1}{\sqrt{n}} \left< \eta^{n}_{z}, \frac{\partial \Psi}{\partial y}(\cdot,\overline{\gamma}(z)) \right> \right|^{2} \right]^{1/2} \mathbb{E}\left[ |\Gamma^{n}_{z}|^{2} \right]^{1/2}\\
\leq  \frac{1}{\sqrt{n}} \mathbb{E}\left[ ||\eta^{n}_{z}||^{2}_{-1,\alpha} \right]^{1/2} \left\Vert \frac{\partial \Psi}{\partial y}(\cdot,\overline{\gamma}(z)) \right\Vert_{1,\alpha} \mathbb{E}\left[ |\Gamma^{n}_{z}|^{2} \right]^{1/2},
\end{multline*}
for any $\alpha>1/2$ and \eqref{eq:convergence:ou:on:utilise:CS} follows from Proposition \ref{prop:control:-1:alpha:combined}-$(i)$ and Equations \eqref{eq:f:k:alpha:leq:f:Ckb}, \eqref{eq:Psi:dPsi:C2b:locally:in:time} and \eqref{eq:power:Gamma:n:locally:bounded}.\\

Return to \eqref{eq:closed:equation:n:2}. The right-hand side is tight since it is convergent  (Corollary \ref{cor:convergence:couple}) and the last term in the left hand side is tight since $(\eta^{n})_{n\geq 1}$ is tight (with continuous limit points) and $\eta\mapsto \int_{0}^{t} h(t-z) \left< \eta_{z}, \Psi(\cdot,\overline{\gamma}(z)\right> dz$ is continuous at every point $\eta$ in $\mathcal{C}(\mathbb{R}_{+},\mathcal{W}^{-2,\alpha}_{0})$ thanks to Lemma \ref{lem:continuity:mappings:eta} (remind \eqref{eq:f:k:alpha:leq:f:Ckb} and \eqref{eq:Psi:dPsi:C2b:locally:in:time}). Moreover, the term in the middle may be simplified by means of \eqref{eq:convergence:to:0:rest:terms:equation:Gamma}. Hence it remains to prove the tightness of the sequence of continuous processes $(I^{n})_{n\geq 1}$ defined, for all $t\geq 0$, by
\begin{equation*}
I^{n}_{t}:=\int_{0}^{t} h(t-z) \left< u_{z}, \frac{\partial \Psi}{\partial y}(\cdot,\overline{\gamma}(z)) \right> \Gamma^{n}_{z} dz.
\end{equation*}
We use Aldous criterion \cite[Theorem 16.10.]{Billing_Convergence}, that is the simplified version of the one stated on page \pageref{sec:tightness:criterion} but for real valued processes. First, for all $\theta\geq 0$,
\begin{equation*}
\mathbb{E}\left[ \sup_{t\in [0,\theta]} |I^{n}_{t}| \right]\leq h_{\infty}(\theta) {\rm Lip}(\Psi) \int_{0}^{\theta} \mathbb{E}\left[ |\Gamma^{n}_{z}| \right] dz,
\end{equation*}
is bounded uniformly with respect to $n$ thanks to Equation \eqref{eq:power:Gamma:n:locally:bounded}. And Markov's inequality implies that, for every $\theta\geq 0$ and $\varepsilon>0$, there exists $a>0$ such that 
$$\sup_{n\geq 1} \mathbb{P}\left( \sup_{t\in[0,\theta]} |I^{n}_{t}| \geq a \right)\leq \varepsilon,$$ 
which is the standard compactness condition.

Then, for the Aldous criterion, let us consider $\delta_{0}>0$, $\delta\leq \delta_{0}$ and for all $n\geq 1$, an $\mathbb{F}$-stopping time smaller than $\theta$ denoted by $\tau_{n}$. Assume for a while that $h(0)=0$ and extend the function $h$ to the whole real line by setting $0$ on the negative real numbers. As for Equation \eqref{eq:quadratic:variation:C:n}, we have
\begin{equation*}
|I^{n}_{\tau_{n}+\delta}-I^{n}_{\tau_{n}}| \leq {\rm H\ddot{o}l}(h) \delta_{0}^{\beta(h)} {\rm Lip}(\Psi) \int_{0}^{\theta+\delta_{0}} |\Gamma^{n}_{z}| dz.
\end{equation*}
Hence, as before, \eqref{eq:power:Gamma:n:locally:bounded} implies that $\sup_{n\geq 1} \mathbb{E}[|I^{n}_{\tau_{n}+\delta}-I^{n}_{\tau_{n}}|]\leq C (\theta+\delta_{0}) \delta_{0}^{\beta(h)}$ which is arbitrary small for $\delta_{0}$ small enough and Markov's inequality gives that, for any $\varepsilon_{1},\varepsilon_{2}>0$, there exists $\delta_{0}$ such that $\sup_{n\geq 1} \sup_{\delta\leq \delta_{0}} \mathbb{P}\left( |I^{n}_{\tau_{n}+\delta}-I^{n}_{\tau_{n}}| \geq \varepsilon_{1} \right)\leq \varepsilon_{2},$ that is Aldous criterion. Hence, $(I^{n})_{n\geq 1}$ is tight in $\mathcal{C}(\mathbb{R}_{+},\mathbb{R})$.\\

Now, if $h(0)\neq 0$, one can use the following decomposition,
$$
I^{n}(t)=\int_{0}^{t} (h(t-z)-h(0))\left< u_{z}, \frac{\partial \Psi}{\partial y}(\cdot,\overline{\gamma}(z)) \right> \Gamma^{n}_{z} dz+ h(0) \int_{0}^{t} \left< u_{z}, \frac{\partial \Psi}{\partial y}(\cdot,\overline{\gamma}(z)) \right> \Gamma^{n}_{z} dz.
$$
The first term is tight thanks to what we have done in the case $h(0)=0$ whereas the tightness of the second one is simpler and left to the reader (use Equation \eqref{eq:power:Gamma:n:locally:bounded}).\\

It only remains to check that the limit points are continuous. The idea is the same as for Remark \ref{rem:C:tight}. According to \cite[Theorem 13.4.]{Billing_Convergence}, it suffices to prove that for all $\theta\geq 0$, the maximal jump size of $\Gamma^{n}$ on $[0,\theta]$ converges to $0$. Yet, using the continuity of $\overline{\gamma}$ and the reverse triangle inequality, we have
\begin{equation*}
\Delta \Gamma^{n}_{t}:= |\Gamma^{n}_{t} - \Gamma^{n}_{t-}|\leq \sqrt{n} |\gamma^{n}_{t+}-\gamma^{n}_{t}|,
\end{equation*}
where we remind the definition $\gamma^{n}_{t}=n^{-1}  \sum_{j=1}^n \int_{0}^{t-} h(t-z)N^{n,j}(dz)$ and we define $\gamma^{n}_{t+}:=n^{-1}  \sum_{j=1}^n \int_{0}^{t} h(t-z)N^{n,j}(dz)$. Now, Assumption (\hyperref[ass:h:Holder]{$\mathcal{A}^{h}_{\rm H\ddot{o}l}$}) implies that $h$ is continuous and so we deduce that
\begin{equation*}
\Delta \Gamma^{n}_{t}\leq \sqrt{n} \frac{1}{n} \sum_{i=1}^{n} h(0) \mathds{1}_{t\in N^{n,i}}.
\end{equation*}
Since almost surely there is no common point to any two of the point processes $(N^{n,i})_{i=1,\dots ,n}$, there is, almost surely, for all $t\geq 0$, at most one of the $\mathds{1}_{t\in N^{n,i}}$ which is non null. Hence, $\sup_{t\in[0,\theta]} \Delta \Gamma^{n}_{t}\leq h(0) n^{-1/2}$ a.s., which gives the desired convergence to $0$.

\subsection{Proof of Theorem \ref{thm:limit:equation}}
\label{sec:proof:thm:limit:equation}

First, let us notice that we use the two following statements whose proofs are similar to those of Proposition \ref{prop:control:-1:alpha:combined}-$(iv)$ and Remark \ref{rem:control:L(f):etoile}:
for any $\alpha>1/2$ and $\theta\geq 0$,
\begin{equation}\label{eq:control:Lz:3:alpha:to:2:alpha}
\text{for all $\varphi$ in $\mathcal{W}^{3,\alpha}_{0}$,} \quad \sup_{z\in [0,\theta]} \frac{|| L_{z} \varphi ||^{2}_{2,\alpha}}{|| \varphi ||^{2}_{3,\alpha}}<+\infty,
\end{equation}
\begin{equation}\label{eq:control:Lz:etoile:-2,alpha:to:-3,alpha}
\text{and, for all $w$ in $\mathcal{W}^{-2,\alpha}_{0}$,} \quad \sup_{z\in [0,\theta]} \frac{|| L_{z}^{*} w ||^{2}_{-3,\alpha}}{|| w ||^{2}_{-2,\alpha}}<+\infty.
\end{equation}\\

As a consequence of tightness and continuity of the limit trajectories, we have tightness of the process $(\eta^{n},\Gamma^{n},W^{n},V^{n})_{n\geq 1}$ in $\mathcal{D}(\mathbb{R}_{+},\mathcal{W}^{-2,\alpha}\times\mathbb{R}\times\mathcal{W}^{-2,\alpha}_{0}\times \mathbb{R})$. Hence, let us assume without loss of generality that the sequence converges to $(\eta,\Gamma,W,V)$ in $\mathcal{D}(\mathbb{R}_{+},\mathcal{W}^{-2,\alpha}\times\mathbb{R}\times\mathcal{W}^{-2,\alpha}_{0}\times \mathbb{R})$.

Then, let $(\varphi_{k})_{k\geq 1}$ be an orthonormal basis of $\mathcal{W}^{3,\alpha}_{0}$ and define the following applications: for all $k\geq 1$, $F_{k}:\mathcal{D}(\mathbb{R}_{+},\mathcal{W}^{-2,\alpha}\times\mathbb{R}\times\mathcal{W}^{-2,\alpha}_{0}) \to \mathcal{D}(\mathbb{R}_{+},\mathbb{R})$ satisfy for all $t\geq 0$,
\begin{multline*}
F_{k}(f^{1},f^{2},f^{3})(t):= \left< f^{1}_{t},\varphi_{k} \right>-\left< f^{1}_{0},\varphi_{k} \right>-\int_{0}^{t} \left< f^{1}_{z},L_{z}\varphi_{k} \right> dz \\
- \int_{0}^{t} \left< u_{z},R\varphi_{k}\,\frac{\partial \Psi}{\partial y}(\cdot,\overline{\gamma}(z)) \right> f^{2}_{z} dz - f^{3}_{t}(R\varphi_{k}),
\end{multline*}
and $G:\mathcal{D}(\mathbb{R}_{+},\mathcal{W}^{-2,\alpha}\times\mathbb{R}\times\mathbb{R}) \to \mathcal{D}(\mathbb{R}_{+},\mathbb{R})$ satisfy for all $t\geq 0$,
\begin{multline*}
G(g^{1},g^{2},g^{3})(t):= g^{2}_{t} - \int_{0}^{t} h(t-z) \left< g^{1}_{z}, \Psi(\cdot,\overline{\gamma}(z)\right> dz \\
- \int_{0}^{t} h(t-z) \left< u_{z},\frac{\partial \Psi}{\partial y}(\cdot,\overline{\gamma}(z))\right> g^{2}_{z} dz - g^{3}_{t}.
\end{multline*}
Notice that the system \eqref{eq:closed:equation:limit:1}-\eqref{eq:closed:equation:limit:2} is equivalent to 
\begin{equation}\label{eq:Fk:G:equivalent:system}
\begin{cases}
\forall k\geq 1, \quad F_{k}(\eta,\Gamma,W)=0\\
G(\eta,\Gamma,V)=0.
\end{cases}
\end{equation}

\paragraph{Step one.}
Let us show that the first line of \eqref{eq:Fk:G:equivalent:system} is satisfied. First, we prove that for all $k\geq 1$, $F_{k}$ is continuous at every point $(\tilde{f}^{1},\tilde{f}^{2},\tilde{f}^{3})$ in $\mathcal{C}(\mathbb{R}_{+},\mathcal{W}^{-2,\alpha}\times\mathbb{R}\times\mathcal{W}^{-2,\alpha}_{0})$. To state continuity of $F_{k}$ at a continuous trajectory, it suffices to show continuity with respect to each coordinate $f^{1}$, $f^{2}$ and $f^{3}$.

- Equation \eqref{eq:control:Lz:3:alpha:to:2:alpha} implies that $z\mapsto || L_{z} \varphi_{k} ||_{2,\alpha}$ is locally bounded and  Lemma \ref{lem:continuity:mappings:eta} gives the following: 
$$f^{1}\mapsto \Big(t\mapsto \left< f^{1}(t),\varphi_{k} \right> - \left< f^{1}(0),\varphi_{k} \right> - \int_{0}^{t} \left< f^{1}(z),L_{z}\varphi_{k} \right> dz \Big)$$
is a mapping from $\mathcal{D}(\mathbb{R}_{+},\mathcal{W}^{-2,\alpha}_{0})$ into $\mathcal{D}(\mathbb{R}_{+},\mathbb{R})$ which is continuous at every point $\tilde{f}^{1}$ in $\mathcal{C}(\mathbb{R}_{+},\mathcal{W}^{-2,\alpha}_{0})$.

- Then, notice that $|\big< u_{z},R\varphi_{k}\,\frac{\partial \Psi}{\partial y}(\cdot,\overline{\gamma}(z)) \big>| \leq {\rm Lip}(\Psi) \mathbb{E}[ |D_{0,\overline{S}^{1}_{z}}(\varphi_{k})| ]$ and Lemma \ref{lem:control:normes:D:H} gives that
\begin{equation*}
z\mapsto \left< u_{z},R\varphi_{k}\,\frac{\partial \Psi}{\partial y}(\cdot,\overline{\gamma}(z)) \right> \text{ is locally bounded,}
\end{equation*}
so that applying Lemma \ref{lem:convergence:Gamma} gives the continuity of $F_{k}$ with respect to $f^{2}$.

- Finally, $F_{k}$ is clearly continuous with respect to $f^{3}$.\\

Notice that \eqref{eq:closed:equation:n:1} gives for any $k\geq 1$,
\begin{multline}\label{eq:closed:equation:Fk}
\forall t\geq 0, \  F_{k}(\eta^{n},\Gamma^{n},W^{n})(t) - \int_{0}^{t} R^{n,(1)}_{z}(\varphi_{k}) dz \\
- \int_{0}^{t} \left< \overline{\mu}^{n}_{S_{z}}-u_{z}, \frac{\partial \Psi}{\partial y}(\cdot,\overline{\gamma}(z))R\varphi_{k}\right> \Gamma^{n}_{z} dz=0.
\end{multline}
Yet, we have, on the one hand,  for all $\theta\geq 0$,
\begin{multline*}
\sup_{z\in [0,\theta]} \mathbb{E}\left[ \big| R^{n,(1)}_{z}(\varphi_{k}) \big| \right]  \leq  \sqrt{n} {\rm Lip}(\Psi) \sup_{z\in [0,\theta]} \mathbb{E}\left[ \frac{1}{n} \sum_{i=1}^{n} |D_{0,S^{n,i}_{z}}(\varphi_{k})| |r^{n,i}_{z}| \right]\\
\leq  n^{-1/2} C {\rm Lip}(\Psi) (1+(M_{T_{0}}+\theta)^{\alpha}) ||\varphi_{k} ||_{2,\alpha} \sup_{z\in [0,\theta]} \mathbb{E}\left[ |\Gamma^{n}_{z-}|^{2} \right] \to 0,
\end{multline*}
where we used Lemma \ref{lem:control:normes:D:H}, Equations \eqref{eq:bound:rni:Gamma} and \eqref{eq:power:Gamma:n:locally:bounded}, and on the other hand,
$$\sup_{z\in [0,\theta]} \mathbb{E}\left[ \big| \left< \overline{\mu}^{n}_{S_{z}}-u_{z}, \frac{\partial \Psi}{\partial y}(\cdot,\overline{\gamma}(z))R\varphi_{k}\right> \Gamma^{n}_{z} \big| \right] \to 0,$$
which follows from Cauchy-Schwarz inequality as we have done for \eqref{eq:convergence:ou:on:utilise:CS}.

These two convergences above imply
\begin{equation}\label{eq:convergence:to:0:rest:terms:equation:eta}
\begin{cases}
\displaystyle \mathbb{E}\left[\left| \sup_{t\in [0,\theta]} \int_{0}^{t} R^{n,(1)}_{z}(\varphi_{k}) dz \right|\right] \to 0,\\
\displaystyle \mathbb{E}\left[\left| \sup_{t\in [0,\theta]} \int_{0}^{t} \left< \overline{\mu}^{n}_{S_{z}}-u_{z}, \frac{\partial \Psi}{\partial y}(\cdot,\overline{\gamma}(z))R\varphi_{k}\right> \Gamma^{n}_{z} dz \right|\right] \to 0.
\end{cases}
\end{equation}

On the one hand, gathering \eqref{eq:closed:equation:Fk} and \eqref{eq:convergence:to:0:rest:terms:equation:eta} gives the convergence of $F_{k}(\eta^{n},\Gamma^{n},W^{n})$ to $0$ in probability and, on the other hand, applying the continuous mapping theorem gives the convergence in law of $F_{k}(\eta^{n},\Gamma^{n},W^{n})$ to $F_{k}(\eta,\Gamma,W)$. Identifying the limits gives $F_{k}(\eta,\Gamma,W)=0$ which ends this step.

\paragraph{Step two.} 
Let us show that the second line of \eqref{eq:Fk:G:equivalent:system} is satisfied. First, we prove that $G$ is continuous at every point $(\tilde{g}^{1},\tilde{g}^{2},\tilde{g}^{3})$ in $\mathcal{C}(\mathbb{R}_{+},\mathcal{W}^{-2,\alpha}\times\mathbb{R}\times\mathbb{R})$. To state continuity of $G$ at a continuous trajectory, it suffices to show continuity with respect to each coordinate $g^{1}$, $g^{2}$ and $g^{3}$.

- Equations \eqref{eq:f:k:alpha:leq:f:Ckb}, \eqref{eq:Psi:dPsi:C2b:locally:in:time} and Lemma \ref{lem:continuity:mappings:eta} give the following: 
$$g^{1}\mapsto \Big( t\mapsto \int_{0}^{t} h(t-z) \left< g^{1}_{z}, \Psi(\cdot,\overline{\gamma}(z)\right> dz \Big)$$
is a mapping from $\mathcal{D}(\mathbb{R}_{+},\mathcal{W}^{-2,\alpha}_{0})$ into $\mathcal{C}(\mathbb{R}_{+},\mathbb{R})$ which is continuous at every point $\tilde{g}^{1}$ in $\mathcal{C}(\mathbb{R}_{+},\mathcal{W}^{-2,\alpha}_{0})$.

- Then, notice that $|\big< u_{z},\frac{\partial \Psi}{\partial y}(\cdot,\overline{\gamma}(z)) \big>|\leq {\rm Lip}(\Psi)$ so that
\begin{equation*}
z\mapsto \left< u_{z},\frac{\partial \Psi}{\partial y}(\cdot,\overline{\gamma}(z))\right> \text{ is locally bounded,}
\end{equation*}
so that applying Lemma \ref{lem:convergence:Gamma} gives the continuity of $G$ with respect to $g^{2}$.

- Finally, $G$ is clearly continuous with respect to $g^{3}$.\\

Notice that \eqref{eq:closed:equation:n:2} gives
\begin{multline*}
\forall  t\geq 0, \  G(\eta^{n},\Gamma^{n},V^{n})(t) - \int_{0}^{t} h(t-z)R^{n,(2)}_{z} dz \\
- \int_{0}^{t} h(t-z) \left< \overline{\mu}^{n}_{S_{z}}-u_{z}, \frac{\partial \Psi}{\partial y}(\cdot,\overline{\gamma}(z)) \right> \Gamma^{n}_{z} dz=0.
\end{multline*}
Finally, the argument used to end the previous step also applies here.\\

To conclude, the two steps above give \eqref{eq:Fk:G:equivalent:system} which gives that the process $(\eta,\Gamma)$ is a solution of \eqref{eq:closed:equation:limit:1}-\eqref{eq:closed:equation:limit:2}. Finally, its trajectories are supported in $\mathcal{C}(\mathbb{R}_{+},\mathcal{W}^{-2,\alpha}_{0}\times \mathbb{R})$ since $\eta$ is supported in $\mathcal{C}(\mathbb{R}_{+},\mathcal{W}^{-2,\alpha}_{0})$ and $\Gamma$ is supported in $\mathcal{C}(\mathbb{R}_{+},\mathbb{R})$ as a solution of \eqref{eq:closed:equation:limit:2} (remind that $h$ is Hölder continuous).

\subsection{Proof of Proposition \ref{prop:almost:solution}}
\label{sec:proof:prop:almost:solution}

First, the following Lemma states the well-posedness of $\hat{\gamma}^n$ as defined by \eqref{eq:defintion:X:hat:n}.

\begin{lem}\label{lem:well:posed:X:hat:n}
Under (\hyperref[ass:for:CLT]{$\mathcal{A}_{\text{\tiny{CLT}}}$}), assume furthermore that $\Psi$ is in $\mathcal{C}^{4}_{b}$.
For all $T\geq 0$, there exists a pathwise unique of solution $(\hat{\gamma}^n_t)_{t\in [0,T]}$ of Equation \eqref{eq:defintion:X:hat:n}. Furthermore, the solution $\hat{\gamma}^n$ has continuous paths.
\end{lem}
\begin{proof}
We first deal with the continuity of $U^n_t:= \int_0^t h(t-z) d\hat{W}^n_z(\mathbf{1})$ appearing in \eqref{eq:defintion:X:hat:n}. Following the arguments given in the proof of Corollary \ref{cor:convergence:couple} to get the control \eqref{eq:control:kolmogorov:V:n}, one can prove that, for all $r<t<T$ and $p\geq 0$, there exists a universal constant $C_p$ such that
\begin{equation*}
\mathbb{E}\left[ |U^{n}_{t}-U^{n}_{r}|^{2p} \right] \leq C_{p} {\rm H\ddot{o}l}(h)^{2p} \left|t-r\right|^{2\beta(h) p} \int_0^t \frac{\left< u_z, \Psi(\cdot,\overline{\gamma}_z)\right>}{n} dz.
\end{equation*}
The integral above being clearly bounded by $T||\Psi||_{\infty}$, one can for instance take $p=1/\beta(h)$ in the equation above in order to apply Kolmogorov continuity theorem.

Hence, without loss of generality, one can deduce that any solution of \eqref{eq:defintion:X:hat:n} admits a modification with continuous paths. Let $t_0<T$ and $g$ be in $\mathcal{C}([0,T],\mathbb{R})$ and consider the application $F_{t_0,g}:\mathcal{C}([0,t_0],\mathbb{R})\to \mathcal{C}([0,t_0],\mathbb{R})$ defined by
\begin{equation*}
F_{t_0,g}(\gamma)(t):= \int_0^t h(t-z) \left< \hat{u}^n_z, \Psi(\cdot,\gamma(z))\right> dz  + g(t).
\end{equation*}
Remind that $\hat{u}^n_t=u_t+n^{-1/2} \eta_t$ and remark that
\begin{eqnarray*}
| \left< \eta_t, \Psi(\cdot,\gamma_1(t))-\Psi(\cdot,\gamma_2(t))\right> | & \leq & C ||\eta_t||_{-2,1} ||\Psi(\cdot,\gamma_1(t))-\Psi(\cdot,\gamma_2(t))||_{\mathcal{C}^2_b}\\
&\leq & C ||\eta_t||_{-2,1} |\gamma_1(t)-\gamma_2(t)|,
\end{eqnarray*}
since $\Psi$ is $\mathcal{C}^4_b$, where the $C$'s are deterministic constants. Hence, we have
\begin{multline*}
||F_{t_0,g}(\gamma_1)-F_{t_0,g}(\gamma_2)||_{\mathcal{C}([0,t_0])}\leq t_0h_{\infty}(t_0) ||\gamma_1-\gamma_2 ||_{\mathcal{C}([0,t_0])} \big( {\rm Lip}(\Psi) \\
 +  C \sup_{t\in [0,T]}  ||\eta_t||_{-2,1} \big).
\end{multline*}
Yet, $\sup_{t\in [0,T]}  ||\eta_t||_{-2,1} $ is almost surely finite as a consequence of \eqref{eq:C:tight} so one can find $t_0$ small enough\footnote{The time $t_0$ is random and depends on the time horizon $T$.} such that $F_{t_0,g}$ is a contraction. The proof is then completed by iteration. Let us mention how goes the second step: fix $t_0$ such that $F_{t_0,g}$ is a contraction and consider the application $G_{t_0,g}:\mathcal{C}([t_0,2t_0],\mathbb{R})\to \mathcal{C}([t_0,2t_0],\mathbb{R})$ defined by
\begin{equation*}
G_{t_0,g}(\gamma)(t):= \int_0^t h(t-z) \left< \hat{u}^n_z, \Psi(\cdot,\gamma(z))\right> dz  + g(t),
\end{equation*}
where the value of $\gamma(t)$ for $t$ in $[0,t_0]$ is given by the unique Banach-Picard fixed point of $F_{t_0,g}$. The same kind of computations as before gives
\begin{multline*}
||G_{t_0,g}(\gamma_1)-G_{t_0,g}(\gamma_2)||_{\mathcal{C}([t_0,2t_0])}\leq t_0 h_{\infty}(t_0) ||\gamma_1-\gamma_2 ||_{\mathcal{C}([t_0,2t_0])} \big( {\rm Lip}(\Psi)  \\
+  C \sup_{t\in [0,T]}  ||\eta_t||_{-2,1} \big),
\end{multline*}
so that $G_{t_0,g}$ is also a contraction.
\end{proof}

Then, the following Lemma relates $\hat{\gamma}^n_t$ with its approximation appearing through the CLT, namely $\check{\gamma}^n_t:=\overline{\gamma}(t)+n^{-1/2}\Gamma_t$.

\begin{lem}\label{lem:control:X:hat:n:CLT}
We have
\begin{equation}\label{eq:control:X:hat:n:CLT}
\mathbb{E}\left[\left| \hat{\gamma}^n_t - \check{\gamma}^n_t \right| \right] \lesssim_t n^{-1}.
\end{equation}
Furthermore, we show within the proof that
\begin{equation}\label{eq:control:X:hat:n:X:bar:square}
\mathbb{E}[ | \hat{\gamma}^n_t - \overline{\gamma}(t) |^2 ]\lesssim_t n^{-1}.
\end{equation}
\end{lem}
\begin{proof}
Let us prove the a priori rough bound,
\begin{equation}\label{eq:rough:bound:gamma:hat:gamma:check}
\mathbb{E}\left[\left| \hat{\gamma}^n_t - \check{\gamma}^n_t \right|^2 \right]^{1/2} \lesssim_t n^{-1/2}.
\end{equation}
We use the decomposition $\hat{\gamma}^n_t - \check{\gamma}^n_t=( \hat{\gamma}^n_t - \overline{\gamma}(t)) - n^{-1/2}\Gamma_t$. On the one hand, $\mathbb{E}\left[|n^{-1/2}\Gamma_t |^2 \right]^{1/2}\lesssim_t n^{-1/2}$ thanks to \eqref{eq:power:Gamma:locally:bounded:limit:version}. On the other hand, we use the decomposition 
\begin{equation*}
\mathbb{E}\left[ \left| \hat{\gamma}^n_t - \overline{\gamma}(t) \right|^2 \right] \leq 3(  A^{n}_1(t) + A^{n}_2(t)+ A^{n}_3(t)),
\end{equation*}
where
\begin{equation*}
\begin{cases}
A^{n}_1(t):= \mathbb{E}\left[ \left| \int_{0}^{t} h(t-z) \left< \hat{u}^n_z -u_z, \Psi(\cdot,\hat{\gamma}^n_z) \right> dz \right|^{2} \right],\\
A^{n}_2(t):= \mathbb{E}\left[ \left| \int_{0}^{t} h(t-z) \left< u_z, \Psi(\cdot,\hat{\gamma}^n_z) - \Psi(\cdot,\overline{\gamma}(z)) \right> dz \right|^{2} \right],\\
A^{n}_3(t):= \mathbb{E}\left[ \left| \int_{0}^{t} h(t-z) d\hat{W}^n_z(\mathbf{1}) \right|^{2} \right].\\
\end{cases}
\end{equation*}

- \underline{Study of $A^{n}_1(t)$.}
Using that $\hat{u}^n_z -u_z=n^{-1/2}\eta_z$ and $\Psi$ belongs to $\mathcal{C}^2_b$, we have
\begin{equation*}
A^n_1(t)  \leq  t^2 h_{\infty}(t)^2 n^{-1} \mathbb{E}\left[ \sup_{z\in [0,t]} ||\eta_z||^2_{-2,1}\right] ||\Psi||^2_{\mathcal{C}^2_b}\lesssim_t n^{-1},
\end{equation*}
where we used \eqref{eq:C:tight}.

- \underline{Study of $A^{n}_2(t)$.}
Using the Lipschitz continuity of $\Psi$, we have
\begin{equation*}
A^n_2(t)  \leq  t h_{\infty}(t)^2 {\rm Lip}(\Psi) \int_0^t \mathbb{E}\left[ \left| \hat{\gamma}^n_z - \overline{\gamma}(z) \right|^2 \right] dz  \lesssim_t \int_0^t \mathbb{E}\left[ \left| \hat{\gamma}^n_z - \overline{\gamma}(z) \right|^2 \right] dz,
\end{equation*}
which is convenient to apply the Grönwall-type Lemma \ref{lem:Gronwall:Generalization}.

- \underline{Study of $A^{n}_3(t)$.}
By definition of the bracket of $\hat{W}^n$, we have
\begin{equation*}
A^n_3(t)  =  \int_0^t h(t-z)^2 n^{-1} \left< u_z, \Psi(\cdot,\overline{\gamma}(z)) \right> dz  \lesssim_t n^{-1}.
\end{equation*}

As expected, applying Lemma \ref{lem:Gronwall:Generalization} gives $\mathbb{E}[ | \hat{\gamma}^n_t - \overline{\gamma}(t) |^2 ]\lesssim_t n^{-1}$, that is Equation \eqref{eq:control:X:hat:n:X:bar:square}. Then, gathering the two steps above proves \eqref{eq:rough:bound:gamma:hat:gamma:check}.\\

Then, let us show how to use \eqref{eq:rough:bound:gamma:hat:gamma:check} in order to prove \eqref{eq:control:X:hat:n:CLT}. We have,
\begin{multline*}
\check{\gamma}^n_t = \int_0^t h(t-z) \left[ \left< \hat{u}^n_z, \Psi(\cdot,\overline{\gamma}(z))\right> + \left< u_{z},\frac{\partial \Psi}{\partial y}(\cdot,\overline{\gamma}(z)) \right> n^{-1/2} \Gamma_{z}\right] dz\\
+ \int_0^t h(t-z) d\hat{W}^n_z(\mathbf{1}).
\end{multline*}
Hence we use the decomposition 
\begin{equation*}
\mathbb{E}\left[ \left| \hat{\gamma}^n_t - \check{\gamma}^n_t \right| \right] \leq 3(  B^{n}_1(t) + B^{n}_2(t)+ B^{n}_3(t)),
\end{equation*}
where
\begin{equation*}
\begin{cases}
B^{n}_1(t):= \mathbb{E}\left[ \left| \int_{0}^{t} h(t-z) \left< \hat{u}^n_z, \Psi(\cdot,\hat{\gamma}^n_z) - \Psi(\cdot,\check{\gamma}^n_z) \right> dz \right| \right],\\
B^{n}_2(t):= \mathbb{E}\left[ \left| \int_{0}^{t} h(t-z) \left< \hat{u}^n_z, \Psi(\cdot,\check{\gamma}^n_z ) - \Psi(\cdot,\overline{\gamma}(z)) - \frac{\partial \Psi}{\partial y}(\cdot,\overline{\gamma}(z)) \frac{\Gamma_{z}}{\sqrt{n}} \right> dz \right| \right],\\
B^{n}_3(t):= \mathbb{E}\left[ \left| \int_0^t h(t-z) \left< \hat{u}^n_z - u_z, \frac{\partial \Psi}{\partial y}(\cdot,\overline{\gamma}(z))  \right> \frac{\Gamma_{z}}{\sqrt{n}} dz \right| \right].
\end{cases}
\end{equation*}

- \underline{Study of $B^{n}_1(t)$.}
Using Lemma \ref{lem:estimates:u:hat}, Cauchy-Schwarz inequality and finally \eqref{eq:C:tight} and the a priori rough bound \eqref{eq:rough:bound:gamma:hat:gamma:check}, we have
\begin{eqnarray*}
B^n_1(t)  &\leq & h_{\infty}(t)  \int_0^t \mathbb{E}\left[ \left(1 + \frac{C}{\sqrt{n}} ||\eta_t||_{-2,1}\right)  ||\Psi||_{\mathcal{C}^3_b} |\hat{\gamma}^n_z-\check{\gamma}^n_z| \right] dz ,\\
&\lesssim_t & \int_0^t \mathbb{E}\left[ |\hat{\gamma}^n_z-\check{\gamma}^n_z| \right] dz + n^{-1/2} \int_0^t \mathbb{E}\left[ ||\eta_t||_{-2,1}^2 \right]^{1/2} \mathbb{E}\left[ |\hat{\gamma}^n_z-\check{\gamma}^n_z|^2 \right]^{1/2} dz\\
&\lesssim_t &  \int_0^t \mathbb{E}\left[ |\hat{\gamma}^n_z-\check{\gamma}^n_z| \right] dz + n^{-1}
\end{eqnarray*}
where we used \eqref{eq:C:tight}.

- \underline{Study of $B^{n}_2(t)$.}
From Taylor's inequality it follows that
\begin{equation*}
\left\Vert \Psi(\cdot,\check{\gamma}^n_z ) - \Psi(\cdot,\overline{\gamma}(z)) - \frac{\partial \Psi}{\partial y}(\cdot,\overline{\gamma}(z)) \frac{\Gamma_{z}}{\sqrt{n}} \right\Vert_{\mathcal{C}^2_b} \leq ||\Psi||_{\mathcal{C}^4_b} \frac{|\Gamma_z|^2}{n},
\end{equation*}
and so, using Lemma \ref{lem:estimates:u:hat}, Cauchy-Schwarz inequality and finally \eqref{eq:C:tight} and \eqref{eq:power:Gamma:locally:bounded:limit:version}, we have
\begin{eqnarray*}
B^n_2(t) & \leq & h_{\infty}(t) \int_0^t \mathbb{E}\left[ \left(1 + \frac{C}{\sqrt{n}} ||\eta_t||_{-2,1}\right)  ||\Psi||_{\mathcal{C}^4_b} \frac{|\Gamma_z|^2}{n} \right] dz, \\
 &\lesssim_t & n^{-1} \int_0^t \mathbb{E}\left[ |\Gamma_z|^2 \right] dz + n^{-3/2} \int_0^t \mathbb{E}\left[ ||\eta_z||_{-2,1}^2 \right]^{1/2}  \mathbb{E}\left[ |\Gamma_z|^4 \right]^{1/2} dz,\\
&\lesssim_t & n^{-1}.
\end{eqnarray*}

- \underline{Study of $B^{n}_3(t)$.}
By definition of $\hat{u}^n=u-n^{-1/2}\eta$ and doing as above, we have
\begin{equation*}
B^n_3(t) \leq  h_{\infty}(t) n^{-1} \int_0^t \mathbb{E}\left[ ||\eta_z||_{-2,1}  \left\Vert \frac{\partial \Psi}{\partial y}(\cdot,\overline{\gamma}(z))\right\Vert_{\mathcal{C}^2_b} |\Gamma_z| \right] dz \lesssim_t  n^{-1}.
\end{equation*}

Finally, applying Lemma \ref{lem:Gronwall:Generalization} gives \eqref{eq:control:X:hat:n:CLT} and ends the proof.
\end{proof}

We are now in position to prove Proposition \ref{prop:almost:solution}. Let us use the decomposition
$$
\mathbb{E}\left[|r^n_t(\varphi)|\right] \leq \int_0^t A^n_1(z)+A^n_2(z)+A^n_3(z) dz,
$$
where
\begin{equation*}
\begin{cases}
A^{n}_1(t):= \mathbb{E}\left[ \left| \left< \hat{u}^{n}_{t}-u_t, (L_t-\hat{L}_t^n) \varphi\right> \right| \right],\\
A^{n}_2(t):= \mathbb{E}\left[ \left| \left< u_{t}, \left[ \Psi(\cdot,\overline{\gamma}(t)) + \frac{\partial \Psi}{\partial y}(\cdot,\overline{\gamma}(t)) n^{-1/2} \Gamma_{t} - \Psi(\cdot,\check{\gamma}^n_t)\right]  R\varphi \right> \right| \right],\\
A^{n}_3(t):= \mathbb{E}\left[ \left| \left< u_{t}, \left[ \Psi(\cdot,\check{\gamma}^n_t) - \Psi(\cdot,\hat{\gamma}^n_t)\right]  R\varphi \right> \right| \right].\\
\end{cases}
\end{equation*}

- \underline{Study of $A^{n}_1(t)$.}
Using that $\hat{u}^n_t -u_t=n^{-1/2}\eta_t$, $\Psi$ belongs to $\mathcal{C}^3_b$, and an inequality similar to the second line of \eqref{eq:lipschitz:control:vs:eta}, we have
\begin{equation*}
A^n_1(t)  \lesssim_t  n^{-1/2} \mathbb{E}\left[ ||\eta_t||_{-2,1} ||\Psi||_{\mathcal{C}^3_b} |\overline{\gamma}(t)-\hat{\gamma}^n_t| \right] ||\varphi||_{2,1}.
\end{equation*}
Cauchy-Schwarz inequality with Equations \eqref{eq:C:tight} and \eqref{eq:control:X:hat:n:X:bar:square} gives $A^n_1(t)  \lesssim_t n^{-1}$.

- \underline{Study of $A^{n}_2(t)$.}
From Taylor's inequality and then \eqref{eq:embedding:Ck:sobolev} and \eqref{eq:power:Gamma:locally:bounded:limit:version}, it follows that
\begin{equation*}
A^n_2(t)  \lesssim_t  n^{-1} ||\Psi||_{\mathcal{C}^2_b}  \mathbb{E}\left[ |\Gamma_t|^2 \right] ||\varphi||_{\infty} \lesssim_t n^{-1} ||\varphi||_{2,1}.
\end{equation*}

- \underline{Study of $A^{n}_3(t)$.}
Finally, using \eqref{eq:control:X:hat:n:CLT} and then \eqref{eq:embedding:Ck:sobolev}, we have
\begin{equation*}
A^n_3(t)  \lesssim_t  {\rm Lip}(\Psi) \mathbb{E}\left[\left| \hat{\gamma}^n_t - \check{\gamma}^n_t \right| \right]  ||\varphi||_{\infty} \lesssim_t n^{-1} ||\varphi||_{2,1}.
\end{equation*}
Gathering the computations above, we prove the first part of Proposition \ref{prop:almost:solution}. For the second part, let us use the decomposition,
$$
\mathbb{E}\left[ \left| {\rm DM}_t(\varphi_1,\varphi_2) - \hat{\rm DM}_t(\varphi_1,\varphi_2) \right| \right] \leq \frac{1}{n} \int_0^{t} B^n_1(z)+B^n_2(z) dz,
$$
where
\begin{equation*}
\begin{cases}
B^{n}_1(t):= \mathbb{E}\left[ \left| \left< u_{t}, \left[ \Psi(\cdot,\hat{\gamma}^n_t) - \Psi(\cdot,\overline{\gamma}(t)) \right] \varphi_1\varphi_2 \right> \right| \right],\\
B^{n}_2(t):= n^{-1/2} \mathbb{E}\left[ \left| \left< \eta_{t}, \Psi(\cdot,\hat{\gamma}^n_t)\varphi_1\varphi_2 \right> \right| \right].
\end{cases}
\end{equation*}

- \underline{Study of $B^{n}_1(t)$.}
Using the Lipschitz continuity of $\Psi$ and then \eqref{eq:control:X:hat:n:X:bar:square}, we have
\begin{equation*}
B^n_1(t)  \leq {\rm Lip}(\Psi) ||\varphi_1||_{\infty} ||\varphi_2||_{\infty} \mathbb{E}\left[\left| \hat{\gamma}^n_t - \overline{\gamma}(t) \right| \right]  \lesssim_t n^{-1/2} ||\varphi_1||_{\infty} ||\varphi_2||_{\infty}.
\end{equation*}

- \underline{Study of $B^{n}_2(t)$.}
Using Lemma \ref{lem:control:for:linear:operator}, we have,
\begin{equation*}
\left| \left< \eta_{t}, \Psi(\cdot,\hat{\gamma}^n_t)\varphi_1\varphi_2 \right> \right| \leq C ||\eta_t||_{-2,1} ||\Psi(\cdot,\hat{\gamma}^n_t)||_{\mathcal{C}^2_b(\mathbb{R}_+)} ||\varphi_1\varphi_2||_{2,1}.
\end{equation*}
Yet, $||\Psi(\cdot,\hat{\gamma}^n_t)||_{\mathcal{C}^2_b(\mathbb{R}_+)}\leq |\Psi||_{\mathcal{C}^2_b(\mathbb{R}_+\times\mathbb{R})}$ and, by combining Lemma \ref{lem:control:for:linear:operator} and \eqref{eq:embedding:Ck:sobolev}, $||\varphi_1\varphi_2||_{2,1}\leq C ||\varphi_1||_{\mathcal{C}^2_b} ||\varphi_2||_{2,1}\leq C ||\varphi_1||_{3,1} ||\varphi_2||_{3,1}$. Hence, Equation \eqref{eq:C:tight} gives $B^{n}_2(t) \lesssim_t n^{-1/2} ||\varphi_1||_{3,1} ||\varphi_2||_{3,1}$.
Gathering the computations above ends the proof.

\section{Lemmas}
The following lemma is a generalization of the standard Grönwall lemma.
\begin{lem}\label{lem:Gronwall:Generalization}
Let $f,g:\mathbb{R}_{+}\rightarrow\mathbb{R}_{+}$ be two locally bounded non-negative measurable functions. Assume that for all $t\geq 0$,
\begin{equation}\label{eq:assumption:Gronwall}
f(t)\lesssim_{t} g(t) + \int_{0}^{t} f(s)ds.
\end{equation}
Then, for any $\theta\geq 0$, $\sup_{t\in \left[0,\theta\right]} f(t) \lesssim_{\theta} \sup_{t\in \left[0,\theta\right]}g(t)$.
\end{lem}
\begin{proof}
For a fixed $\theta$, Equation \eqref{eq:assumption:Gronwall} implies that there exists a constant $C$ such that for all $t\leq \theta$, $f(t)\leq C(\sup_{t\in[0,\theta]} g(t) +\int_{0}^{t} f(s)ds)$. Hence, standard Grönwall's inequality gives $\sup_{t\in[0,\theta]} f(t) \leq C\sup_{t\in[0,\theta]} g(t) e^{C\theta}$ which ends the proof.
\end{proof}

The next lemma proves continuity in time for the law of the age process associated with a point process.
\begin{lem}\label{lem:continuity:Pt}
Assume that $N$ admits the bounded $\mathbb{F}$-intensity $\lambda_{t}$ and satisfy Assumption (\hyperref[ass:initial:condition:density:compact:support]{$\mathcal{A}^{u_0}_{\infty}$}). Denote by $(S_{t})_{t\geq 0}$ its associated age process. Then, the law of $S_{t}$ denoted by $w_{t}$ is such that $t\mapsto w_{t}$ belongs to $\mathcal{C}(\mathbb{R}_{+},\mathcal{W}^{-2,\alpha}_{0})$ for any $\alpha>1/2$.
\end{lem}
\begin{proof}
This continuity result comes from the fact that the probability that $N$ has a point in an interval goes to $0$ as the size of the interval goes to $0$.
Fix $\alpha>1/2$ and let $t,t'$ be positive real numbers. First, remark that $S_{t+t'}=S_{t}+t'$ as soon as there is no point of $N$ in the interval $[t,t+t']$ and so one has for all $\varphi$ in $\mathcal{W}^{2,\alpha}_{0}$,
\begin{equation*}
|\varphi(S_{t+t'})-\varphi(S_{t})| \leq  ||D_{S_{t}+t',S_{t}}||_{-2,\alpha} ||\varphi||_{2,\alpha} + (|\varphi(S_{t+t'})|+|\varphi(S_{t})|) \mathds{1}_{N([t,t+t'])\neq 0}.
\end{equation*}
The bound obtained in Lemma \ref{lem:control:normes:D:H} for the operator $D_{x,y}$ is too rough here. We need a finer bound: it holds that there exists a constant $C$ such that $||D_{x,y}||_{-2,\alpha}\leq C |x-y| (1+\max(|x|^{\alpha},|y|^{\alpha}))$. Indeed, by density, let us assume that $\varphi$ is $\mathcal{C}^{\infty}$ with compact support and remark that
\begin{multline*}
|\varphi(x)-\varphi(y)|\leq |x-y| \sup_{z, \, |z|\leq \max(|x|,|y|)} |\varphi'(z)| \leq |x-y| (1+\max(|x|^{\alpha},|y|^{\alpha})) ||\varphi||_{\mathcal{C}^{1,\alpha}} \\
\leq C |x-y| (1+\max(|x|^{\alpha},|y|^{\alpha})) ||\varphi||_{2,\alpha},
\end{multline*}
where we used \eqref{eq:embedding:Ck:sobolev} in the last inequality. Since (\hyperref[ass:initial:condition:density:compact:support]{$\mathcal{A}^{u_0}_{\infty}$}) is satisfied, $S_{t+t'}$ and $S_{t}$ are upper bounded by $M_{S_{0}}+t+t'$ so that
\begin{equation*}
\begin{cases}
||D_{S_{t}+t',S_{t}}||_{-2,\alpha}\leq C t' (1+(M_{S_{0}}+t+t')^{\alpha})\\
(|\varphi(S_{t+t'})|+|\varphi(S_{t})|)\leq 2 (1+(M_{S_{0}}+t+t')^{\alpha}) ||\varphi||_{\mathcal{C}^{0,\alpha}}.
\end{cases}
\end{equation*}
Hence, \eqref{eq:embedding:Ck:sobolev} gives
\begin{equation*}
|\varphi(S_{t+t'})-\varphi(S_{t})| \leq  C(t'+ \mathds{1}_{N([t,t+t'])\neq 0}) ||\varphi||_{2,\alpha}.
\end{equation*}
Yet, $\mathbb{P}\left( N([t,t+t'])\neq 0 \right)\leq \mathbb{E}\left[ N([t,t+t']) \right] = \mathbb{E}[ \int_{t}^{t+t'} \lambda_{z} dz ]$ goes to $0$ as $t'$ goes to $0$. 
The same argument for $t'<0$ gives continuity.
\end{proof}

The three lemmas below are used to get the limit equation satisfied by the fluctuations.

\begin{lem}\label{lem:continuity:mappings:eta}
Let $h$ be a locally bounded function and $(\varphi_{t})_{t\geq 0}$ be a family of test functions in $\mathcal{W}^{2,\alpha}_{0}$ such that $t\mapsto ||\varphi_{t}||_{2,\alpha}$ is locally bounded. Then, $F: g\mapsto \int_{0}^{t} h(t-z) \left< g(z),\varphi_{z} \right> dz$ is a mapping from $\mathcal{D}(\mathbb{R}_{+},\mathcal{W}^{-2,\alpha}_{0})$ to $\mathcal{C}(\mathbb{R}_{+},\mathbb{R})$ which is continuous at every point $g_{0}$ in $\mathcal{C}(\mathbb{R}_{+},\mathcal{W}^{-2,\alpha}_{0})$.
\end{lem}
\begin{proof}
Let $(g_{n})_{n\geq 1}$ be any sequence such that $g_{n}\to g_{0}$ for the Skorokhod topology. Since $g_{0}$ is continuous, the convergence also holds true for the local uniform topology \cite[Proposition VI.1.17.]{Jacod_2003}. We have for all $\theta\geq 0$,
\begin{equation}\label{eq:control:Getan:Geta0}
\sup_{t\in [0,\theta]} |F(g_{n})(t)-F(g_{0})(t)|\leq \sup_{z\in [0,\theta]} h(z) \sup_{z\in [0,\theta]} ||g_{n}(z)-g_{0}(z)||_{-2,\alpha} \sup_{z\in [0,\theta]} ||\varphi_{z}||_{2,\alpha}.
\end{equation}
Yet, the right hand side of \eqref{eq:control:Getan:Geta0} goes to $0$ as $n$ goes to infinity, which ends the proof.
\end{proof}

\begin{lem}\label{lem:convergence:Gamma}
Assume that $(g^{n})_{n\geq 1}$ converges to $g$ for the Skorokhod topology in $\mathcal{D}(\mathbb{R}_{+},\mathbb{R})$. If $h$ satisfies (\hyperref[ass:h:Holder]{$\mathcal{A}^{h}_{\rm H\ddot{o}l}$}) and $f$ is locally bounded, then
\begin{equation*}
\int_{0}^{t} h(t-z)f(z)g^{n}(z) dz \xrightarrow[n\to +\infty]{} \int_{0}^{t} h(t-z)f(z)g(z) dz ,
\end{equation*}
as functions of $t$ in $\mathcal{C}(\mathbb{R}_{+},\mathbb{R})$ for the local uniform topology. In particular, the application $F$ from $\mathcal{D}(\mathbb{R}_{+},\mathbb{R})$ to $\mathcal{D}(\mathbb{R}_{+},\mathbb{R})$ defined by
\begin{equation*}
F(g)(t):=\int_{0}^{t} h(t-z)f(z)g(z) dz,
\end{equation*}
is continuous.
\end{lem}
\begin{proof}
Let $c^{n}(t):=\int_{0}^{t} h(t-z)f(z)g^{n}(z) dz$ and $c(t):=\int_{0}^{t} h(t-z)f(z)g(z) dz $. Assume for a while that $h(0)=0$ and extend the function $h$ to the whole real line by  setting $0$ on the negative real numbers. Then, for all $t,\delta\geq 0$,
\begin{eqnarray*}
|c^{n}(t+\delta)-c^{n}(t)| & \leq & \int_{0}^{t+\delta} |h(t+\delta-z)-h(t-z)| |f(z)| |g^{n}(z)| dz \\
&\leq &(t+\delta) {\rm H\ddot{o}l}(h) \sup_{z\in [0,t+\delta]} |f(z)| \sup_{z\in [0,t+\delta]}|g^{n}(z)| \delta^{\beta(h)}.
\end{eqnarray*}
Yet, since $g^{n}$ is convergent, we have $\sup_{n\geq 1} \sup_{z\in [0,t+\delta]}|g^{n}(z)|<+\infty$ (see \cite[Proposition VI.2.4.]{Jacod_2003} for instance) which implies that for all $\theta\geq 0$,
\begin{equation*}
\sup_{n\geq 1} \sup_{t\in [0,\theta]} |c^{n}(t+\delta)-c^{n}(t)| \to 0 \quad \text{ as } \delta\to 0.
\end{equation*}
Hence, the sequence $(c^{n})_{n\geq 1}$ is uniformly continuous. Moreover, for all $n\geq 1$, $c^{n}(0)=0$ and the uniform continuity gives the uniform boundedness
\begin{equation*}
\sup_{n\geq 1} \sup_{t\in [0,\theta]} |c^{n}(t)| <+\infty.
\end{equation*} 
Then, Ascoli-Arzela theorem implies that the sequence $(c^{n})_{n\geq 1}$ is relatively compact. It only remains to identify the limit for all $t\geq 0$. Yet, as a consequence of the dominated convergence and the fact that for almost every $z$, $g^{n}(z)\to g(z)$, we have $\int_{0}^{t} h(t-z)f(z)g^{n}(z) dz \to \int_{0}^{t} h(t-z)f(z)g(z) dz $.

Now, if $h(0)\neq 0$, one can use the following decomposition,
$$
c^{n}(t)=\int_{0}^{t} (h(t-z)-h(0))f(z)g^{n}(z) dz+ h(0) \int_{0}^{t} f(z)g^{n}(z) dz.
$$
The first term is convergent thanks to what we have done in the case $h(0)=0$ whereas the convergence of the second one is simpler and left to the reader.
\end{proof}

\paragraph{Acknowledgement}
This research was partly supported by the french Agence Nationale de la Recherche (ANR 2011 BS01 010 01 projet Calibration), by the interdisciplanary axis MTC-NSC of the University of Nice Sophia-Antipolis and by the Labex MME-DII (ANR11-LBX-0023-01).
The author would like to thank François Delarue for helpful discussions which improved this paper.

\newpage
\pagestyle{plain}

\bibliographystyle{abbrv}		
{\footnotesize \bibliography{references.bib}}

\begin{thebibliography}{10}

\bibitem{adams2003sobo}
R.~A. Adams and J.~J.~F. Fournier.
\newblock {\em {Sobolev spaces}}, volume 140 of {\em {Pure and Applied
  Mathematics (Amsterdam)}}.
\newblock Elsevier/Academic Press, Amsterdam, second edition, 2003.

\bibitem{bacry_2012}
E.~Bacry, K.~Dayri, and J.-F. Muzy.
\newblock {Non-parametric kernel estimation for symmetric {H}awkes processes.
  {A}pplication to high frequency financial data}.
\newblock {\em The European Physical Journal B-Condensed Matter and Complex
  Systems}, 85(5):1--12, 2012.

\bibitem{CLT_Hawkes_lineaire}
E.~Bacry, S.~Delattre, M.~Hoffmann, and J.~F. Muzy.
\newblock {Scaling limits for Hawkes processes and application to financial
  statistics}, Feb. 2012.

\bibitem{bao_2015}
P.~Bao, H.-W. Shen, X.~Jin, and X.-Q. Cheng.
\newblock {Modeling and Predicting Popularity Dynamics of Microblogs using
  Self-Excited {H}awkes Processes}.
\newblock {\em arXiv preprint
  \href{http://arxiv.org/abs/1503.02754}{arXiv:1503.02754}}, 2015.

\bibitem{Billing_Convergence}
P.~Billingsley.
\newblock {\em {Convergence of probability measures}}.
\newblock {Wiley Series in Probability and Statistics: Probability and
  Statistics}. John Wiley \& Sons Inc., New York, second edition, 1999.
\newblock A Wiley-Interscience Publication.

\bibitem{Bremaud_PP}
P.~Br{\'e}maud.
\newblock {\em {Point processes and queues}}.
\newblock Springer-Verlag, New York, 1981.
\newblock Martingale dynamics, Springer Series in Statistics.

\bibitem{brunel1999fast}
N.~Brunel and V.~Hakim.
\newblock {Fast global oscillations in networks of integrate-and-fire neurons
  with low firing rates}.
\newblock {\em Neural computation}, 11(7):1621--1671, 1999.

\bibitem{buice2013dynamic}
M.~A. Buice and C.~C. Chow.
\newblock {Dynamic finite size effects in spiking neural networks}.
\newblock {\em PLoS Comput Biol}, 9(1):e1002872, 2013.

\bibitem{chevallier2015mean}
J.~Chevallier.
\newblock {Mean-field limit of generalized {H}awkes processes}.
\newblock {\em arXiv preprint
  \href{http://arxiv.org/abs/1510.05620}{arXiv:1510.05620}}, 2015.

\bibitem{chevallier2016phd}
J.~Chevallier.
\newblock {\em {Modelling large neural networks via Hawkes processes}}.
\newblock PhD thesis, Universit{\'e} Nice Sophia Antipolis, 2016. Manuscript
  available
  \href{https://perso.u-cergy.fr/~jchevallier/Papers/these.pdf}{here}.

\bibitem{chevallier2015microscopic}
J.~Chevallier, M.~J. C{\'a}ceres, M.~Doumic, and P.~Reynaud-Bouret.
\newblock {Microscopic approach of a time elapsed neural model}.
\newblock {\em Mathematical Models and Methods in Applied Sciences},
  25(14):2669--2719, 2015.

\bibitem{crane2008robust}
R.~Crane and D.~Sornette.
\newblock {Robust dynamic classes revealed by measuring the response function
  of a social system}.
\newblock {\em Proceedings of the National Academy of Sciences},
  105(41):15649--15653, 2008.

\bibitem{dumont2016private}
G.~Dumont et~al.
\newblock {Private communication about ongoing work}.

\bibitem{dumont2016theoretical}
G.~Dumont, J.~Henry, and C.~O. Tarniceriu.
\newblock {Theoretical connections between mathematical neuronal models
  corresponding to different expressions of noise}.
\newblock {\em arXiv preprint arXiv:1602.03764}, 2016.

\bibitem{faugeras2014asymptotic}
O.~Faugeras and J.~Maclaurin.
\newblock {Asymptotic description of stochastic neural networks. I. Existence
  of a large deviation principle}.
\newblock {\em Comptes Rendus Mathematique}, 352(10):841--846, 2014.

\bibitem{Meleard_97}
B.~Fernandez and S.~M{{\'e}}l{{\'e}}ard.
\newblock {A {H}ilbertian approach for fluctuations on the {M}c{K}ean-{V}lasov
  model}.
\newblock {\em Stochastic Process. Appl.}, 71(1):33--53, 1997.

\bibitem{gerstner2002spiking}
W.~Gerstner and W.~M. Kistler.
\newblock {\em {Spiking neuron models: Single neurons, populations,
  plasticity}}.
\newblock Cambridge university press, 2002.

\bibitem{gill_1997}
R.~D. Gill, N.~Keiding, and P.~K. Andersen.
\newblock {\em {Statistical models based on counting processes}}.
\newblock Springer, 1997.

\bibitem{gusto_2005}
G.~Gusto and S.~Schbath.
\newblock {FADO: A Statistical Method to Detect Favored or Avoided Distances
  between Occurrences of Motifs using the Hawkes' Model}.
\newblock {\em Statistical Applications in Genetics and Molecular Biology},
  4(1), 2005.

\bibitem{hawkes_1971}
A.~G. Hawkes.
\newblock {Spectra of some self-exciting and mutually exciting point
  processes}.
\newblock {\em Biometrika}, 58(1):83--90, 1971.

\bibitem{Jacod_2003}
J.~Jacod and A.~N. Shiryaev.
\newblock {\em {Limit theorems for stochastic processes}}, volume 288 of {\em
  {Grundlehren der Mathematischen Wissenschaften [Fundamental Principles of
  Mathematical Sciences]}}.
\newblock Springer-Verlag, Berlin, second edition, 2003.

\bibitem{joffe1986weak}
A.~Joffe and M.~M{\'e}tivier.
\newblock {Weak convergence of sequences of semimartingales with applications
  to multitype branching processes}.
\newblock {\em Advances in Applied Probability}, pages 20--65, 1986.

\bibitem{jourdain1998propagation}
B.~Jourdain and S.~M{\'e}l{\'e}ard.
\newblock {Propagation of chaos and fluctuations for a moderate model with
  smooth initial data}.
\newblock In {\em {Annales de l'IHP Probabilit{\'e}s et statistiques}},
  volume~34, pages 727--766, 1998.

\bibitem{kagan_2010}
Y.~Y. Kagan.
\newblock {Statistical distributions of earthquake numbers: consequence of
  branching process}.
\newblock {\em Geophysical Journal International}, 180(3):1313--1328, 2010.

\bibitem{liniger2009multivariate}
T.~J. Liniger.
\newblock {\em {Multivariate hawkes processes}}.
\newblock PhD thesis, Diss., Eidgen{\"o}ssische Technische Hochschule ETH
  Z{\"u}rich, Nr. 18403, 2009, 2009.

\bibitem{luccon2015transition}
E.~Lu\c{c}on and W.~Stannat.
\newblock {Transition from Gaussian to non-Gaussian fluctuations for mean-field
  diffusions in spatial interaction}.
\newblock {\em arXiv preprint
  \href{http://arxiv.org/abs/1502.00532}{arXiv:1502.00532}}, 2015.

\bibitem{mattia2004finite}
M.~Mattia and P.~{Del Giudice}.
\newblock {Finite-size dynamics of inhibitory and excitatory interacting
  spiking neurons}.
\newblock {\em Physical Review E}, 70(5):052903, 2004.

\bibitem{Meleard_96}
S.~M{{\'e}}l{{\'e}}ard.
\newblock {Asymptotic behaviour of some interacting particle systems;
  {M}c{K}ean-{V}lasov and {B}oltzmann models}.
\newblock In {\em {Probabilistic models for nonlinear partial differential
  equations ({M}ontecatini {T}erme, 1995)}}, volume 1627 of {\em {Lecture Notes
  in Math.}}, pages 42--95. Springer, Berlin, 1996.

\bibitem{merlevede2013rosenthal}
F.~Merlev{\`e}de, M.~Peligrad, et~al.
\newblock {Rosenthal-type inequalities for the maximum of partial sums of
  stationary processes and examples}.
\newblock {\em The Annals of Probability}, 41(2):914--960, 2013.

\bibitem{mohler_2011}
G.~O. Mohler, M.~B. Short, P.~J. Brantingham, F.~P. Schoenberg, and G.~E. Tita.
\newblock {Self-exciting point process modeling of crime}.
\newblock {\em Journal of the American Statistical Association}, 106(493),
  2011.

\bibitem{ogata_1998}
Y.~Ogata.
\newblock {Space-time point-process models for earthquake occurrences}.
\newblock {\em Annals of the Institute of Statistical Mathematics},
  50(2):379--402, 1998.

\bibitem{pakdaman2010dynamics}
K.~Pakdaman, B.~Perthame, and D.~Salort.
\newblock {Dynamics of a structured neuron population}.
\newblock {\em Nonlinearity}, 23(1):55, 2010.

\bibitem{pakdaman2013relaxation}
K.~Pakdaman, B.~Perthame, and D.~Salort.
\newblock {Relaxation and self-sustained oscillations in the time elapsed
  neuron network model}.
\newblock {\em SIAM Journal on Applied Mathematics}, 73(3):1260--1279, 2013.

\bibitem{Rebolledo_80}
R.~Rebolledo.
\newblock {Central limit theorems for local martingales}.
\newblock {\em Z. Wahrsch. Verw. Gebiete}, 51(3):269--286, 1980.

\bibitem{Revuz_1999}
D.~Revuz and M.~Yor.
\newblock {\em {Continuous Martingales and Brownian Motion (Grundlehren der
  mathematischen Wissenschaften)}}.
\newblock Springer-Verlag, 3rd edition, 1999.

\bibitem{reynaud_2010}
P.~Reynaud-Bouret, S.~Schbath, et~al.
\newblock {Adaptive estimation for Hawkes processes; application to genome
  analysis}.
\newblock {\em The Annals of Statistics}, 38(5):2781--2822, 2010.

\bibitem{riedler2012limit}
M.~G. Riedler, M.~Thieullen, and G.~Wainrib.
\newblock {Limit theorems for infinite-dimensional piecewise deterministic
  Markov processes. Applications to stochastic excitable membrane models}.
\newblock {\em Electron. J. probab}, 17(55):1--48, 2012.

\bibitem{shorack2009empirical}
G.~R. Shorack and J.~A. Wellner.
\newblock {\em {Empirical processes with applications to statistics}},
  volume~59.
\newblock Siam, 2009.

\bibitem{Sznitman_91}
A.-S. Sznitman.
\newblock {Topics in propagation of chaos}.
\newblock In {\em {{\'E}cole d'{\'E}t{\'e} de {P}robabilit{\'e}s de
  {S}aint-{F}lour {XIX}---1989}}, volume 1464 of {\em {Lecture Notes in
  Math.}}, pages 165--251. Springer, Berlin, 1991.

\bibitem{MTGAUE}
C.~Tuleau-Malot, A.~Rouis, F.~Grammont, and P.~Reynaud-Bouret.
\newblock {Multiple Tests Based on a Gaussian Approximation of the Unitary
  Events Method with delayed coincidence count}.
\newblock {\em \emph{appearing in} Neural Computation}, 26:7, 2014.

\bibitem{wainrib2010randomness}
G.~Wainrib.
\newblock {\em {Randomness in neurons: a multiscale probabilistic analysis}}.
\newblock PhD thesis, PhD thesis, 2010.

\bibitem{yosida1980func}
K.~Yosida.
\newblock {\em {Functional analysis}}, volume 123 of {\em {Grundlehren der
  Mathematischen Wissenschaften [Fundamental Principles of Mathematical
  Sciences]}}.
\newblock Springer-Verlag, Berlin-New York, sixth edition, 1980.

\bibitem{zhu7531nonlinear}
L.~Zhu.
\newblock {\em {Nonlinear Hawkes Processes}}.
\newblock PhD thesis, New York University, 2013.

\end{thebibliography}

\end{document}